\documentclass{amsproc}

\usepackage[mathcal]{euscript}
\usepackage{amsthm, color}
\usepackage{amsfonts}
\usepackage{amsmath,amssymb}
\usepackage{indentfirst,latexsym,bm,amsthm,graphicx,colortbl}
\usepackage{fancyhdr}
\usepackage{times}
\usepackage{amsmath,amsfonts,amssymb,amsthm}
\usepackage{epsfig}
\usepackage[all]{xy}
\newtheorem{theorem}{Theorem}[section]
\newtheorem{lemm}[theorem]{Lemma}
\newtheorem{prop}[theorem]{Proposition}
\newtheorem{theo}[theorem]{Theorem}
\theoremstyle{definition}
\newtheorem{defi}[theorem]{Definition}

\newtheorem{coro}[theorem]{Corollary}
\theoremstyle{remark}
\newtheorem{remark}[theorem]{Remark}

\numberwithin{equation}{section}

\def\om{\omega}
\def\al{\alpha}

\newcommand{\la}{\lambda}

\newcommand{\vep}{\varepsilon}

\def\om{\omega}
\def\ot{\otimes}
\def\ra{\rangle}
\def\la{\langle}

\begin{document}

\title[Two-parameter twisted quantum affine algebras]
{Two-parameter twisted quantum affine algebras}

\author[Jing]{Naihuan Jing}
\address{
Department of Mathematics,
   North Carolina State University,
   Ra\-leigh, NC 27695, USA}
\email{jing@math.ncsu.edu}

\author[Zhang]{Honglian Zhang$^\star$}
\address{Department of Mathematics, Shanghai University,
Shanghai 200444, China} \email{hlzhangmath@shu.edu.cn}

\thanks{$^\star$ H.Zhang, Corresponding Author}


\begin{abstract}
We establish Drinfeld realization for the two-parameter twisted
quantum affine algebras using a new method. The Hopf algebra structure for
Drinfeld generators is given
for both untwisted and twisted two-parameter quantum affine algebras, which
include the quantum affine algebras as special cases.
\end{abstract}

\maketitle

\section{ Introduction}

Drinfeld realization \cite{Dr} is a loop algebra type realization of the quantum affine algebra.
It was introduced in studying finite dimensional representations of quantum affine algebras,
and has since played an important role in representation theory such as in vertex representations
\cite{FJ, J1} and finite dimensional representations of quantum affine algebras
for quiver varieties \cite{GKV, N}.

Drinfeld realization was first proved by Beck \cite{B} using Lusztig's braid groups \cite{L}. One of us \cite{J2} also
gave an elementary proof of the untwisted quantum affine algebras using
$q$-commutators. Subsequently both of these
methods are generalized to twisted quantum affine algebras in \cite{ZJ, JZ3, JZ4} using the Hopf algebraic structures and braid groups.

Two-parameter quantum enveloping algebras
were introduced as generalization of the one-parameter
quantum enveloping algebras \cite{T, BW, BGH1, BGH2}.
It was known that the theory has analogous properties with the
one-parameter counterpart such as a similar Schur-Weyl duality and Drinfeld double structure \cite{BW}.
Recent advances on geometric representations \cite{FL} have realized
the two-parameter quantum groups naturally, where the second parameter
turns out to be closely associated with the Tate twist.

Two-parameter quantum enveloping algebras have a
generalized root space structure where the action of generalized Lusztig's braid groups
is not closed, but a Weyl groupoid action sends $U_{r, s}(\mathfrak g)$ to $U_{r', s'}(\mathfrak g)$ \cite{H}.
Therefore the Drinfeld realizations of two-parameter quantum enveloping algebras
cannot be studied by the braid group action.
Hu, Rosso and Zhang \cite{HRZ} first studied the Drinfeld realization of the two-parameter quantum affine algebra in type $A$
by constructing its quantum
Lyndon basis and the vertex representation. Furthermore, vertex representations of two-parameter quantum affine algebras in other types
were also considered in \cite{HZ1, HZ2, GHZ}.

The goal of this paper is twofold. First, we extend the $q$-commutator approach \cite{J2} to derive Drinfeld realizations for all twisted two-parameter quantum affine algebras using a new method, and establish the isomorphism between the Drinfeld realization and the Drinfeld-Jimbo form of the two-parameter quantum
enveloping algebras. A simpler set of generators is used to replace
the original full set of the Drinfeld generators in the quantized algebra, which enables us to
simplify many computations involved with Drinfeld generators and
Drinfeld-Jimbo generators. Thus the current work in two parameter cases
contains a brand new proof of the Drinfeld realization for quantum affine algebras as a
special case.

Second, our new method gives explicit formulae for the Hopf algebra structure
of ${U}_{r,s}(\widehat{\mathfrak g})$ in terms of the Drinfeld generators.
It was recently announced by Guay and Nakajima that the affine Yangian algebra $\mathrm{Y}(\widehat{\mathfrak g})$
has a simple Hopf algebra structure
in terms of the Drinfeld generators. The special case of our result for the
quantum affine algebra induces a
Hopf algebra structure for the double affine Yangian algebra $\mathrm{DY}({\mathfrak g})$ in view of \cite{GT}.

The paper is organized as follows. After a quick introduction
of preliminaries in section 2, we introduce the
two-parameter twisted quantum affine algebras
in the Drinfeld-Jimbo form in section 3. The loop algebra formulation
of the two-parameter twisted quantum affine algebras was given
in section 4. We obtain in section 5 that the Drinfeld realization
is isomorphic to the Drinfeld-Jimbo form as associative algebras.
This isomorphism is proved using a new method in sections 6 and 7,  which is based on a set of
simple generators. In section 8, we define the actions of a commultiplication  on the simple generators of
 Drinfeld realization, thus we can obtain the Hopf algebra structure of Drinfeld realization. Moreover,
 we show that there exists a  Hopf algebra isomorphism between the above two realizations.

\section{Definitions and Preliminaries}
\subsection{Finite order automorphisms of $\frak{g}$}
We begin with a brief review of basic terminologies and notations of twisted affine Lie algebras
following \cite{K}, in particular, on finite order automorphisms of the finite dimensional simple Lie algebra.

Let $\frak{g}$ be a simple finite-dimensional Lie algebra with the Cartan matrix $A=(A_{ij})$, $i,\,j \in \{1,\,2,\,\ldots,\, N\}$, of simply laced type $A_N\, (N\geqslant 2)$, $D_N\, (N>4)$, $E_6$ or $D_4$. Let $\sigma$ be an  Dynkin diagram automorphism of $\frak{g}(A)$ of order $k$. We denote by $I=\{1,\,2,\,\ldots,\,n\}$ the set of $\sigma$-orbits on $\{1,\,2,\,\ldots,\, N\}$. The action of $\sigma$ on the Dynkin diagram is listed as follows:

\begin{equation*}
\begin{split}
A_N: &\, \sigma(i)=N+1-i,\\
D_N: &\, \sigma(i)=i, 1\leqslant i\leqslant N-2; \, \sigma(N-1)=N,\\
D_4: &\, \sigma(1,2,3,4)=(3,2,4,1),\\
E_6: &\, \sigma(i)=6-i, 1\leqslant i\leqslant 5; \, \sigma(6)=6,
\end{split}
\end{equation*}

Let $\omega=exp\frac{2\pi i}{k}$ be the primitive $k$th root of unity. Then we have
the $\mathbb{Z}/r\mathbb{Z}$-graded decomposition
$$\frak{g}(A)=\bigoplus_{j\in \mathbb{Z}/k\mathbb{Z}}\frak{g}_{_j},$$
 where $\frak{g}_{_j}$ is the eigenspace relative to the eigenvalue $\omega^j$. Subsequently $\frak{g}_{_0}$ is a Lie subalgebra of $\frak{g}(A)$. Obviously, the nodes of the Dynkin diagram of $\frak{g}_{_0}$ are indexed by $I$.

\subsection{Twisted affine Lie algebras}\,
 For a nontrivial automorphism $\sigma$ of the Dynkin diagram, the twisted affine Lie algebra $\hat{\frak{g}}^{\sigma}$ is the following graded algebra:
$$\hat{\frak{g}}^{\sigma}=\Big(\bigoplus_{j\in \mathbb{Z}}\frak{g}_{_{[j]}} \otimes \mathbb{C}t^j\Big)\oplus \mathbb{C}c
\oplus \mathbb{C}d,$$
where $c$ is the central element and $ad(d)=t\frac{d}{dt}$. Here we use $[j]$ to denote $j (mod\, k)$. Thus the twisted affine Lie algebra $\hat{\frak{g}}^{\sigma}$ is the universal central extension of the twisted loop algebra. We denote by $\hat{I}=I \bigcup \{0\}$ the nodes set of Dynkin diagram of $\hat{\frak{g}}^{\sigma}$.

Let $A^{\sigma}=(a_{ij}^{\sigma})\, (i,\,j\in \hat{I})$ be the Cartan matrix of
the twisted affine Lie algebra $\hat{\frak{g}}^{\sigma}$ of type $X_N^{(r)}$. The Cartan matrix $A^{\sigma}$ is symmetrizable, i.e., there exists a diagonal matrix $D=diag(d_i|i\in \hat{I})$ such that $DA^{\sigma}$ is symmetric.
Let $\alpha_i\, (i\in \hat{I})$ be the simple roots of $\hat{\frak{g}}^{\sigma}$, and $\Delta$, $\hat{\Delta}$ the root systems of  $\frak{g}_0$, $\hat{\frak{g}}^{\sigma}$ respectively. Then $\Delta=\{\alpha_1,\,\ldots,\, \alpha_n\}$ and  $\hat{\Delta}=\{\alpha_0\}\bigcup \Delta$, where $\alpha_0=\delta-\theta$, $\delta$ is the simple imaginary root and $\theta$ is the maximal root of $\frak{g}_0$. We let $(\ , \ )$ be the canonical bilinear form
of the Cartan subalgebra such that $\frac{2(\alpha_i, \alpha_j)}{(\alpha_i, \alpha_i)}=a_{ij}^{\sigma}$, $i, j\in \hat{I}$.
Let $d_i=\frac{(\alpha_i, \alpha_i)}2$, $i \in \hat{I}$.

\section{Two-parameter twisted quantum affine algebras $U_{r,s}(\hat{\frak{g}}^{\sigma})$}

We assume that the ground field $\mathbb{K}$ is $\mathbb{Q}(r,s)$, the field of rational functions in two
indeterminates $r$, $s$ ($r\ne \pm s$).

\subsection{Drinfeld-Jimbo form of $U_{r,s}(\hat{\frak{g}}^{\sigma})$}
Let $r_i=r^{d_i},\,s_i=s^{d_i},\, i\in \hat{I}$.
Define the $(r,s)$-integer by:
$$[n]_i=\frac{r_i^n-s_i^n}{r_i-s_i},\, [n]_i!=\prod_{k=1}^{n}[k]_i.$$

\begin{defi}
The two-parameter twisted quantum affine algebra $U_{r,s}(\hat{\frak{g}}^{\sigma})$ is the unital
associative algebra over $\mathbb{K}$ generated by the elements
$e_j,\, f_j,\, \omega_j^{\pm 1},\, \omega_j'^{\,\pm 1}\, (j\in
\hat{I}),\, \gamma^{\pm\frac{1}2}$, $\gamma'^{\pm\frac{1}2}
$, satisfying the following relations:

\medskip
\noindent $(\mathcal{X}{1})$ \
$\gamma^{\pm\frac{1}2},\,\gamma'^{\pm\frac{1}2}$ are central with
$\gamma=\om_\delta$, $\gamma'=\om'_\delta$, $\gamma\gamma'=(rs)^c$, such
that $\omega_i\,\omega_i^{-1}=\omega_i'\,\omega_i'^{\,-1}=1
$, and
\begin{equation*}
\begin{split}[\,\omega_i^{\pm 1},\omega_j^{\,\pm 1}\,]
=[\,\omega_i^{\pm 1},\omega_j'^{\,\pm 1}\,]
=[\,\omega_i'^{\pm
1},\omega_j'^{\,\pm 1}\,]=0;
\end{split}
\end{equation*}
\noindent$(\mathcal{X}2)$  \hskip 60 pt 
$\omega_je_i\omega_j^{-1}=\langle i, j\rangle\,e_i,~~~~~~~~~~~
\omega_jf_i\omega_j^{-1}=\langle i, j\rangle^{-1}\,f_i;
$

\noindent$(\mathcal{X}3)$ \hskip 60 pt
$\omega'_je_i\omega^{' -1}_j=\langle j, i\rangle^{-1}\,e_i,~~~~~~~~~~~
\omega'_jf_i\omega^{' -1}_j=\langle j, i\rangle\,f_i;
$

\noindent$(\mathcal{X}4)$ 
\hskip 60 pt
$[\,e_i, f_j\,]=\frac{\delta_{ij}}{r-s}(\omega_i-\omega'_i);$

\noindent$(\mathcal{X}5)$ $(r,s)$-Serre relations for $e_i$: for any $i\ne j\in \hat{I}$, 
\begin{gather*}
\bigl(\text{ad}_l\,e_i\bigr)^{1-a_{ij}}\,(e_j)=0,
\end{gather*}
$(\mathcal{X}6)$ $(r,s)$-Serre relations for $f_i$: for any $i\ne j\in \hat{I}$, 
\begin{gather*}
\bigl(\text{ad}_r\,f_i\bigr)^{1-a_{ij}}\,(f_j)=0,
\end{gather*}
where the left-adjoint action $\text{ad}_l\,e_i$
and right-adjoint action $\text{ad}_r\,f_i$ are given in next section to save space (cf. Proposition \ref{hopf}).

Here the structural constants $\langle i,\,j\rangle$ in relations $(\mathcal{X}2)$ and $(\mathcal{X}3)$ are given in the two-parameter quantum Cartan matrix $J$  defined below. In fact, the  submatrix after removing the zeroth row and zeroth column of $J$ is exactly the two-parameter quantum Cartan matrix of the finite dimensional type, while the data of the zeroth row and the zeroth column are compatible with the results of section 6. 
We list the two-parameter quantum Cartan matrices of the twisted cases for future reference.

For type $A_{2n-1}^{(2)}$,

$$J=\left(\begin{array}{cccccc}
rs^{-1}& (rs)^{-1}& r^{-1} & \cdots & 1 & (rs)^2 \\
rs & rs^{-1} & r^{-1}  & \cdots & 1 & 1\\
s & s & rs^{-1}  & \cdots & 1 & 1\\
\cdots &\cdots &\cdots & \cdots & \cdots & \cdots\\
1 & 1 & 1  & \cdots & rs^{-1} & r^{-2}\\
 (rs)^{-2} & 1 & 1 & \cdots & s^2 & r^2s^{-2}
\end{array}\right).$$

For type $\mathrm{A}_{2n}^{(2)}$,
$$J=\left(\begin{array}{cccccc}
r^2s^{-2}& r^{-2}& 1 & \cdots & 1 & rs \\
s^2 & rs^{-1} & r^{-1}  & \cdots & 1 & 1\\
1 & s & rs^{-1}  & \cdots & 1 & 1\\
\cdots &\cdots &\cdots & \cdots & \cdots & \cdots\\
1 & 1 & 1  & \cdots & rs^{-1} & r^{-1}\\
 (rs)^{-1} & 1 & 1 & \cdots & s & r^{\frac{1}{2}}s^{-\frac{1}{2}}
\end{array}\right).$$

For type $\mathrm{D}_{n+1}^{(2)}$,
$$J=\left(\begin{array}{cccccc}
rs^{-1}& r^{-2}& 1 & \cdots & 1 & rs \\
s^2 & r^{2}s^{-2} & r^{-2}  & \cdots & 1 & 1\\
1 & s^2 & r^2s^{-2}  & \cdots & 1 & 1\\
\cdots &\cdots &\cdots & \cdots & \cdots & \cdots\\
1 & 1 & 1  & \cdots & r^{2}s^{-2} & r^{-2}\\
 (rs)^{-1} & 1 & 1 & \cdots & s^2 & rs^{-1}
\end{array}\right).$$

For type $\mathrm{D}_{4}^{(3)}$,
$$J=\left(\begin{array}{cccccc}
rs^{-1}& r^{-2}s^{-1}& (rs)^{3} \\
rs^2 & rs^{-1} & r^{-3}  \\
(rs)^{-3} & s^3 & r^{3}s^{-3}
\end{array}\right).$$

For type  $\mathrm{E}_{6}^{(2)}$
$$J=\left(\begin{array}{ccccccc}
rs^{-1}& r^{-2}s^{-1}& (rs)^{-1} &(rs)^{2} & (rs)^2\\
rs^2 & rs^{-1} & r^{-1} &1 & 1 \\
rs &s &rs^{-1} & r^{-2} & 1 \\
(rs)^{-2} & 1 & s^2 & r^2s^{-2} & r^{-2}\\
(rs)^{-2}& 1 & 1 & s^2& r^2s^{-2}
\end{array}\right)$$

\end{defi}
\smallskip

\begin{remark} \,  For later purpose, we write down explicitly relation $(\mathcal{X}5)$
case by case as follows.

 {\bf Case (I)}: In the case of $A_{2n-1}^{(2)}$, for $i=1,\,2 \ldots, n-1$ and $j,\,k\in \hat{I}$ such that $a_{jk}^{\sigma}=0$, we have the $(r,s)$-Serre relations:
\begin{eqnarray*}
& &e_j e_k-\langle k, j\rangle e_k e_j=0,\\
& &e_i^2e_{i+1}-(r_i+s_i)e_ie_{i+1}e_i+(r_is_i)e_{i+1}e_i^2=0,\\
& &e_ie_{i+1}^2-(r_{i+1}+s_{i+1})e_{i+1}e_{i}e_{i+1}+(r_{i+1}s_{i+1})e_{i+1}^2 e_i=0,\\
& &e_0^2e_{2}-(r+s)e_0e_{2}e_0+(rs)e_{2}e_0^2=0\\
& &e_0e_{2}^2-(r+s)e_{2}e_{0}e_{2}+(rs)e_{2}^2 e_0=0.
\end{eqnarray*}

{\bf Case (II)}: In the case of $A_{2n}^{(2)}$, for $i=1,\,2 \ldots, n-1$ and $j,\,k\in \hat{I}$ such that $a_{jk}^{\sigma}=0$, we have the $(r,s)$-Serre relations:
\begin{eqnarray*}
& &e_j e_k-\langle k, j\rangle e_k e_j=0,\\
& &e_i^2e_{i+1}-(r_i+s_i)e_ie_{i+1}e_i+(r_is_i)e_{i+1}e_i^2=0,\\
& &e_ie_{i+1}^2-(r_{i+1}+s_{i+1})e_{i+1}e_{i}e_{i+1}+(r_{i+1}s_{i+1})e_{i+1}^2 e_i=0,\\
& &e_0^2e_{1}-(r^2+s^2)e_0e_{1}e_0+(rs)^2e_{1}e_0^2=0,\\
& &e_0e_{1}^3-(rs)[3]e_{1}e_{0}e_{1}^2+[3]e_{1}^2 e_0e_1-(rs)^3e_1^3e_0=0,\\
& &e_{n-1}^2e_{n}-(r+s)e_{n-1}e_{n}e_{n-1}+(rs)e_{n}e_{n-1}^2=0,\\
& &e_{n-1}e_{n}^3-(rs)^{\frac{1}{2}}[3]_{\frac{1}{2}}e_{n}e_{n-1}e_{n}^2+[3]_{\frac{1}{2}}e_{n}^2 e_{n-1}e_n-(rs)^{\frac{3}{2}}e_n^3e_{n-1}=0.
\end{eqnarray*}

{\bf Case (III)}: In the case of $D_{n+1}^{(2)}$, for $i=0,\, 1,\ldots, n-1$ and $j,\,k\in \hat{I}$ such that $a_{jk}^{\sigma}=0$, we have the $(r,s)$-Serre relations:
\begin{eqnarray*}
& &e_j e_k-\langle k, j\rangle e_k e_j=0,\\
& &e_i^2e_{i+1}-(r_i+s_i)e_ie_{i+1}e_i+(r_is_i)e_{i+1}e_i^2=0,\\
& &e_ie_{i+1}^2-(r_{i+1}+s_{i+1})e_{i+1}e_{i}e_{i+1}+(r_{i+1}s_{i+1})e_{i+1}^2 e_i=0,\\
& &e_{n-1}^2e_{n}-(r^2+s^2)e_{n-1}e_{n}e_{n-1}+(rs)^2e_{n}e_{n-1}^2=0,\\
& &e_{n-1}e_{n}^3-(rs)[3]e_{n}e_{n-1}e_{n}^2+[3]e_{n}^2 e_{n-1}e_n-(rs)^{3}e_n^3e_{n-1}=0.
\end{eqnarray*}

{\bf Case (IV)}: In the case of $D_{4}^{(3)}$, for $i,\,j\in \hat{I}$ such that $a_{ij}^{\sigma}=0$, we have the $(r,s)$-Serre relations:
\begin{eqnarray*}
& &e_j e_k-\langle k, j\rangle e_k e_j=0,\\
& &e_{0}^2e_{1}-(rs)[2]e_{0}e_{1}e_{0}+(rs)^3e_{1}e_{0}^2=0,\\
& &e_{0}e_{1}^2-(rs)[2]e_{1}e_{0}e_{1}+(rs)^3 e_{1}^2 e_{0}=0,\\
& &e_1e_{2}^2-(r^3+s^3)e_{2}e_{1}e_{2}+(rs)^3e_{2}^2 e_1=0,\\
& &e_1^4e_{2}-[4]e_1^3e_{2}e_1+(rs)\frac{[4][3]}{[2]}e_{1}^2e_2e_1^2
-(rs)^3[4]e_{1}e_2e_1^3+(rs)^6e_{2}e_1^4=0.
\end{eqnarray*}

{\bf Case (V)}: In the case of $E_{6}^{(2)}$, for $i,\,j\in \hat{I}$ such that $a_{ij}^{\sigma}=0$, we have the $(r,s)$-Serre relations:
\begin{eqnarray*}
& &e_i e_j-\langle j, i\rangle e_j e_i=0,\\
& &e_{0}^2e_{1}-rs(r+s)e_{0}e_{1}e_{0}+(rs)^2e_{1}e_{0}^2=0,\\
& &e_{0}e_{1}^2-rs(r+s)e_{1}e_{0}e_{1}+(rs)^2 e_{1}^2 e_{0}=0,\\
& &e_{1}^2e_{2}-(r+s)e_{1}e_{2}e_{1}+(rs)e_{2}e_{1}^2=0,\\
& &e_{1}e_{2}^2-(r+s)e_{2}e_{1}e_{2}+(rs) e_{2}^2 e_{1}=0,\\
& &e_{3}^2e_{4}-(r^2+s^2)e_{3}e_{4}e_{3}+(rs)^2e_{4}e_{3}^2=0,\\
& &e_{3}e_{4}^2-(r^2+s^2)e_{4}e_{3}e_{4}+(rs)^2 e_{4}^2 e_{3}=0,\\
& &e_2e_{3}^2-(r^2+s^2)e_{3}e_{2}e_{3}+(rs)^2e_{3}^2 e_2=0,\\
& &e_2^3e_{3}-[3]e_2^2e_{3}e_2+(rs)[3]e_{2}e_3e_2^2
-(rs)^3e_{3}e_2^3=0.
\end{eqnarray*}

\end{remark}

\medskip

\subsection{Hopf algebra structure}
\setcounter{equation}{0}

One can check the following fact directly.

\begin{prop} {\label{hopf} The two-parameter twisted quantum affine algebra $U_{r,s}(\hat{\frak{g}}^{\sigma})$ is a Hopf algebra with
the comultiplication $\Delta$, the counit $\vep$ and antipode $S$
defined below: for $i\in \hat{I}$, we have
\begin{gather*}
\Delta(\gamma^{\pm\frac{1}2})=\gamma^{\pm\frac{1}2}\otimes
\gamma^{\pm\frac{1}2}, \qquad
\Delta(\gamma'^{\,\pm\frac{1}2})=\gamma'^{\,\pm\frac{1}2}\otimes
\gamma'^{\,\pm\frac{1}2}, \\
\Delta(\omega_i)=\omega_i\ot \omega_i, \qquad \Delta(\omega_i')=\omega_i'\ot \omega_i',\\
\Delta(e_i)=e_i\ot 1+\omega_i\ot e_i, \qquad \Delta(f_i)=f_i\ot \omega_i'+1\ot
f_i,\\
\varepsilon(e_i)=\varepsilon(f_i)=0,\quad \varepsilon(\gamma^{\pm\frac{1}2})
=\varepsilon(\gamma'^{\,\pm\frac{1}2})
=\varepsilon(\omega_i)=\varepsilon(\omega_i')=1,\\
S(\gamma^{\pm\frac{1}2})=\gamma^{\mp\frac{1}2},\qquad
S(\gamma'^{\pm\frac{1}2})=\gamma'^{\mp\frac{1}2},\qquad\\
S(e_i)=-\omega_i^{-1}e_i,\qquad S(f_i)=-f_i\,\omega_i'^{-1},\qquad
S(\omega_i)=\omega_i^{-1}, \qquad S(\omega_i')=\omega_i'^{-1}.
\end{gather*}}
\end{prop}

\begin{remark} (1)  When $r = q = s^{-1}$, the quotient Hopf algebra
of $U=U_{r,s}(\hat{\frak{g}}^{\sigma})$ modulo the Hopf ideal generated by
the elements $\omega'_i-\omega_i^{-1}\, (i\in \hat{I})$ and
$\gamma'^{\,\frac{1}2}-\gamma^{-\frac{1}2}$
is the classical twisted quantum affine algebra $U_q(\hat{\frak{g}}^{\sigma})$ in Drinfeld-Jimbo
type.

\smallskip
\ (2) \ In the Hopf algebra $U_{r,s}(\hat{\frak{g}}^{\sigma})$, the left-adjoint and
right-adjoint actions are defined in the usual manner:
$$
\text{ad}_{ l}\,a\,(b)=\sum_{(a)}a_{(1)}\,b\,S(a_{(2)}), \qquad
\text{ad}_{ r}\,a\,(b)=\sum_{(a)}S(a_{(1)})\,b\,a_{(2)},
$$
where $\Delta(a)=\sum_{(a)}a_{(1)}\ot a_{(2)}$ for any $a$, $b\in
U_{r,s}(\hat{\frak{g}}^{\sigma})$.

\end{remark}

\medskip

\subsection{Triangular decomposition of $U_{r,s}(\hat{\frak{g}}^{\sigma})$}

The two-parameter twisted quantum affine algebra $U_{r,s}(\hat{\frak{g}}^{\sigma})$ is endowed with a Drinfeld double structure (see Proposition \ref{DB}), which is similar to the untwisted cases. The following statement
can be proved by standard arguments similarly (see \cite{HRZ}).

\begin{prop} {\label{DB}
$U_{r,s}(\hat{\frak{g}}^{\sigma})$ is isomorphic to its Drinfeld double
as a Hopf algebra.}
\end{prop}

\smallskip

Let $U^0=\mathbb
K[\om_0^{\pm1},\ldots,\om_n^{\pm1},{\om_0'}^{\pm1},\ldots,{\om_n'}^{\pm1},\, \gamma,\,\gamma']$
denote the Cartan subalgebra of 
$U_{r,s}(\hat{\frak{g}}^{\sigma})$.

 Furthermore, let $U_{r,s}(\widehat{\frak n})$ $($resp.
$U_{r,s}(\widehat{\frak n}^-)$\,$)$ be the subalgebra of
$U_{r,s}(\hat{\frak{g}}^{\sigma})$ generated
by $e_i$ $($resp. $f_i$$)$ for all $i\in \hat{I}$.
Then, we get the standard triangular decomposition
of $U_{r,s}(\hat{\frak{g}}^{\sigma})$.

\begin{coro}
$U_{r,s}(\hat{\frak{g}}^{\sigma})\cong U_{r,s}(\widehat{\frak n}^-)\ot U^0\ot
U_{r,s}(\widehat{\frak n})$ as vector spaces.\hfill\qed
\end{coro}

\begin{defi} (cf. \cite{HRZ})
Let $\tau$ be the $\mathbb{Q}$-algebra anti-automorphism of
$U_{r,s}(\hat{\frak{g}}^{\sigma})$ such that $\tau(r)=s$, $\tau(s)=r$,
$\tau(\la \om_i',\om_j\ra^{\pm1})=\la \om_j',\om_i\ra^{\mp1}$, and
\begin{gather*}
\tau(e_i)=f_i, \quad \tau(f_i)=e_i, \quad \tau(\om_i)=\om_i',\quad
\tau(\om_i')=\om_i,\\
\tau(\gamma)=\gamma',\quad
\tau(\gamma')=\gamma.
\end{gather*}

In fact $\tau$ is an analog of the Chevalley anti-involution on $U_{r,s}(\hat{\frak{g}}^{\sigma})$.
\end{defi}

\medskip

\section{Drinfeld realization  of two-parameter twisted quantum affine algebras}

In this section, we will generalize the result from  the untwisted cases \cite{HZ2} to the twisted ones
and state the Drinfeld realization in the general form.

For convenience, if two roots
$\alpha=\al_{i_1}+\cdots+\al_{i_m},\,\beta=\al_{j_1}+\cdots+\al_{j_m}$
are decompositions into simple roots, we
denote  $\langle \al,\, \beta \rangle
=\prod\limits_{k=1}^{m}\prod\limits_{l=1}^{n}\langle
i_k,\, j_l \rangle$

\subsection{Generating functions for two-parameter cases}

To state Drinfeld realization of
two-parameter quantum affine algebra $U_{r,s}(\hat{\frak{g}}^{\sigma})$,
we need to define some functions $g_{ij}^{\pm}(z)$.

For $i,\,j=1,\,\ldots,\,n$, let
\begin{align*}
F_{ij}^\pm(z,w)&=\prod_{l\in \mathbb{Z}/k\mathbb{Z}}(z-\omega^l(\la i,\sigma^l(j)\ra\la \sigma^l(j),i\ra)^{\pm\frac1{2}}w),\\
G_{ij}^\pm(z,w)&=\prod_{l\in \mathbb{Z}/k\mathbb{Z}}(\la \sigma^l(j),i\ra^{\pm1}z-(\la \sigma^l(j),i\ra\la
i,\sigma^l(j)\ra^{-1})^{\pm\frac1{2}}w).
\end{align*}

For simple roots $\al_i,\,\al_j \in \Delta$, we set $g_{ij}^{\pm}(z)=\sum_{n\in
\mathbb{Z}_+}c^{\pm}_{\al_i,\,\al_j,\,n}z^{n}=\sum_{n\in
\mathbb{Z}_+}c^{\pm}_{ijn}z^{n}$, 
where the coefficients $c^{\pm}_{ijn}$ are determined from the Taylor
series expansion at $\xi=0$ of the
function

$$  \sum_{n\in \mathbb{Z}_+}c^{\pm}_{ijn}\xi^{n}=g^{\pm}_{ij}(\xi)=
\frac{G_{ij}^{\pm}(\xi, 1)} {F_{ij}^{\pm}(\xi, 1)}
$$

To write the relations in a compact form, we need the formal distribution
\begin{align*}
\delta(z)=\sum_{k\in\mathbb{Z}}z^k.
\end{align*}

\begin{defi}   The two-parameter twisted quantum affine algebra
$\mathcal{U}_{r,s}(\hat{\frak{g}}^{\sigma})$ is the unital associative algebra with the generators
\begin{equation*}
\{\, x_i^{\pm}(k), \phi_i(m), \varphi_i(-m),
\gamma^{\frac{1}{2}}, {\gamma'}^{\frac{1}{2}}| i=1,\ldots,n, k\in\mathbb Z, m\in\mathbb Z_+ \,\}
\end{equation*}
satisfying the defining relations in terms of the generating functions:
\begin{gather*}
x_i^{\pm}(z) = \sum_{k \in \mathbb{Z}}x_i^{\pm}(k) z^{-k}, \quad
\phi_i(z) = \sum_{m \in \mathbb{Z}_+}\phi_i(m) z^{-m},\\
 \varphi_i(z) = \sum_{m \in \mathbb{Z}_+}\varphi_i(-m) z^{m}. \end{gather*}
The relations are given as follows.
\begin{gather}
 \gamma'^{\pm\frac1{2}},\, \gamma^{\pm\frac1{2}} \quad \hbox{are
central and invertible such that } \gamma\gamma'=(rs)^c ,\\
x_{\sigma(i)}^{\pm}(k)=\omega^kx_i^{\pm}(k),\, \varphi_{\sigma(i)}(m)=\omega^m\varphi_i(m),\,\phi_{\sigma(i)}(n)=\omega^n\phi_i(n)\\
\varphi_i(0)\phi_j(0)=\phi_j(0)\varphi_i(0),\\
[\,\varphi_i(z),\,\varphi_j(w)\,]=[\,\phi_i(z),\,\phi_j(w)\,]=0, \\
\varphi_i(z)\phi_j(w)
=\phi_j(w)\varphi_i(z)\frac{g_{ij}(zw^{-1}(\gamma\gamma')^{\frac{1}{2}}\gamma')}{g_{ij}(zw^{-1}(\gamma\gamma')^{\frac{1}{2}}\gamma)},\\
\varphi_i(z)x_j^{\pm}(w)\varphi_i(z)^{-1}
=g_{ij}(\frac{z}{w}(\gamma\gamma')^{\frac{1}{2}}\gamma^{\mp
\frac{1}{2}})^{\pm1}x_j^{\pm}(w),  \\
\phi_i(z)x_j^{\pm}(w)\phi_i(z)^{-1}=g_{ji}(\frac{w}{z}(\gamma\gamma')^{\frac{1}{2}}\gamma'^{\pm
\frac{1}{2}})^{\mp1}x_j^{\pm}(w),
\end{gather}
\begin{gather}
[\,x_i^+(z),
x_j^-(w)\,]=\frac{\delta_{ij}}{r_i-s_i}\Big(\delta(zw^{-1}\gamma')\phi_i(w\gamma^{\frac{1}2})
-\delta(zw^{-1}\gamma)\varphi_i(z\gamma'^{-\frac1{2}})\Big),\\
F_{ij}^\pm(z,\,w)\,x_i^{\pm}(z)x_j^{\pm}(w)=G_{ij}^\pm(z,\,w)\,x_j^{\pm}(w)\,x_i^{\pm}(z),\\
 x_i^{\pm}(z)x_j^{\pm}(w)=\langle
j,i\rangle^{\pm1}x_j^{\pm}(w)x_i^{\pm}(z), \qquad\textit{for }
\ a^{\sigma}_{ij}=0,\\
Sym_{z_1,z_2,z_{3}}\Big\{((rs^{-2})^{\mp \frac{k}{4}}z_1-(r^{\frac{k}{4}}+s^{\frac{k}{4}})z_2+
(r^{-2}s)^{\mp \frac{k}{4}}z_3)x_i^{\pm}(z_1)x_i^{\pm}(z_2)x_i^{\pm}(z_{3})\Big\}=0,
\end{gather}
\begin{gather}
\quad\hbox{for} \quad A_{i,\sigma(i)}=-1, \nonumber\\
Sym_{z_1, z_{2}}\Big\{P_{ij}^{\pm}(z_1,z_2)\sum_{t=0}^{t=2}(-1)^t(r_is_i)^{\pm\frac{t(t-1)}{2}}
\Big[{2\atop  t}\Big]_{\pm{i}}x_i^{\pm}(z_1)\cdots x_i^{\pm}(z_t)\\
\hskip0.8cm \times x_j^{\pm}(w)x_i^{\pm}(z_{t+1})\cdots x_i^{\pm}(z_{2})\Big\}=0,\nonumber\\
\quad\hbox{for} \quad A_{i,j}=-1, \quad  \hbox{and}\quad 1\leqslant j<i\leqslant N \quad\hbox{such\quad that}\quad \sigma(i)\neq j,\nonumber\\
Sym_{z_1, z_{2}}\Big\{P_{ij}^{\pm}(z_1,z_2)\sum_{t=0}^{t=2}(-1)^t(r_is_i)^{\mp\frac{t(t-1)}{2}}
\Big[{2\atop  t}\Big]_{\mp{i}}x_i^{\pm}(z_1)\cdots x_i^{\pm}(z_t)\\
\hskip0.8cm \times x_j^{\pm}(w)x_i^{\pm}(z_{t+1})\cdots x_i^{\pm}(z_{2})\Big\}=0,\nonumber\\
\quad\hbox{for} \quad A_{i,j}=-1, \quad  \hbox{and}\quad 1\leqslant i<j\leqslant N \quad\hbox{such\quad that}\quad \sigma(i)\neq j,\nonumber
\end{gather}
where $[l ]_{\pm{i}}= \frac{r_i^{\pm l}-s_i^{\pm
l}}{r_i^{\pm1}-s_i^{\pm1}}$,\, $[l]_{\mp{i}}= \frac{r_i^{\mp
l}-s_i^{\mp l}}{r_i^{\mp1}-s_i^{\mp1}}$, $Sym_{z_1,\ldots, z_k}$ means the symmetrization over the variables $z_1, \ldots, z_k$, and $P_{ij}^{\pm}$ are given by
\begin{eqnarray*}
&&\hbox{If}\,\, \sigma(i)=i,\, \hbox{then}\,\,  P_{ij}^{\pm}(z,w)=1,\,  d_{ij}=k,\\
&&\hbox{If}\,\, A_{i,\sigma(i)}=0, \, \sigma(j)=j, \, \hbox{then} P_{ij}^{\pm}(z,w)=
\frac{z^r(rs^{-1})^{\pm k}-w^r}{z(rs^{-1})^{\pm 1}-w},\,  d_{ij}=k, \\
&&\hbox{If}\,\,  A_{i,\sigma(i)}=0, \, \sigma(j)\neq
j,\, \hbox{then}\,\, P_{ij}^{\pm}(z,w)=1,\, d_{ij}=1/2, \\
&&\hbox{If}\,\,   A_{i,\sigma(i)}=-1,\, \hbox{then}\,\, P_{ij}^{\pm}(z,w)=
z(rs^{-1})^{\pm k/4}+w,\,  d_{ij}=k/2.
\end{eqnarray*}
\end{defi}

\begin{remark}\, Note that in relation (4.1), $\gamma$ and $\gamma'$ are related by the central element $c$. In the one-parameter case, the central element $c$ is absent in the relation, since $\gamma$ and $\gamma'$ are inverse to each other.
\end{remark}

We now give the component form of the two-parameter Drinfeld realization which is equivalent to the earlier formulation.

\begin{defi}\, \label{d:drinfeld} The unital
associative algebra ${\mathcal U}_{r,s}(\hat{\mathfrak {g}^{\sigma}})$
 over $\mathbb{K}$  is generated by the
elements  $x_i^{\pm}(k)$, $a_i(\ell)$, $\om_i^{\pm1}$,
${\om'_i}^{\pm1}$, $\gamma^{\pm\frac{1}{2}}$,
${\gamma'}^{\,\pm\frac{1}2}$, 
$(i\in I$,
$k,\,k' \in \mathbb{Z}$, $\ell,\,\ell' \in \mathbb{Z}\backslash
\{0\})$, subject to the following defining relations:

\noindent $(\textrm{D1})$ \  $\gamma^{\pm\frac{1}{2}}$,
$\gamma'^{\,\pm\frac{1}{2}}$ are central such that
$\gamma\gamma'=(rs)^c$,\,
$\omega_i\,\omega_i^{-1}=\omega_i'\,\omega_i'^{\,-1}=1$ $(i\in I)$,
and for $i,\,j\in I$, one has
\begin{equation*}
\begin{split}
[\,\omega_i^{\pm 1},\omega_j^{\,\pm 1}\,]
=[\,\omega_i^{\pm 1},\omega_j'^{\,\pm 1}\,]
=[\,\omega_i'^{\pm1},\omega_j'^{\,\pm 1}\,]=0.
\end{split}
\end{equation*}
$$x_{\sigma(i)}^{\pm}(l)=\omega^lx_i^{\pm}(l),\, a_{\sigma(i)}(m)=\omega^ma_i(m)\leqno(\textrm{D2})$$
$$[\,a_i(\ell),a_j(\ell')\,]
=\delta_{\ell+\ell',0}\sum_{t=0}^{k-1}\frac{
(\gamma\gamma')^{\frac{|\ell|}{2}}(r_is_i)^{-\frac{\ell A_{i,\sigma^t(j)}}2}[\,\ell
A_{i,\sigma^t(j)}\,]_i}{|\ell|}
\cdot\frac{\gamma^{|\ell|}-\gamma'^{|\ell|}}{r-s},
 \leqno(\textrm{D3})
$$
$$[\,a_i(\ell),~\om_j^{{\pm }1}\,]=[\,\,a_i(\ell),~{\om'}_j^{\pm
1}\,]=0.\leqno(\textrm{D4})
$$
$$
\om_i\,x_j^{\pm}(k)\, \om_i^{-1} =\sum_{t=0}^{k-1}  \langle \sigma^t(j),
i\rangle^{\pm 1} x_j^{\pm}(k),\leqno(\textrm{D5})$$
$$ \om'_i\,x_j^{\pm}(k)\,
\om_i'^{\,-1} = \sum_{t=0}^{k-1} \langle i, \sigma^t(j)\rangle
^{\mp1}x_j^{\pm}(k).
$$
$$
\begin{array}{lll}
[\,a_i(\ell),x_j^{\pm}(k)\,]=\pm \sum_{t=0}^{k-1} \frac{  (\la
i,\,i\ra^{\frac{\ell A_{i,\sigma^t(j)}}{2}}- \la i,\,i\ra^{\frac{-\ell
A_{i,\sigma^t(j)}}{2}})}
{\ell(r_i-s_i)}(\gamma\gamma')^{\frac{\ell }{2}}\gamma'^{\pm\frac{\ell}2}x_j^{\pm}(\ell{+}k),\\
\hskip8.5cm \quad \textit{for} \quad \ell>0,
\end{array}\leqno{(\textrm{D$6_1$})}
$$
$$
\begin{array}{lll}
[\,a_i(\ell),x_j^{\pm}(k)\,]=\pm \sum_{t=0}^{k-1} \frac{ (\la
i,\,i\ra^{\frac{\ell A_{i,\sigma^t(j)}}{2}}- \la i,\,i\ra^{\frac{-\ell
A_{i,\sigma^t(j)}}{2}})}
{\ell(r_i-s_i)}(\gamma\gamma')^{\frac{-\ell }{2}}\gamma^{\pm\frac{\ell}2}x_j^{\pm}(\ell{+}k),\\
 \hskip8.5cm \quad \textit{for} \quad \ell<0,
\end{array}\leqno{(\textrm{D$6_2$})}
$$
$$
F_{ij}^\pm(z,\,w)\,x_i^{\pm}(z)x_j^{\pm}(w)=G_{ij}^\pm(z,\,w)\,x_j^{\pm}(w)\,x_i^{\pm}(z),\leqno{(\textrm{D7})}
$$

$$
[\,x_i^{+}(k),~x_j^-(k')\,]=\frac{\delta_{ij}}{r_i-s_i}\Big(\gamma'^{-k}\,{\gamma}^{-\frac{k+k'}{2}}\,
\phi_i(k{+}k')-\gamma^{k'}\,\gamma'^{\frac{k+k'}{2}}\,\varphi_i(k{+}k')\Big),\leqno(\textrm{D8})
$$
where $\phi_i(0)=\om_i$, $\varphi_i(0)=\om_i'$, and $\phi_i(m)$, $\varphi_i(-m)~(m\in \mathbb{Z}_{\geq 0})$
 are defined as below:
\begin{gather*}\sum\limits_{m=0}^{\infty}\phi_i(m) z^{-m}=\om_i \exp \Big(
(r_i{-}s_i)\sum\limits_{\ell=1}^{\infty}
 a_i(\ell)z^{-\ell}\Big), \\ 
\sum\limits_{m=0}^{\infty}\varphi_i(-m) z^{m}=\om'_i \exp
\Big({-}(r_i{-}s_i)
\sum\limits_{\ell=1}^{\infty}a_i(-\ell)z^{\ell}\Big).  
\end{gather*}
$$x_i^{\pm}(m)x_j^{\pm}(k)=\langle j,i\rangle^{\pm1}x_j^{\pm}(k)x_i^{\pm}(m),
\qquad\ \hbox{for} \quad a^{\sigma}_{ij}=0,\leqno(\textrm{D$9_1$})$$
$$
\begin{array}{lll}
& Sym_{z_1,z_2,z_{3}}\Big\{((rs^{-2})^{\mp \frac{k}{4}}z_1-(r^{\frac{k}{4}}+s^{\frac{k}{4}})z_2+
(r^{-2}s)^{\mp \frac{k}{4}}z_3)x_i^{\pm}(z_1)x_i^{\pm}(z_2)x_i^{\pm}(z_{3})\Big\}=0,\\
&  \hskip1.8cm \quad\hbox{for} \quad A_{i,\sigma(i)}=-1, \\
\end{array} \leqno{(\textrm{D$9_2$})}
$$
$$
\begin{array}{lll}
&Sym_{z_1, z_{2}}\Big\{P_{ij}^{\pm}(z_1,z_2)\sum_{t=0}^{t=2}(-1)^t(r_is_i)^{\pm\frac{t(t-1)}{2}}
\Big[{2\atop  t}\Big]_{\pm{i}}x_i^{\pm}(z_1)\cdots x_i^{\pm}(z_t)\\
&\hskip0.8cm \times x_j^{\pm}(w)x_i^{\pm}(z_{t+1})\cdots x_i^{\pm}(z_{2})\Big\}=0,\\
&\hskip1.8cm \quad\hbox{for} \quad A_{i,j}=-1, \quad  \hbox{and}\quad 1\leqslant j<i\leqslant N \quad\hbox{such\quad that}\quad \sigma(i)\neq j,\\
\end{array} \leqno{(\textrm{D$9_3$})}
$$
$$
\begin{array}{lll}
&Sym_{z_1, z_{2}}\Big\{P_{ij}^{\pm}(z_1,z_2)\sum_{t=0}^{t=2}(-1)^t(r_is_i)^{\mp\frac{t(t-1)}{2}}
\Big[{2\atop  t}\Big]_{\mp{i}}x_i^{\pm}(z_1)\cdots x_i^{\pm}(z_t)\\
&\hskip0.8cm \times x_j^{\pm}(w)x_i^{\pm}(z_{t+1})\cdots x_i^{\pm}(z_{2})\Big\}=0,\\
&\hskip1.8cm\quad\hbox{for} \quad A_{i,j}=-1, \quad  \hbox{and}\quad 1\leqslant i<j\leqslant N \quad\hbox{such\quad that}\quad \sigma(i)\neq j.
\end{array} \leqno{(\textrm{D$9_4$})}
$$
where $[l ]_{\pm{i}}= \frac{r_i^{\pm l}-s_i^{\pm
l}}{r_i^{\pm1}-s_i^{\pm1}}$,\, $[l]_{\mp{i}}= \frac{r_i^{\mp
l}-s_i^{\mp l}}{r_i^{\mp1}-s_i^{\mp1}}$, and $P_{ij}^{\pm}$ are given by
\begin{eqnarray*}
&&\hbox{If}\,\, \sigma(i)=i,\, \hbox{then}\,\,  P_{ij}^{\pm}(z,w)=1,\,  d_{ij}=k,\\
&&\hbox{If}\,\, A_{i,\sigma(i)}=0, \, \sigma(j)=j, \, \hbox{then} P_{ij}^{\pm}(z,w)=
\frac{z^r(rs^{-1})^{\pm k}-w^r}{z(rs^{-1})^{\pm 1}-w},\,  d_{ij}=k, \\
&&\hbox{If}\,\,  A_{i,\sigma(i)}=0, \, \sigma(j)\neq
j,\, \hbox{then}\,\, P_{ij}^{\pm}(z,w)=1,\, d_{ij}=1/2, \\
&&\hbox{If}\,\,   A_{i,\sigma(i)}=-1,\, \hbox{then}\,\, P_{ij}^{\pm}(z,w)=
z(rs^{-1})^{\pm k/4}+w,\,  d_{ij}=k/2.
\end{eqnarray*}
\end{defi}
\medskip

\subsection{The anti-involution $\tau$}

The following analog of Chevalley anti-homomorphism can be checked directly.
\begin{prop} The following $\mathbb Q$-linear and multiplicative
mapping $\tau$ defines an anti-automorphism of $\,{\mathcal{U}}_{r,s}(\hat{\frak{g}}^{\sigma})$:
$\tau(r)=s$, $\tau(s)=r$,
$\tau(\la\om_i',\om_j\ra^{\pm1})=\la\om_j',\om_i\ra^{\mp1}$ and
\begin{gather*}
\tau(\om_i)=\om_i',\quad \tau(\om_i')=\om_i,\\
\tau(\gamma)=\gamma',\quad \tau(\gamma')=\gamma,\\
\tau(a_i(\ell))=a_i(-\ell),\\
\tau(x_i^{\pm}(m))=x_i^{\mp}(-m), \\
\tau(\phi_i(m))=\varphi_i(-m), \quad\tau(\varphi_i(-m))=\phi_i(m).
\end{gather*}
\end{prop}

\subsection{Triangular decomposition of
$\mathcal{U}_{r,s}(\hat{\frak{g}}^{\sigma})$}

Let $\mathcal U_{r,s}(\frak{g}^{\sigma})$ be the subalgebra of $\mathcal{U}_{r,s}(\hat{\frak{g}}^{\sigma})$ generated by $x_i^{\pm}(0),\, \omega_i,\, \omega'_i$, ($i\in I$).
Clearly $\mathcal U_{r,s}(\frak{g}^{\sigma})\cong
U_{r,s}({\frak g}^{\sigma})$, the subalgebra of $U_{r,s}(\frak{g}^{\sigma})$ generated
by $e_i,\, f_i,\, \omega_i,\, \omega'_i$ ($i\in I$).

Using defining relations (D1)--(D9), one can easily show that
$\mathcal{U}_{r,s}(\hat{\frak{g}}^{\sigma})$ has a triangular
decomposition:
$$\mathcal{U}_{r,s}(\hat{\frak{g}}^{\sigma})=
\mathcal{U}_{r,s}(\widetilde{\frak{n}}^-)\otimes\mathcal{U}_{r,s}^0(\widehat{\frak{g}})
\otimes\mathcal{U}_{r,s}(\widetilde{\frak{n}}^+),$$ where
$\mathcal{U}_{r,s}(\widetilde{\frak{n}}^\pm)=\bigoplus_{\alpha\in\dot
Q^\pm}\mathcal{U}_{r,s}(\widetilde{\frak{n}}^\pm)_\alpha$ is
generated respectively by $x_i^\pm(k)$ ($i\in I$), and
$\mathcal{U}_{r,s}^0(\widehat{\frak{g}})$ is the subalgebra
generated by $\om_i^{\pm1}$, $\om_i'^{\pm1}$,
$\gamma^{\pm\frac1{2}}$, $\gamma'^{\pm\frac1{2}}$,
and $a_i(\pm\ell)$ for $i\in I$, $\ell\in \mathbb{N}$.
Namely, $\mathcal{U}_{r,s}^0(\widehat{\frak{g}})$ is generated by
the subalgebra $\mathcal{U}_{r,s}(\widehat{\frak{g}})^0$ and
the quantum Heisenberg subalgebra $\mathcal
H_{r,s}(\widehat{\frak{g}})$, which is generated by the quantum imaginary
root vectors $a_i(\pm\ell)$ ($i\in I$, $\ell\in \mathbb{N}$).

\section{Two-parameter Drinfeld isomorphism theorem}

\subsection{Quantum Lie bracket}\, We recall the quantum Lie bracket from
\cite{J2}.

\begin{defi} For $q_i\in \mathbb K^*=\mathbb{K}\backslash \{0\}$ and $i=1,2,\ldots s-1$,
the Lie q-brackets $[\,a_1, a_2,\ldots,
a_s\,]_{(q_1,\,q_2,\,\ldots,\, q_{s-1})}$ and
$[\,a_1, a_2, \ldots,
a_s\,]_{\la q_1,\,q_2,\,\ldots, \,q_{s-1}\ra}$ are defined
inductively by
 \begin{gather*}    [\,a_1, a_2\,]_{v_1}=a_1a_2-v_1\,a_2a_1,\\
   [\,a_1, a_2, \ldots, a_s\,]_{(v_1,\,v_2,\,\ldots,
  \,v_{s-1})}=[\,a_1, \ldots, [a_{s-1},
  a_s\,]_{v_1}]_{(v_2,\,\ldots,\,v_{s-1})},\\
   [\,a_1, a_2, \ldots, a_s\,]_{\la v_1,\,v_2,\,\ldots,
  \,v_{s-1}\ra}=[\,[\,a_1, a_2]_{v_1} \ldots, a_{s-1}\,]_{\la
  v_2,\,\ldots,\,v_{s-2}\ra}
 \end{gather*}
\end{defi}
It follows from the definition that the quantum brackets satisfy the following identities.
\begin{eqnarray}
&&[\,a, bc\,]_v=[\,a, b\,]_q\,c+q\,b\,[\,a, c\,]_{\frac{v}q},\\
&&[\,ab, c\,]_v=a\,[\,b, c\,]_q+q\,[\,a, c\,]_{\frac{v}q}\,b, \\
&& [\,a,[\,b,c\,]_u\,]_v=[\,[\,a,b\,]_q,
c\,]_{\frac{uv}q}+q\,[\,b,[\,a,c\,] _{\frac{v}q}\,]_{\frac{u}q},\label{b:1}\\
&&[\,[\,a,b\,]_u,c\,]_v=[\,a,[\,b,c\,]_q\,]_{\frac{uv}q}+q\,[\,[\,a,c\,]
_{\frac{v}q},b\,]_{\frac{u}q}.\label{b:2}
\end{eqnarray}

In particular, we have that
\begin{eqnarray}
&&[\,a, [\,b_1, \ldots, b_s\,]_{(v_1,\,\ldots,\,
v_{s-1})}\,]=\sum\limits_i[\,b_1,\ldots,[\,a, b_i\,],
\ldots,b_s\,]_{(v_1,\,\ldots,\, v_{s-1})},\label{b:3}\hskip0.2cm \\
&&[\,a, a, b\,]_{(u,\,
v)}=a^2b-(u{+}v)\,aba+(uv)\,ba^2=(uv)[\,b, a, a\,]_{\la u^{-1},v^{-1}\ra},\label{b:4}\hskip0.2cm \\
&&[\,a, a, a, b\,]_{(u^2,\,uv,\,v^2)}=a^3b-[3]_{u,v}\,a^2ba
+(uv)[3]_{u,v}aba^2-(uv)^3ba^3, \label{b:5}\hskip0.2cm\\
&&[\,a, a, a, a, b\,]_{(u^3, u^2v, uv^2,
v^3)}=a^4b-[4]_{u,v}\,a^3ba+uv\left[4\atop 2\right]_{u,v}\,a^2ba^2 \label{b:6}\hskip0.2cm\\
&&\hskip5cm-\,(uv)^3[4]_{u,v}\,aba^3+(uv)^6ba^4.\nonumber
\end{eqnarray}
where $[n]_{u,v}=\frac{u^n{-}v^n}{u{-}v}$,
$[n]_{u,v}!=[n]_{u,v}\cdots [2]_{u,v}[1]_{u,v}$, $\left[n\atop
m\right]_{u,v}=\frac{[n]_{u,v}!}{[m]_{u,v}![n-m]_{u,v}!}$.

\subsection{Quantum root vectors} \,

In this paragraph, we define the quantum root vectors using the $q$-bracket.
For our purpose, we need to fix a particular path to realize the maximum
root of $\mathfrak g_0$.

Let $\theta=\alpha_{i_{h-1}}+\cdots+\alpha_{i_2}+\alpha_{i_1}$ be the maximum root and let
\begin{align}\label{a1}
X_{\theta}=[e_{i_{h-1}}, [e_{i_{h-2}}, \ldots, [e_{i_2}, e_{i_1}]\cdots ]
\end{align}
be the corresponding root vector in the Lie algebra $\mathfrak g_0$, which gives rise to
a sequence from $\{1, \ldots, n\}$: $i_1, i_2, \ldots, i_{h-1}$. We call such a sequence a
root chain to the maximum root. Clearly root chains are not unique.

From now on we fix such a sequence or root chain to the maximum root: $i_1, i_2, \ldots, i_{h-1}$. We define for $2\leqslant k\leqslant h-1$
\begin{align}
(\alpha_{i_1}+\cdots+\alpha_{i_{k-1}}, \alpha_{i_{k}})=\epsilon_k\neq 0.
\end{align}

\medskip

\begin{remark} Using the above fact, we fix root chains to the maximal root $\theta$ for our five twisted cases
as follows.

\ (1) \ \, For the case of $A_{2n-1}^{(2)}$, the root chain is
$$\alpha_1\rightarrow \alpha_2\rightarrow \cdots \rightarrow\alpha_{n-1}\rightarrow \alpha_n\rightarrow \alpha_{n-1}\rightarrow\cdots \rightarrow\alpha_2$$


\ (2) \ \, For the case of $A_{2n}^{(2)}$, the root chain is
$$\alpha_1\rightarrow \alpha_2\rightarrow \cdots \rightarrow\alpha_{n}\rightarrow \alpha_n\rightarrow \alpha_{n-1}\rightarrow\cdots \rightarrow\alpha_1$$


\ (3) \ \, For the case of $D_{n+1}^{(2)}$, the root chain is
$$\alpha_n\rightarrow \alpha_{n-1}\rightarrow \cdots \rightarrow\alpha_{1}$$


\ (4) \ \, For the case of $E_{6}^{(2)}$, the root chain is
$$\alpha_1\rightarrow \alpha_2 \rightarrow \alpha_{3}\rightarrow \alpha_4\rightarrow \alpha_{2}\rightarrow \alpha_3\rightarrow \alpha_2\rightarrow\alpha_1$$


\ (5) \ \, For the case of $D_{4}^{(3)}$, the root chain is
$$\alpha_1\rightarrow \alpha_2 \rightarrow \alpha_{1}$$


\end{remark}

Note that the root $\alpha_0=\delta-\theta$ in the affine Lie algebras. In the following we list the quantum root vectors $x^-_{\theta}(1)$
and $x^+_{\theta}(-1)$ for root vectors $e_0$ and $f_0$ respectively, corresponding to the root chains given above.
\smallskip

\ {\bf Case(I)} \ \, For $A_{2n-1}^{(2)}$,
if $\alpha_{1t}=\alpha_1+\cdots+\alpha_t\, (2\leqslant t\leqslant n)$ and $\alpha_{11}= \alpha_1$, we define the quantum root vectors associated to the roots $\delta-\alpha_{1t}$  and $-\delta+\alpha_{1t}$ inductively as follows:

$$x_{1t}^-(1)= x_{\alpha_{1t}}^-(1)=[\,x_t^-(0),\, x_{1(t-1)}^-(1)\,]_{\langle t, t-1\rangle\cdots \langle t, 1\rangle},$$

$$x_{1t}^+(-1)= x_{\alpha_{1t}}^+(-1)=[\,x_{1(t-1)}^+(-1),\, x_{t}^+(0)\,]_{\langle t-1, t\rangle^{-1}\cdots \langle 1, t\rangle^{-1}}.$$

 Denote $\beta_{1t}=\alpha_1+\cdots+\alpha_n+\alpha_{n-1}+\cdots+\alpha_t\, (2\leqslant t\leqslant n-1)$,  so $\beta_{1(n-1)}=\alpha_{1n}+\alpha_{n-1}$, and $\theta=\beta_{12}$.
We define the quantum root vectors associated to the roots $\delta-\beta_{1t}$  and $-\delta+\beta_{1t}$ inductively as follows:
\begin{align*}
y_{1t}^-(1)&= x_{\beta_{1t}}^-(1)=[\,x_t^-(0),\, y_{1(t-1)}^-(1)\,]_{\langle t, t+1\rangle\cdots \langle t, n\rangle \langle t, n-1\rangle\cdots \langle t, 1\rangle}, \\
y_{1t}^+(-1)&= x_{\beta_{1t}}^+(-1)=[\,y_{1(t-1)}^+(-1),\, x_{t}^+(0)\,]_{\langle t+1, t\rangle^{-1}\cdots \langle n, t\rangle^{-1}\langle n-1, t\rangle^{-1}\cdots\langle 1, t\rangle^{-1}}.
\end{align*}
In particular,
\begin{align*}
&x^-_{\theta}(1)=y_{1\,2}^-(1)\\
&=[\,x_2^-(0),\ldots,x_{n-1}^-(0),\, x_n^-(0),\,\ldots,\,x_{1}^-(1)\,]
_{(s,\,\ldots,\, s,\,s^2,\,r^{-1},\,\ldots,r^{-1})},\\
&x_{\theta}^+(-1)=y_{1\,2}^+(-1)\\
&=[\,x_{1}^+(-1),\,\ldots,\,x_n^+(0),\,
x_{n-1}^+(0),\,\ldots, x_2^+(0)\,]_{\langle r,\,\ldots,\, r,\, r^{2},\,s^{-1},\,\ldots,\, s^{-1}\rangle}.
\end{align*}

\medskip

\ {\bf Case(II)} \ \,For $A_{2n}^{(2)}$,
if $\alpha_{1t}=\alpha_1+\cdots+\alpha_t\, (2\leqslant t\leqslant n)$ and $\alpha_{11}= \alpha_1$, we define the quantum root vectors associated to the roots $\delta-\alpha_{1t}$  and $-\delta+\alpha_{1t}$ inductively as follows:
\begin{align*}
x_{1t}^-(1)&= x_{\alpha_{1t}}^-(1)=[\,x_t^-(0),\, x_{1(t-1)}^-(1)\,]_{\langle t, t-1\rangle\cdots \langle t, 1\rangle},\\
x_{1t}^+(-1)&= x_{\alpha_{1t}}^+(-1)=[\,x_{1(t-1)}^+(-1),\, x_{t}^+(0)\,]_{\langle t-1, t\rangle^{-1}\cdots \langle 1, t\rangle^{-1}}.
\end{align*}

Denote $\beta_{1t}=\alpha_1+\cdots+\alpha_n+\alpha_n+\alpha_{n-1}+\alpha_{n-2}+\cdots+\alpha_t\, (1\leqslant t\leqslant n)$,  so $\beta_{1n}=\alpha_{1n}+\alpha_{n}$, and $\beta_{11}=\theta$. We define the quantum root vectors associated to the roots $\delta-\beta_{1t}$  and $-\delta+\beta_{1t}$ inductively as follows:
\begin{equation*}
y_{1t}^-(1)= x_{\beta_{1t}}^-(1)=[\,x_t^-(0),\, y_{1(t-1)}^-(1)\,]_{\langle t, t+1\rangle\cdots \langle t, n\rangle^2 \langle t, n-1\rangle\cdots \langle t, 1\rangle},
\end{equation*}
where the initial one is $x_{\theta}^-(1)=y_{11}^-(1)$.
\begin{equation*}
y_{1t}^+(-1)= x_{\beta_{1t}}^+(-1)=[\,y_{1(t-1)}^+(-1),\, x_{t}^+(0)\,]_{\langle t+1, t\rangle^{-1}\cdots \langle n, t\rangle^{-2}\langle n-1, t\rangle^{-1}\cdots\langle 1, t\rangle^{-1}},
\end{equation*}
where the first one is $x_{\theta}^+(-1)=y_{11}^+(-1)$.

In particular,
\begin{align*}
&x^-_{\theta}(1)=y_{1\,n}^-(1)\\
&=[\,x_1^-(0),\ldots,x_{n}^-(0),\, x_n^-(0),\,\ldots,\,x_{1}^-(1)\,]
_{(s,\,\ldots,\, s,\,(rs)^{\frac{1}{2}},\,r^{-1},\,\ldots,r^{-1},\, (rs)^{-1})},\\
&x_{\theta}^+(-1)=y_{1\,n}^+(-1)\\
&=[\,x_{1}^+(-1),\,\ldots,\,x_n^+(0),\,
x_{1}^+(0),\,\ldots, x_n^+(0)\,]_{\langle r,\,\ldots,\, r,\,(rs)^{\frac{1}{2}},\,s^{-1},\,\ldots,s^{-1},\, (rs)^{-1}\rangle}.
\end{align*}

\medskip

\ {\bf Case(III)} \ \, For $D_{n+1}^{(2)}$,
if $\alpha_{nt}=\alpha_n+\alpha_{n-1}+\cdots+\alpha_t\, (1\leqslant t\leqslant n)$ and $\alpha_{nn}= \alpha_n$, so $\theta=\alpha_{n1}$. We define the quantum root vectors associated to the roots $\delta-\alpha_{nt}$  and $-\delta+\alpha_{nt}$ inductively as follows:
\begin{align*}
x_{nt}^-(1)&= x_{\alpha_{nt}}^-(1)=[\,x_t^-(0),\, x_{n(t-1)}^-(1)\,]_{\langle t, t+1\rangle\cdots \langle t, n\rangle},\\
x_{nt}^+(-1)&= x_{\alpha_{nt}}^+(-1)=[\,x_{n(t-1)}^+(-1),\, x_{t}^+(0)\,]_{\langle t+1, t\rangle^{-1}\cdots \langle n, t\rangle^{-1}}.
\end{align*}

In particular,
\begin{align*}
&x^-_{\theta}(1)=x_{n\,1}^-(1)=[\,x_1^-(0),\ldots,x_{n}^-(1)\,]
_{(r^{-2},\,\ldots,\, r^{-2})},\\
&x_{\theta}^+(-1)=x_{n\,1}^+(-1)=[\,x_{n}^+(-1),\,\ldots,\,x_1^+(0)\,]_{\langle s^{-2},\,\ldots,\, s^{-2}\rangle}.
\end{align*}

\medskip

\ {\bf Case(IV)} \ \, For $E_{6}^{(2)}$, if $\eta_{j}=\alpha_{i_1}+\alpha_{i_2}+\cdots+\alpha_{i_j}$\, where $(i_k|1\leq k\leq 8)=(1,2,3,4,2,3,2,1)$, so $\theta=\eta_{8}$. We define the quantum root vectors associated to the roots $\delta-\eta_{j}$  and $-\delta+\eta_{j}$ inductively as follows:
\begin{align*}
z_{j}^-(1)&= x_{\eta_{j}}^-(1)=[\,x_{i_j}^-(0),\, z_{j-1}^-(1)\,]_{\langle i_j, i_{j-1}\rangle\cdots \langle i_j, i_1\rangle}, \\
z_{j}^+(-1)&= x_{\eta_{j}}^+(-1)=[\,z_{j-1}^+(-1),\, x_{i_j}^+(0)\,]_{\langle i_{j-1}, i_j\rangle^{-1}\cdots \langle i_1, i_j\rangle^{-1}}.
\end{align*}

In particular,
\begin{equation*}
\begin{split}
&x^-_{\theta}(1)=z_{8}^-(1)\\
&=[\,x_1^-(0),\,x_2^-(0),\,x_3^-(0),\,x_{2}^-(0),\, x_4^-(0),\,\ldots,\,x_{1}^-(1)\,]
_{(s,\,s^2,\, s^2,\,r^{-1},\,s^{2},\,r^{-2}s^{-1},\,r^{-2}s^{-1})},\\
&x_{\theta}^+(-1)=z_{8}^+(-1)\\
&=[\,x_{1}^+(-1),\,\ldots,\,x_4^+(0),\,
x_{2}^+(0),\,x_3^+(0), x_2^+(0),\,x_1^+(0)\,]_{\langle r,\,r^2,\, r^2,\,s^{-1},\,r^{2},\,r^{-1}s^{-2},\,r^{-1}s^{-2}\rangle}.
\end{split}
\end{equation*}

\medskip

\ {\bf Case(V)} \ \, For $D_{4}^{(3)}$, we only consider the quantum root vectors associated to the roots $\delta-\alpha_{1}-\alpha_2,\, \delta-\theta$  and $-\delta+\alpha_{1}+\alpha_2,\, -\delta+\theta$, where $\theta=2\alpha_1+\alpha_2$ is the maximal root of $G_2$.
\begin{align*}
x_{1\,2}^-(1)&=[\,x_{2}^-(0),\, x_{1}^-(1)\,]_{s^{3}}, \\
x_{\theta}^-(1)&=[\,x_1^-(0),\,x_{2}^-(0),\, x_{1}^-(1)\,]_{(s^{3},\,r^{-2}s^{-1})}.
\end{align*}

On the other hand,
$$x_{1\,2}^+(-1)=[\,x_{1}^+(-1),\, x_{2}^+(0)\,]_{r^{3}}, $$
and $$x_{\theta}^+(-1)=[\,x_{1}^+(-1),\, x_{2}^+(0),\,x_1^+(0)\,]_{\langle r^{3},\,r^{-1}s^{-2}\rangle}.$$

\subsection{Two-parameter Drinfeld isomorphism theorem}
\setcounter{equation}{0}
In this subsection, we establish the isomorphism
between the two-parameter quantum affine algebra $\mathcal{U}_{r,s}(\hat{\frak{g}}^{\sigma})$  and the $(r,s)$-analogue of Drinfeld quantum affinization of
$U_{r,s}(\hat{\frak{g}}^{\sigma})$.  The identification of these two
forms has been proved for the case of $\widehat{\frak{sl}}_{n}$ in \cite{HRZ}. We will
give a new proof for the most general case in the next two sections.

We keep the same notations and assumptions as above. In particular, $i_1,\,\ldots,\, i_{h-1}$ is the fixed
sequence associated
with the maximum root given in Eq. (\ref{a1}).

For simplicity, we denote $$\langle i_j,\, i_{j-1}\ldots i_2 i_{1}\rangle=\langle i_j,\, i_{j-1}\rangle \cdots \langle i_j,\, i_{1}\rangle,$$
and $$\langle i_1i_2\ldots i_{j-1} ,\, i_j\rangle^{-1}=\langle i_1,\, i_j\rangle^{-1}\cdots \langle i_{j-1} ,\, i_j\rangle^{-1}.$$
 For $j=2,\ldots, h-1$, let $t_{i_j}=\frac{q_{i_j}-p_{i_j}}{r_{i_j}-s_{i_j}}$,
$$p_{i_j}=\langle i_j,\, i_{j-1}\ldots i_2 i_{1}\rangle,$$ and $$q_{i_j}=\langle i_1i_2\ldots i_{j-1} ,\, i_j\rangle^{-1}.$$

Now we state our main theorem as follows.

\begin{theo} {\it Let
$\theta=\alpha_{i_1}+\cdots+\alpha_{{i_{h-1}}}$ be the maximal
positive root of a simple Lie algebra $\frak{g}$ and fix the associated root chain.
Then there exists an algebra isomorphism $\Psi:
U_{r,s}(\hat{\frak{g}}^{\sigma}) \longrightarrow {\mathcal
U}_{r,s}(\hat{\frak{g}}^{\sigma})$ given as follows. For each $i\in I,$
\begin{eqnarray*}
\omega_i&\longmapsto& \om_i\\
\omega'_i&\longmapsto& \om'_i \\
\omega_0&\longmapsto& \gamma'^{-1}\, \om_{\theta}^{-1}\\
\omega'_0&\longmapsto& \gamma^{-1}\, \om_{\theta}'^{-1}\\
\gamma^{\pm\frac{1}2}&\longmapsto& \gamma^{\pm\frac{1}2}\\
\gamma'^{\,\pm\frac{1}2}&\longmapsto& \gamma'^{\,\pm\frac{1}2}\\
e_i&\longmapsto& x_i^+(0)\\
f_i&\longmapsto& \frac{1}{p_i}x_i^-(0)\\
e_0&\longmapsto& a x^-_{\theta}(1)\cdot(\gamma'^{-1}\,\om_{\theta}^{-1})\\
f_0 &\longmapsto& (\gamma^{-1}\,{\om'}_{\theta}^{-1})\cdot x_{\theta}^+(-1)
\end{eqnarray*}
where $\om_{\theta}=\om_{i_1}\cdots\,
\om_{i_{h-1}},\,\om'_{\theta}=\om'_{i_1}\cdots\, \om'_{i_{h-1}}$
and 
$$p_i=\begin{cases}
r, \quad \hbox{if} \quad \sigma(i)=i \vspace{3pt}\\
1, \qquad \hbox{otherwise}
\end{cases}.$$ The constant $a\in \mathbb{K}$ is given as follows:
$$a=\left \{\begin{array}{lll} &(rs)^{n-2}, \,
&  \textit{ for  $A_{2n-1}^{(2)}$ ;}\\
&(rs)^{n-2}[2]_n^{-2},\,
& \textit{ for  $A_{2n}^{(2)}$ ;}\\
&(rs)^{2(n-1)},\,
& \textit{ for  $D_{n+1}^{(2)}$  ;}\\
&(rs)^2,\,
&  \textit{ for  $D_{4}^{(3)}$  ;}\\
&(rs)^{5}[2]_3^{-1}, \, &  \textit{ for  $E_6^{(2)}$. }
\end{array}\right.$$}
\end{theo}
\bigskip

We divide the proof into three steps corresponding to three theorems: ({\bf Theorems $\mathcal{A}$, $\mathcal{B}$, $\mathcal{C}$}),
which will be proved in the following two sections.

\section{$\Psi$ is an algebra homomorphism}
\setcounter{equation}{0}

In this section, we show that $\Psi$ is an algebra homomorphism
(Theorem  $\mathcal{A}$). The proof will be divided into five cases.

{\bf Theorem $\mathcal{A}$.}  {\it  The map
$\Psi$ defined above is an algebra homomorphism
from $U_{r,s}(\hat{\frak{g}}^{\sigma})$ to ${\mathcal
U}_{r,s}(\hat{\frak{g}}^{\sigma})$.}

Let $E_i,\,F_i$, $\om_i$, $\om_i'$ denote
the images of $e_i,\,f_i$, $\om_i$, $\om_i'$  ($i\in \hat{I}$) in the
algebra $\mathcal{U}_{r,s}(\hat{\frak{g}}^{\sigma})$ under the map
$\Psi$ respectively. We shall check that the elements $E_i,~F_i,~\om_i,~\om'_i$ $(i\in
\hat{I}),\, \gamma^{\pm\frac{1}2},\, \gamma'^{\,\pm\frac{1}2}
$ satisfy the defining relations
$(\mathcal{X}1)$--$(\mathcal{X}6)$, where $(\mathcal{X}=\mathcal{A}, \mathcal{B}, \mathcal{C}, \mathcal{D}, \mathcal{E})$ are given in Definition 3.1.  First of all, the defining relations
$(\mathcal{X}1)$--$(\mathcal{X}3)$ can be verified directly as in the untwisted case \cite{HRZ}, so we are left to check relations $(\mathcal{X}4)$--$(\mathcal{X}6)$ involved with $i=0$ case by case.

\subsection{Proof of Theorem $\mathcal{A}$ for the case of $U_{r,s}(\mathrm{A}_{2n-1}^{(2)})$}

 For relation $(\mathcal{A}4)$, when $i\neq 0$, it follows from definition that
\begin{align*}
[\,E_0,F_i\,]
&=a\bigl[\,x^-_{\theta}(1)\cdot(\gamma'^{-1}\om_{\theta}^{-1}),\,
\frac{1}{p_i}x_i^-{(0)}\,\bigr]\\
&=-\frac{a}{p_i}\bigl[\,x_i^-{(0)},\,x^-_{\theta}(1)\,\bigr]_{\langle i,~
0 \rangle^{-1}}(\gamma'^{-1}\om_{\theta}^{-1}).
\end{align*}

To prove this, we need the following technical lemma,
which is proved similarly as untwisted types
(see \cite{HZ2}).

\begin{lemm}\label{L:comm1} The following identities are true.
\begin{align}\label{l:1}
 [\,x_{i-1}^-(0),\,y_{i-1\, i+1}^-(1)\,]_{s^{-1}}&=0, \qquad 1<i<n, \\ \label{l:2}
 [\,x_{i}^-(0),\,y_{i\,  i+1}^-(1)\,]_{(rs)^{-1}}&=0,\qquad 1\leq i \leq n-1,\\ \label{l:3} [\,x_2^-(0),\,y_{1\,4}^-(1)\,]&=0,\\ \label{l:4}
[\,x_{i}^-(0),\,y_{1\,  i+2}^-(1)\,]&=0, \qquad 1\leq i \leq n-2\\ \label{l:5}
[\,x_{i-1}^-(0),\,y_{1\,  i+2}^-(1)\,]&=0,\qquad 2\leq i \leq n-2.
 \end{align}
\end{lemm}
\begin{proof} For (6.1), it is easy to get that
\begin{equation*}
\begin{split}
y_{i-1\, i+1}^-(1)
&=\big[\,x_{i+1}^-(0),\,\ldots,\,x_n^-(0),\,x_{n-2}^-(0),\,
\ldots,\, x_{i+1}^-(0),\,\\
&\quad\underbrace{[\,x_{i}^-(0),x_{i-1}^-(1)\,]_s}\,
\big]_{(s,\,\ldots,\,s,\,r^{-1},\,\ldots,\,r^{-1})}\\
&=-(rs)^{\frac{1}{2}}\big[\,x_{i+1}^-(0),\,\ldots,\,x_n^-(0),\,x_{n-2}^-(0),\,
\ldots,\, x_{i+1}^-(0),\,\\
&\quad[\,x_{i-1}^-(0),x_{i}^-(1)\,]_{r^{-1}}\,
\big]_{(s,\,\ldots,\,s,\,r^{-1},\,\ldots,\,r^{-1})}.
\end{split}
\end{equation*}
Then we have that
\begin{equation*}
\begin{split}
[\,x_{i-1}^-&(0),\,y_{i-1\, i+1}^-(1)\,]_{s^{-1}}\\
&=-(rs)^{\frac{1}{2}}\Big[\,x_{i-1}^-(0),\,
\big[\,x_{i+1}^-(0),\,\ldots,\,x_n^-(0),\,x_{n-2}^-(0),\,
\ldots,\, x_{i+1}^-(0),\, \\
&\quad[\,x_{i-1}^-(0),\,x_{i}^-(1)\,]_{r^{-1}}\,\big]_{(s,\,\ldots,\,s,\,r^{-1},\,\ldots,\,r^{-1})}\,\Big]_{s^{-1}}
\qquad\qquad {\textrm{(by (\ref{b:1}) \& (D9$_1$))}}\\
&=-(rs)^{\frac{1}{2}}
\big[\,x_{i+1}^-(0),\,\ldots,\,x_n^-(0),\,x_{n-2}^-(0),\,
\ldots,\, x_{i+1}^-(0),\,\\
&\quad\underbrace{
[\,x_{i-1}^-(0),\,x_{i-1}^-(0),x_{i}^-(1)\,]_{(r^{-1},\,s^{-1})}}\,
\big]_{(s,\,\ldots,\,s,\,r^{-1},\,\ldots,\,r^{-1})}\qquad {\textrm{(=0 by (D9$_3$))}}\\
&=0.
\end{split}
\end{equation*}

\noindent For (6.2), we argue inductively on $i$.
The case $i=n{-}1$ follows from definition:
\begin{eqnarray*}
[\,x_{n-1}^-(0),\,y_{n-1\,n}^-(1)\,]_{(rs)^{-1}}
=[\,x_{n-1}^-(0),\, x_{n}^-(1)\,]_{(rs)^{-1}}=0.
\end{eqnarray*}

Suppose (6.2) holds for the case of $i$, then we have for
the case of $i-1$ that
\begin{equation*}
\begin{split}
&y_{i-1\,  i}^-(1)\\
&=[\,x_{i}^-(0),\,\ldots,\, x_n^-(0),\,x_{n-2}^-(0),\,\ldots,\,
x_i^-(0),\,x_{i-1}^-(1)\,]_
{(s,\,\ldots,\,s,\,r^{-1},\,\ldots,\,r^{-1})}\\
&=r^{-1}s\,\Big[\,x_{i}^-(0),\,\ldots,\,
x_n^-(0),\,x_{n-2}^-(0),\,\ldots,\, [\,x_i^-(1),\,x_{i-1}^-(0)\,]_r\,\Big]_
{(s,\,\ldots,\,s,\,r^{-1},\,\ldots,\,r^{-1})}\\
&\hskip8.5cm\qquad{\textrm{(by (\ref{b:1}) \& $(\textrm{D9}_1)$)}}\\
&=\cdots\cdots \\
&=r^{-1}s\,\bigl[\,x_{i}^-(0),\,[\,x_{i+1}^-(0),\,\ldots,\,
x_n^-(0),\,x_{n-2}^-(0),\ldots,x_i^-(1)\,]_
{(s,\,\ldots,\,s,\,r^{-1},\,\ldots,\,r^{-1})},\\
&\hskip7.4cm
x_{i-1}^-(0)\,\bigr]_{(r,r^{-1})}\qquad {\textrm{(by definition)}}\\
&=r^{-1}s\,[\,x_{i}^-(0),\,y_{i\,i+1}^-(1) ,\,
x_{i-1}^-(0)\,]_{(r,r^{-1})}\qquad\qquad\hskip3.15cm {\textrm{(by (\ref{b:1}))}}\\
&=r^{-1}s\,[\,\underbrace{
[\,x_{i}^-(0),\,y_{i\,i+1}^-(1)\,]_{(rs)^{-1}}}
 ,\,x_{i-1}^-(0)\,]_{rs}\qquad\quad\ \, {\textrm{(=0 by inductive hypothesis)}}\\
&\quad +r^{-2}\,[\,y_{i\,i+1}^-(1) ,\,
[\,x_i^-(0),\,x_{i-1}^-(0)\,]_s\,]_{r^2s}\\
&=r^{-2}\,[\,y_{i\,i+1}^-(1) ,\,
[\,x_i^-(0),\,x_{i-1}^-(0)\,]_s\,]_{r^2s}.
\end{split}
\end{equation*}
Using the above identity, we have
\begin{equation*}
\begin{split}
&[\,x_{i-1}^-(0),\,y_{i-1\,i}^-(1)\,]_{r^{-2}}\\
&\quad=r^{-2}\Big[\,x_{i-1}^-(0),\,
\big[\,y_{i\,i+1}^-(1),\,[\,x_i^-(0),\,x_{i-1}^-(0)\,]_s\,\big]_{r^2s}\,
\Big]_{r^{-2}}\qquad\qquad {\textrm{(by (\ref{b:1}))}}\\
&\quad=r^{-2}\Big[\,[\,x_{i-1}^-(0),\,
y_{i\,i+1}^-(1)\,]_{r^{-1}},\,[\,x_i^-(0),\,x_{i-1}^-(0)\,]_s\,\,
\Big]_{rs}\\
&\quad\quad +r^{-3}\Big[\, y_{i\,i+1}^-(1),\,
\underbrace{[\,x_{i-1}^-(0),\,x_i^-(0),\,x_{i-1}^-(0)\,]_{(s,\,r^{-1})}}
\,\Big]_{(rs)^{2}}\qquad {\textrm{(=0 by $(\textrm{D9}_3)$)}}\\
&\quad=r^{-2}\Big[\,\underbrace{[\,x_{i-1}^-(0),\,
y_{i\,i+1}^-(1)\,]_{r^{-1}}},\,[\,x_i^-(0),\,x_{i-1}^-(0)\,]_s\,\,
\Big]_{rs}.
\end{split}
\end{equation*}
At the same time, we also get
\begin{equation*}
\begin{split}
[\,x_{i-1}^-&(0),\,
y_{i\,i+1}^-(1)\,]_{r^{-1}}\\
&=[\,x_{i-1}^-(0),\,[\,x_{i+1}^-(0),\,\ldots,x_n^-(0),x_{n-2}^-(0),\ldots,
\\
&\hskip3.5cm x_{i+1}^-(0),x_i^-(1)\,]_
{(s,\,\ldots,\,s,\,r^{-1},\,\ldots,\,r^{-1})}\,]_{r^{-1}}\qquad {\textrm{(by $(\textrm{D9}_1)$)}}\\
&=[\,x_{i+1}^-(0),\,\ldots,x_n^-(0),x_{n-2}^-(0),\ldots,
x_{i+1}^-(0),\\
&\quad\underbrace{[\,x_{i-1}^-(0),\,x_i^-(1)\,]_{r^{-1}}}\,]_
{(s,\,\ldots,\,s,\,r^{-1},\,\ldots,\,r^{-1})}\,]_{r^{-1}}\qquad
{\textrm{(by the definition)}}\\
&=-(rs)^{-\frac{1}{2}}[\,x_{i+1}^-(0),\,\ldots,x_n^-(0),x_{n-2}^-(0),\ldots,
x_{i+1}^-(0),\\
&\quad[\,x_{i}^-(0),\,x_{i-1}^-(1)\,]_{s}\,]_
{(s,\,\ldots,\,s,\,
r^{-1},\,\ldots,\,r^{-1})}\,]_{r^{-1}}\qquad\hskip2.1cm {\textrm{(by  definition)}}\\
&=-(rs)^{-\frac{1}{2}}y_{i-1\,i+1}^-(1).
\end{split}
\end{equation*}
Therefore we get
\begin{equation*}
\begin{split}
[\,x_{i-1}^-&(0),\,y_{i-1\,i}^-(1)\,]_{r^{-2}}\\
&=-r^{-\frac{5}{2}}s^{-\frac{1}{2}}[\,y_{i-1\,i+1}^-(1),\,
[\,x_i^-(0),\,x_{i-1}^-(0)\,]_s\,]_{rs}\qquad\quad (\hbox{by (\ref{b:1})})\\
&=-r^{-\frac{5}{2}}s^{-\frac{1}{2}}[\,[\,y_{i-1\,i+1}^-(1),\,
x_i^-(0)\,]_r,\,x_{i-1}^-(0)\,]_{s^{2}}\qquad\quad(\hbox{by definition})\\
&\quad -r^{-\frac{3}{2}}s^{-\frac{1}{2}}[\,x_i^-(0),\,\underbrace{[\,y_{i-1\,i+1}^-(1),\,
x_{i-1}^-(0)\,]_s}\,]_{r^{-1}s}\\
&=(rs)^{-1}[\,y_{i-1\,i}^-(1),\,
x_{i-1}^-(0)\,]_{s^{2}}.
\end{split}
\end{equation*}
Expanding the two sides of the above identity, one gets
$$(1+r^{-1}s)[\,x_{i-1}^-(0),\,y_{i-1\,i}^-(1)\,]_{(rs)^{-1}}=0,$$
which implies that if $r\ne -s$, then
$[\,x_{i-1}^-(0),\,y_{i-1,i}^-(1)\,]_{(rs)^{-1}}=0$. Thus
we have checked (6.2) for the case of $i-1$.
Consequently, (6.2) is true by induction.

\noindent For {\bf (6.3)}, we first note that
\begin{equation*}
\begin{split}
&[\,x_2^-{(0)}, y_{1\,4}^-(1)\,]\hskip5.17cm \hbox{(by definition)}\\
&\quad=[\,x_2^-{(0)},\,
[\,x_4^-(0),\ldots,x_n^-(0),\,x_{n-2}^-(0),\ldots,
x_4^-(0),\,x_{1\,4}^-(1)\,]_{(s,\ldots,s,\,r^{-1},\ldots,r^{-1})}\,]\\
&\hskip7cm \qquad \hbox{(by (\ref{b:1}) \& $(\textrm{D9}_1)$)}\\
&\quad=[\,x_4^-(0),\ldots,x_n^-(0),\,x_{n-2}^-(0),\ldots,
x_4^-(0),\underbrace{[\,x_2^-{(0)},x_{1\,4}^-(1)\,]}\,]_{(s,\ldots,s,\,r^{-1},\ldots,r^{-1})}.
\end{split}
\end{equation*}
 So it suffices to check the relation $[\,x_2^-{(0)},\,x_{1\,4}^-(1)\,]=0$.

 In fact, it is easy to see that
\begin{equation*}
\begin{split}
[\,x_2^-&{(0)},\,x_{1\,4}^-(1)\,]_{r^{-1}s}\hskip5.6cm \hbox{(by definition)}\\
&=[\,x_2^-{(0)},\,
[\,x_3^-(0),\,x_2^-(0),\,x_{1}^-(1)\,]_{(s,\,s)}\,]_{r^{-1}s}
\qquad\qquad\ \;\,\hbox{(by (\ref{b:1}))}\\
&=[\,x_2^-(0),\,[\,x_3^-(0),\,x_{2}^-(0)\,]_s,\,x_{1}^-(1)\,]\,]_{(s,\,r^{-1}s)}
\qquad\qquad\ \hbox{(by (\ref{b:1}))}\\
&\quad +s[\,x_2^-(0),\,x_{2}^-(0),\,\underbrace{[\,x_3^-(0),\,x_{1}^-(1)\,]}\,]_{(1,\,r^{-1}s)}
\qquad\quad\ \;\, \hbox{(=0 by $(\textrm{D9}_1)$)}\\
&=[\,\underbrace{[\,x_2^-(0),\,[\,x_3^-(0),\,x_{2}^-(0)\,]_s\,]_{r^{-1}}},\,x_{1}^-(1)\,]_{s^{2}}
\qquad\quad\quad\ \,\;\hbox{(=0 by $(\textrm{D9}_3)$)}\\
&\quad +r^{-1}[\,[\,x_3^-(0),\,x_{2}^-(0)\,]_s,\,[\,x_2^-(0),\,x_{1}^-(1)\,]_s\,]_{rs}
\qquad\quad\ \hbox{(by definition \& (\ref{b:2}))}\\
&=r^{-1}[\,x_3^-(0),\,\underbrace{[\,x_{2}^-(0),\,[\,x_2^-(0),\,x_{1}^-(1)\,]_s\,]_r}\,]_{s^2}
\qquad\qquad\; \hbox{(=0 by $(\textrm{D9}_2)$)}\\
&\quad +[\,[\,x_3^-(0),\,x_2^-(0),\,x_{1}^-(1)\,]_{(s,\,s)},\,x_{2}^-(0)\,]_{r^{-1}s}
\qquad\qquad \hbox{(by definition)}\\
&=[\,x_{1\,4}^-(1),\,x_{2}^-(0)\,]_{r^{-1}s}.
\end{split}
\end{equation*}
Then, we obtain
$(1+r^{-1}s)[\,x_2^-{(0)},\,x_{1\,4}^-(1)\,]=0$.
When $r\neq -s$, we arrive at our required conclusion
$[\,x_2^-{(0)},\,x_{1\,4}^-(1)\,]=0$.

 Eqs. (6.4) and (6.5) can be verified similarly, which are left to the readers.
\end{proof}
\medskip

The following three lemmas are needed for later purpose.

\begin{lemm}\, \label{l:6} One has that $[x_i^-(0),x_{i-1}^-(0),x_{i}^-(0),x_{i+1}^-(0)]_{(r^{-1},r^{-1},1)}=0$.
\end{lemm}
\begin{proof}\, Using \ref{b:1} and the Serre relations, we have
 \begin{eqnarray*}
&&[\,x_i^-(0),\,[\,x_{i-1}^-(0),\,x_i^-(0),\,x_{i+1}^-(0)\,]_{(r^{-1},\,r^{-1})}
\,]_{r^{-1}s}
\qquad {\hbox{(using (\ref{b:1}))}}\\
&=&[\,x_i^-(0),\,[\,x_{i-1}^-(0),\,x_i^-(0)\,]_{r^{-1}},\,x_{i+1}^-(0)
\,]_{(r^{-1},\,r^{-1}s)}
\qquad {\hbox{(using (\ref{b:1}))}}\\
&&+r^{-1}[\,x_i^-(0),\,x_i^-(0),\,\underbrace{[\,x_{i-1}^-(0),\,x_{i+1}^-(0)\,]}
\,]_{(1,\,r^{-1}s)}
\qquad {\hbox{(=0 by the Serre relation)}}\\
&=&[\,\underbrace{[\,x_i^-(0),\,[\,x_{i-1}^-(0),\,x_i^-(0)\,]_{r^{-1}}\,]_{s}}
,\,x_{i+1}^-(0)\,]_{r^{-2}}
\qquad {\hbox{(=0 by the Serre relation)}}\\
&&+s[\,[\,x_{i-1}^-(0),\,x_i^-(0)\,]_{r^{-1}}
,\,[\,x_{i}^-(0),\,x_{i+1}^-(0)\,]_{r^{-1}}
\,]_{(rs)^{-1}}\qquad {\hbox{(using (\ref{b:2}))}}\\
&=&s[\,x_{i-1}^-(0),\,
\underbrace{[\,x_i^-(0),\,x_{i}^-(0),\,x_{i+1}^-(0)\,]_{(r^{-1},s^{-1})}}
\,]_{r^{-2}}\qquad {\hbox{(=0 by the Serre relation)}}\\
&&+[\,[\,x_{i-1}^-(0),\,x_i^-(0),\,x_{i+1}^-(0)\,]_{(r^{-1},r^{-1})}
,\,x_{i}^-(0)\,]_{r^{-1}s},
\end{eqnarray*}
which implies that $(1+r^{-1}s)[\,x_i^-(0),\,
[\,x_{i-1}^-(0),\,x_i^-(0),\,x_{i+1}^-(0)\,]_{(r{-1},\,r^{-1})}\,]=0.$
Thus if $r\ne -s$, then $$[\,x_i^-(0),\,
[\,x_{i-1}^-(0),\,x_i^-(0),\,x_{i+1}^-(0)\,]_{(r{-1},\,r^{-1})}\,]=0.$$
\end{proof}

\begin{lemm}\, \label{l:7} Using the same notations, we have $[\,x_{n-1}^-(0),x_{1\, n-1}^-(0)\,]_{r}=0$.
\end{lemm}
\begin{proof}\, Combining \ref{b:1} with the Serre relations, we get
 \begin{eqnarray*}
&&[\,x_{n-1}^-(0),x_{1\, n-1}^-(0)\,]_{r}
\qquad {\hbox{(by definition)}}\\
&=&[\,x_{n-1}^-(0),\,[\,x_{n-1}^-(0),\,x_{n-2}^-(0),\,x_{1\,n-3}^-(1)
\,]_{(s,\,s)}\,]_{r}
\qquad {\hbox{(using (\ref{b:1}))}}\\
&=&[\,x_{n-1}^-(0),\,[\,x_{n-1}^-(0),\,x_{n-2}^-(0)\,]_{s},\,x_{1\,n-3}^-(1)
\,]_{(s,\,r)}
\qquad {\hbox{(using (\ref{b:1}))}}\\
&+&s[\,x_{n-1}^-(0),\,x_{n-2}^-(0),\,\underbrace{[\,x_{n-1}^-(0),\,x_{1\,n-3}^-(1)\,]_{1}}
\,]_{(1,\,r)}
\qquad {\hbox{(=0 by the Serre relation)}}\\
&=&[\,\underbrace{[\,x_{n-1}^-(0),\,x_{n-1}^-(0),\,x_{n-2}^-(0)\,]_{(s,\,r)}},\,x_{1\,n-3}^-(1)
\,]_{s}\qquad {\hbox{(=0 by the Serre relation)}}\\
&&+r[\,[\,x_{n-1}^-(0),\,x_{n-2}^-(0)\,]_{s},\,\underbrace{[\,x_{n-1}^-(0),\,x_{1\,n-3}^-(1)\,]_1}
\,]_{r^{-1}s}\quad {\hbox{(=0 by the Serre relation)}}\\
&=&0.
\end{eqnarray*}
\end{proof}

\begin{lemm}\, \label{l:8} We have that $[\,x_{n}^-(0),x_{1\, n}^-(1)\,]_{r^2}=0$.
\end{lemm}
\begin{proof}\, Using \ref{b:1} and the Serre relations, we obtain
 \begin{eqnarray*}
&&[\,x_{n}^-(0),x_{1\, n}^-(1)\,]_{r^2}
\qquad {\hbox{(by definition)}}\\
&=&[\,x_{n-1}^-(0),\,[\,x_{n}^-(0),\,x_{n-1}^-(0),\,x_{1\,n-2}^-(1)
\,]_{(s,\,s^2)}\,]_{r}
\qquad {\hbox{(using (\ref{b:1}))}}\\
&=&[\,x_{n}^-(0),\,[\,x_{n}^-(0),\,x_{n-1}^-(0)\,]_{s^2},\,x_{1\,n-2}^-(1)
\,]_{(s,\,r^2)}
\qquad {\hbox{(using (\ref{b:1}))}}\\
&+&s^2[\,x_{n}^-(0),\,x_{n-1}^-(0),\,\underbrace{[\,x_{n}^-(0),\,x_{1\,n-2}^-(1)\,]_{1}}
\,]_{(s^{-1},\,r^2)}
\qquad {\hbox{(=0 by the Serre relation)}}\\
&=&[\,\underbrace{[\,x_{n}^-(0),\,x_{n}^-(0),\,x_{n-1}^-(0)\,]_{(s^2,\,r^2)}},\,x_{1\,n-3}^-(1)
\,]_{s}\qquad {\hbox{(=0 by the Serre relation)}}\\
&&+r^2[\,[\,x_{n}^-(0),\,x_{n-1}^-(0)\,]_{s^2},\,\underbrace{[\,x_{n}^-(0),\,x_{1\,n-2}^-(1)\,]_1}
\,]_{r^{-2}s}\quad {\hbox{(=0 by the Serre relation)}}\\
&=&0.
\end{eqnarray*}
\end{proof}

Now we turn to relation $(\mathcal{A}4)$.

\begin{prop}\label{p:1} $\bigl[\,x_i^-{(0)},\,
 x^-_{\theta}(1)\,\bigr]_{\langle i,~
0\rangle^{-1}}=0$,\quad for \ $i\in I$.
\end{prop}

\begin{proof} \ (I) \ If $i=1$,
$\langle 1, 0\rangle=rs$. By Lemma \ref{l:2}, fix $i=1$, we immediately have,
$$[\,x_1^-(0),x_{\theta}^-(1)\,]_{(rs)^{-1}}
=[\,x_{i}^-(0),\,y_{i\,i+1}^-(1)\,]_{(rs)^{-1}}=0.$$

\ (II) \   When $i=2$, $\la 2,\, 0\ra=s$. By the definition of quantum
root vectors,  we get easily
\begin{eqnarray*}
&&[\,x_2^-(0),x_{\theta}^-(1)\,]_{s^{-1}}\qquad {\hbox{(by definition)}}\\
&=&[\,x_2^-(0),\underbrace{
[\,x_2^-(0),\,x_3^-(0),\,y_{1\,4}^-(1)\,]_{(r^{-1},\,r^{-1})}}
\,]_{s^{-1}}\qquad {\hbox{(using (\ref{b:1}))}}\\
&=&[\,x_2^-(0),[\,[\,x_2^-(0),\,x_3^-(0)\,]_{r^{-1}}
,\,y_{1\,4}^-(1)\,]_{r^{-1}}
\,]_{s^{-1}}\qquad {\hbox{(using (\ref{b:1}))}}\\
&&+r^{-1}[\,x_2^-(0),\,x_3^-(0),\,\underbrace{[\,x_2^-(0),\,y_{1\,4}^-(1)\,]_1}
\,]_{(1,\,r^{-1})}\qquad {\hbox{(=0 by (\ref{l:4}))}}\\
&=&[\,\underbrace{[\,x_2^-(0),[\,x_2^-(0),\,x_3^-(0)\,]_{r^{-1}}\,]_{s^{-1}}}
,\,y_{1\,4}^-(1)\,]_{r^{-1}}\quad {\hbox{(=0 by the Serre relation)}}\\
&&+s^{-1}[\,[\,x_2^-(0),\,x_3^-(0)\,]_{r^{-1}}
,\,\underbrace{[\,x_2^-(0),\,y_{1\,4}^-(1)\,]_1}
\,]_{r^{-1}s}\quad {\hbox{(=0 by (\ref{l:4}))}}\\
&=&0.
\end{eqnarray*}

(III) \ When $2<i<n$,  $\langle i,\, 0\rangle=1$.  It follows
from (\ref{b:1}) and definition that
\begin{equation*}
\begin{split}
&x_{\theta}^-(1)\\
=&[\,x_2^-(0),\,\ldots,
x_{i-1}^-(0),\,x_{i}^-(0),\,x_{i+1}^-(0)
,\,y_{1\,i+2}^-(1)\,]_{(r^{-1},\ldots,r^{-1})}\quad {\hbox{(using (\ref{b:1}))}}\\
=&[\,x_2^-(0),\,\ldots,
x_{i-1}^-(0),\,[\,x_{i}^-(0),\,x_{i+1}^-(0)\,]_{r^{-1}}
,\,y_{1\,i+2}^-(1)\,]_{(r^{-1},\ldots,r^{-1})}\quad {\hbox{(using (\ref{b:1}))}}\\
&+r^{-1}[\,x_2^-(0),\,\ldots,
x_{i-1}^-(0),\,x_{i+1}^-(0)
,\,\underbrace{[\,x_{i}^-(0),\,y_{1\,i+2}^-(1)\,]_1}
\,]_{(1,\,r^{-1},\ldots,r^{-1})}\\
&\hskip7.5cm {\hbox{(=0 by (\ref{l:4}))}}\\
=&[\,x_2^-(0),\,\ldots,
[\,x_{i-1}^-(0),\,x_{i}^-(0),\,x_{i+1}^-(0)\,]_{(r^{-1},r^{-1})}
,\,y_{1\,i+2}^-(1)\,]_{(r^{-1},\ldots,r^{-1})}\\
&+r^{-1}[\,x_2^-(0),\,\ldots,
[\,x_{i}^-(0),\,x_{i+1}^-(0)\,]_{r^{-1}}
,\,\underbrace{[\,x_{i-1}^-(0),\,y_{1\,i+2}^-(1)\,]_1}
\,]_{(1,\,r^{-1},\ldots,r^{-1})}\\
&\hskip7.5cm {\hbox{(=0 by (\ref{l:5}))}}\\
=&[\,x_2^-(0),\,\ldots,
[\,x_{i-1}^-(0),\,x_{i}^-(0),\,x_{i+1}^-(0)\,]_{(r^{-1},r^{-1})}
,\,y_{1\,i+2}^-(1)\,]_{(r^{-1},\ldots,r^{-1})}
\end{split}
\end{equation*}
The above result implies that
\begin{eqnarray*}
&&[\,x_i^-(0),x_{\theta}^-(1)\,]\\
&=&[\,x_i^-(0),[\,x_2^-(0),\,\ldots,
[\,x_{i-1}^-(0),\,x_{i}^-(0),\,x_{i+1}^-(0)\,]_{(r^{-1},r^{-1})}
,\,y_{1\,i+2}^-(1)\,]_{(r^{-1},\ldots,r^{-1})}\,]\\
&&\hskip7.5cm {\hbox{(by (\ref{b:1}) and the Serre relation)}}\\
&=&[x_2^-(0),\ldots,
\underbrace{[x_i^-(0),[x_{i-1}^-(0),x_{i}^-(0),x_{i+1}^-(0)]_{(r^{-1},r^{-1})}
,y_{1\,i+2}^-(1)]_{(r^{-1},1)}}]_{(r^{-1},\ldots,r^{-1})}.
\end{eqnarray*}
Therefore it suffices to check
 $$[x_i^-(0),[x_{i-1}^-(0),x_{i}^-(0),x_{i+1}^-(0)]_{(r^{-1},r^{-1})}
,y_{1\,i+2}^-(1)]_{(r^{-1},1)}=0.$$  Actually, it is obvious that
\begin{eqnarray*}
&&[x_i^-(0),[x_{i-1}^-(0),x_{i}^-(0),x_{i+1}^-(0)]_{(r^{-1},r^{-1})}
,y_{1\,i+2}^-(1)]_{(r^{-1},1)}\qquad {\hbox{(using (\ref{b:1}))}}\\
&=&[\,\underbrace{[x_i^-(0),x_{i-1}^-(0),x_{i}^-(0),x_{i+1}^-(0)]_{(r^{-1},r^{-1},1)}}
,y_{1\,i+2}^-(1)\,]_{r^{-1}}\quad {\hbox{(=0 by Lemma \ref{l:6})}}\\
&&+[\,[x_{i-1}^-(0),x_{i}^-(0),x_{i+1}^-(0)]_{(r^{-1},r^{-1})}
,\underbrace{[x_i^-(0),y_{1\,i+2}^-(1)]_{1}}\,]_{r^{-1}}\quad{\hbox{(=0 by (\ref{l:4}))}}\\
&=&0.
\end{eqnarray*}

\ (IV) \ When $i=n$, $\langle n,\, 0\rangle=(rs)^{-2}$. Note that
\begin{equation*}
\begin{split}
&y_{1\,n-1}^-(1)\\
=&[\,x_{n-1}^-(0),\,x_n^-(0),\,
x_{1\,n-1}^-(1)\,]_{(s^2,\,r^{-1})}\qquad {\hbox{(using (\ref{b:1}))}}\\
=&[\,[\,x_{n-1}^-(0),\,x_n^-(0)\,]_{r^{-2}},\,
x_{1\,n-1}^-(1)\,]_{rs^2}\\
&+r^{-2}[\,x_{n}^-(0),\,\underbrace{[\,x_{n-1}^-(0),\,
x_{1\,n-1}^-(1)\,]_{r}}\,]_{(rs)^{2}}\quad {\hbox{(=0 by Lemma \ref{l:7})}}\\
=&[\,[\,x_{n-1}^-(0),\,x_n^-(0)\,]_{r^{-2}},\,x_{1\,n-1}^-(1)\,]_{rs^2},
\end{split}
\end{equation*}
Applying the above result, it is easy to get
\begin{eqnarray*}
&&[\,x_n^-(0),\,y_{1\,n-1}^-(1)\,]_{s^4}\\
&=&[\,x_n^-(0),\,[\,x_{n-1}^-(0),\,x_n^-(0)\,]_{r^{-2}}
,\,x_{1\,n-1}^-(1)\,]_{(rs^2,\,s^4)}\quad{\hbox{(using (\ref{b:1}))}}\\
&=&[\,\underbrace{[\,x_n^-(0),\,[\,x_{n-1}^-(0),\,x_n^-(0)\,]_{r^{-2}}\,]_{s^2}}
,\,x_{1\,n-1}^-(1)\,]_{rs^3}\quad{\hbox{(=0 by the Serre relation)}}\\
&&+s^2[\,[\,x_{n-1}^-(0),\,x_n^-(0)\,]_{r^{-2}}
,\,[\,x_n^-(0),\,x_{1\,n-1}^-(1)\,]_{s^2}\,]_{r}\quad{\hbox{(using (\ref{b:2}))}}\\
&=& s^2[\,x_{n-1}^-(0),\,[\,x_n^-(0),\,x_{1\,n}^-(1)\,]_{r^{2}}\,]_{r^{-3}}\quad{\hbox{(=0 by Lemma \ref{l:8})}}\\
&&+ r^{2}s^{2}[\,[\,x_{n-1}^-(0),\,x_{1\,n}^-(1)\,]_{r^{-1}}
,\,x_{n}^-(0)\,]_{r^{-4}}\\
&=& r^{2}s^{2}[\,[\,x_{n-1}^-(0),\,x_{1\,n}^-(1)\,]_{r^{-1}}
,\,x_{n}^-(0)\,]_{r^{-4}}.
\end{eqnarray*}
Expanding two sides of the above relation, we have that
$$(1+r^{-2}s^2)
[\,x_n^-(0),\,y_{1\,n-1}^-(1)\,]_{(rs)^2}=0.$$ So, if $r\ne -s$,
 it holds that
\begin{eqnarray*}
\begin{split}[\,x_n^-(0),\,y_{1\,n-1}^-(1)\,]_{(rs)^2}=0.
\end{split}
\end{eqnarray*}

Consequently, using (\ref{b:1}) and Serre relation repeatedly, our previous result implies immediately that
\begin{eqnarray*}
&&[\,x_n^-(0),\,x_{\theta}^-(1)\,]_{rs}\qquad{\hbox{(by definition)}}\\
&=&[\,x_n^-(0),\,[\,x_2^-(0),\ldots,\,x_{n-2}^-(0),y_{1\,n-1}^-(1)
\,]_{(r^{-1},\,\ldots,\,r^{-1})}\,]_{(rs)^2} \\
&=&[\,x_2^-(0),\ldots,\,x_{n-2}^-(0),\,\underbrace{
[\,x_n^-(0),\,y_{1\,n-1}^-(1)\,]_{(rs)^2}}
\,]_{(r^{-1},\,\ldots,\,r^{-1})}\,]_{rs} \\
&=&0.
\end{eqnarray*}
Hence Proposition \ref{p:1} is proved.
\end{proof}

Next, we turn to the commutation relation in $(\mathcal{A}4)$ involved with
$i=j=0$.

\begin{prop} {\it
$[\,E_0, F_0\,]=
\frac{\gamma'^{-1}\om_{\theta}^{-1}-\gamma^{-1}\om_{\theta'}^{-1}}{r-s}$.}
\end{prop}

\begin{proof}\, First we consider
\begin{equation*}
\begin{split}
\bigl[\,E_0,
F_0\,\bigr]&=(rs)^{n-2}\bigl[\,x^-_{\theta}(1)\,
\gamma'^{-1}{\om_\theta}^{-1},\,\gamma^{-1}{\om'_\theta}^{-1}
x^+_{\theta}(-1)\,\bigr]\\
&=(rs)^{n-2}\bigl[\,x^-_{\theta}(1),\,x^+_{\theta}(-1)\,\bigr]\,
\cdot(\gamma^{-1}\gamma'^{-1}{\om_\theta}^{-1}{\om'_\theta}^{-1}).
\end{split}
\end{equation*}

Recall the constructions of $x^-_{\theta}(1)$ and $x_{\theta}^+(-1)$ in the present case.
\begin{equation*}
\begin{split}
&x^-_{\theta}(1)=y_{1\,2}^-(1)=[\,x_2^-(0),\ldots, x_n^-(0),\,x_{1\,n-1}^-(1)\,]
_{(s^2,\,r^{-1},\,\ldots,r^{-1})},\\
&x_{\theta}^+(-1)=y_{1\,2}^+(-1)\\
=&[\,[\,[\,x_{1\,n-1}^+(-1),\,x_n^+(0)\,]_{r^2},\,
x_{n-1}^+(0)\,]_{s^{-1}},\ldots, x_2^+(0)\,]_{s^{-1}}.
\end{split}
\end{equation*}
Using the result for $A_{n-1}^{(1)}$ \cite{HRZ}, one has
\begin{equation*}[\,x_{1\,n-1}^-(1),\,x_{1\,n-1}^+(-1)\,]=
\frac{\gamma\om_{\alpha_{1\,n-1}}'-\gamma'\om_{\alpha_{1\,n-1}}}{r-s}.
  \end{equation*}

Using a similar calculation, we have that
\begin{equation*}
\begin{split}
&[\,x_{1\,n}^-(1),\,x_{1\,n}^+(-1)\,]\qquad{\hbox{(by definition)}}\\
=&[\,[\,x_{n}^-(0),\,x_{1\,n-1}^-(1)\,]_{s^2}
,\,[\,x_{1\,n-1}^+(-1),\,x_{n}^+(0)\,]_{r^2}\,]\qquad{\hbox{(using (\ref{b:1}))}}\\
=&[\,[\,[\,x_{n}^-(0),\,x_{1\,n-1}^+(-1)\,],\,x_{1\,n-1}^-(1)\,]_{s^2}
,\,x_{n}^+(0)\,]_{r^2}\quad{\hbox{(=0 by (\ref{b:1}) and (D9))}}\\
&+[\,[\,x_{n}^-(0),\,[\,x_{1\,n-1}^-(1),\,x_{1\,n-1}^+(-1)\,]\,]_{s^2}
,\,x_{n}^+(0)\,]_{r^2}\,{\hbox{(using (D6) and (D9))}}\\
&+[\,x_{1\,n-1}^+(-1),\,[\,
[\,x_{n}^-(0),\,x_{n}^+(0)\,],\,x_{1\,n-1}^-(1)\,]_{s^2}\,]_{r^2}\,{\hbox{(using (D9) and (D6))}}\\
&+[\,x_{1\,n-1}^+(-1),\,[\,x_{n}^-(0),\,
[\,x_{1\,n-1}^-(1),\,x_{n}^+(0)\,]\,]_{s^2}
\,]_{r^2}\quad{\hbox{(=0 by (\ref{b:1}) and (D9))}}\\
=&\gamma\om'_{\alpha_{1\,n-1}}\cdot \frac{\om'_n-\om_n}{r-s}
+\frac{\gamma\om'_{\alpha_{1\,n-1}}-\gamma'\om_{\alpha_{1\,n-1}}}
{r-s}\om_n\\
=&\frac{\gamma\om'_{\alpha_{1\, n}}-\gamma'\om_{\alpha_{1\,n}}}{r-s}.
\end{split}
\end{equation*}

We can proceed in the same way to obtain
\begin{equation*}
\begin{split}
&[\,y_{1\,n-1}^-(1),\,y_{1\,n-1}^+(-1)\,]\qquad{\hbox{(by definition)}}\\
=&[\,[\,x_{n-1}^-(0),\,x_{1\,n}^-(1)\,]_{r^{-1}}
,\,[\,x_{1\,n}^+(-1),\,x_{n}^+(0)\,]_{s^{-1}}\,]
\qquad{\hbox{(using (\ref{b:1}))}}\\
=&[\,[\,[\,x_{n-1}^-(0),\,x_{1\,n}^+(-1)\,],\,x_{1\,n}^-(1)\,]_{r^{-1}}
,\,x_{n-1}^+(0)\,]_{s^{-1}}\,{\hbox{(using (\ref{b:1}),(D9) and (\ref{b:3}))}}\\
&+[\,[\,x_{n-1}^-(0),\,[\,x_{1\,n}^-(1),\,x_{1\,n}^+(-1)\,]\,]_{r^{-1}}
,\,x_{n-1}^+(0)\,]_{s^{-1}}\,{\hbox{(=0 by the above result and (D6))}}\\
&+[\,x_{1\,n}^+(-1),\,[\,
[\,x_{n-1}^-(0),\,x_{n-1}^+(0)\,],\,x_{1\,n}^-(1)\,]_{r^{-1}}
\,]_{s^{-1}}\,{\hbox{(=0 by (D9) and (D6))}}\\
&+[\,x_{1\,n}^+(-1),\,[\,x_{n-1}^-(0),\,
[\,x_{1\,n}^-(1),\,x_{n-1}^+(0)\,]\,]_{r^{-1}}
\,]_{s^{-1}}\,{\hbox{(using (\ref{b:1}),(D9) and (\ref{b:3}))}}\\
=&(rs)^{-1}\gamma\om'_{\alpha_{1\,n}}\cdot \frac{\om'_{n-1}-\om_{n-1}}{r-s}
+(rs)^{-1}\frac{\gamma\om'_{\alpha_{1\,n}}-\gamma'\om_{\alpha_{1\,n}}}
{r-s}\om_{n-1}\\
=&(rs)^{-1}\frac{\gamma\om'_{\beta_{1\,n-1}}-\gamma'\om_{\beta_{1\,n-1}}}{r-s}.
\end{split}
\end{equation*}
By the inductive step it is obvious that
\begin{equation*}
\begin{split}
[\,y_{1\,2}^-(1),\,y_{1\,2}^+(-1)\,]
=(rs)^{2-n}\frac{\gamma\om'_{\beta_{1\,2}}-\gamma'\om_{\beta_{1\,2}}}
{r-s}.
\end{split}
\end{equation*}
Hence we arrive at the required relation.
\end{proof}

We pause to recall the following fact, which will be used in the sequel.
\begin{lemm}\label{lemma1}\, If $X\in U_q({\mathfrak{g}_0})^+$, and $[\,A,\, F_k\,]=0, \forall k\in {I}$, then $A=0$. On the other hand, If $X\in U_q({\mathfrak{g}_0})^-$, and $[\,A,\, E_k\,]=0, \forall k\in {I}$, then $A=0$.
\end{lemm}

We now return to check relations ($\mathcal{A}5$) and ($\mathcal{A}6$).  Indeed,
relations ($\mathcal{A}5$) or ($\mathcal{A}6$) can be obtained from the other one by applying $\tau$.
The following three relations are the main statements. 

\begin{lemm} \, Using the above notations, we have the following relations:

 $(1)$ \ $[\,E_0, \,E_n\,]_{(rs)^{-2}}=0$,

$(2)$ \ $[\,E_2,\,E_2,\, E_0\,]_{(s^{-1},\,r^{-1})}=0$,

$(3)$ \ $[\,F_{n-1},\,F_{n-1},\, F_{n-1},\, F_n\,]_{(r^{-2},\,(rs)^{-1},\,s^{-2})}=0$,

\end{lemm}

\begin{proof} \ (1) \  Combining the definitions and Drinfeld relations, we see that
\begin{equation*}
\begin{split}
&\bigl[\,E_0,\,E_n,\bigr]_{(rs)^{-2}}\\
=&a\,\bigl[\,x_{\theta}^-(1),\,
x_n^+(0)\,\bigr]\cdot {\gamma'}^{-1}\omega_{\theta}^{-1} \\
=&a\,\bigl[\,x_2^-(0),\,\ldots,\, x_{n-1}^-(0),\,[\,
x_n^-(0),\,x_n^+(0)\,], x_{1\,n-1}^-(1)\,\bigr]_{(s^2, r^{-1}, \ldots, r^{-1})}\cdot {\gamma'}^{-1}\omega_{\theta}^{-1}\\
=&a\,\bigl[\,x_2^-(0),\,\ldots,\, \underbrace{[\,x_{n-1}^-(0),\, x_{1\,n-1}^-(1)\,]_{r}}\cdots \omega_n\,\bigr]_{(r^{-1},\,\ldots,\,r^{-1})}\cdot {\gamma'}^{-1}\omega_{\theta}^{-1}\\
=&0,
\end{split}
\end{equation*}
where the last step follows from the following calculations

\begin{eqnarray*}
&&[\,x_{n-1}^-(0),\,x_{1\,n-1}^-(1)\,]_{r}\\
&=&[\,x_{n-1}^-(0),\,\underbrace{[\,x_{n-1}^-(0),\,x_{n-2}^-(0),\,x_{1\,n-3}^-(1)\,]_{(s,\,s)}}\,]_{r}\quad{\hbox{(using (\ref{b:1}))}}\\
&=&[\,x_{n-1}^-(0),\,[\,x_{n-1}^-(0),\,x_{n-2}^-(0)\,]_{s},\,x_{1\,n-3}^-(1)\,]_{(s,\,r)}\quad{\hbox{(using (\ref{b:1}))}}\\
&&+s[\,x_{n-1}^-(0),\,x_{n-2}^-(0),\,\underbrace{[\,x_{n-1}^-(0),\,x_{1\,n-3}^-(1)\,]_{1}}\,]_{(1,\,r)}\quad{\hbox{(=0 by (D10$_2$))}}\\
&=& [\,\underbrace{[\,x_{n-1}^-(0),\,x_{n-1}^-(0),\,x_{n-2}^-(0)\,]_{(s,\,r)}},\,x_{1\,n-3}^-(1)\,]_{s}\quad{\hbox{(using (\ref{b:6}))}}\\
&&+ r[\,[\,x_{n-1}^-(0),\,x_{n-2}^-(0)\,]_{s},\,\underbrace{[\,x_{n-1}^-(0),\,x_{1\,n-3}^-(1)\,]_{1}}\,]_{r^{-1}s}\\
&=&0.
\end{eqnarray*}
\smallskip

(2) \  Relation $(D9)$ yields directly that
\begin{equation*}
\begin{split}
&\bigl[\,x_2^+(0),\,x_{\theta}^-(1)\bigr]\\
=&\bigl[\,[\,x_2^+(0),\, x_2^-(0)\,],\,y_{1\,3}^-(1)\,\bigr]_{r^{-1}} \\
+&\bigl[\,x_2^-(0),\,\ldots,\, x_{n}^-(0),\,\ldots,\,x_3^-(0),\,[\,x_2^+(0),\, x_2^-(0)\,],\, x_{1}^-(1)\,\bigr]_{(s,\,\ldots,\, s,\,s^2,\, r^{-1},\,\ldots,\,r^{-1})}\\
=&(rs)^{-1}y_{1\,3}^-(1)\omega_2,
\end{split}
\end{equation*}

The statement (2) follows from the above results and $[\,x_2^+(0),\, y_{1\,3}^-(1)\,]=0$, which holds by direct calculation.

\smallskip

(3) \  Denote that $X=[\,F_{n-1},\,F_{n-1},\, F_{n-1},\, F_n\,]_{(r^{-2},\,(rs)^{-1},\,s^{-2})}$. It is obvious that $X\in \mathcal{U}_{r,s}(\frak{g}^{\sigma})^-$. Therefore, by Lemma \ref{lemma1}, in order to prove $X=0$, it suffices to check $[\,x_i^+(0),\, X\,]=0$ for  $i\in I$. For $i$ not equal to $n-1$ or $n$, the claim is obvious. So we only need to verify the cases of $i=n-1$ and $i=n$.

First we get from relation (D9) in the case of $i=n$,
\begin{equation*}
\begin{split}
&[\,x_{n}^+(0),\, X\,] \\
=&\frac{1}{2}\,\Big([\,x_n^+(0),\,[\,x_{n-1}^-(0),\,x_{n-1}^-(0),\, x_{n-1}^-(0),\, x_{n}^-(0)\,]_{(r^{-2},\,(rs)^{-1},\,s^{-2})}\,]\Big)\\
=&\frac{1}{2}\,\Big([\,x_{n-1}^-(0),\,x_{n-1}^-(0),\, x_{n-1}^-(0),\, [\,x_n^+(0),\,x_{n}^-(0)\,]\,]_{(r^{-2},\,(rs)^{-1},\,s^{-2})}\Big)\\
=&\frac{r^{-2}s^{-6}}{2}\,\omega'_n [\,x_{n-1}^-(0),\,x_{n-1}^-(0),\, x_{n-1}^-(0)]_{(r^{-1}s,\,1)}\\
=&0.
\end{split}
\end{equation*}

One sees that the same is true for $i=n-1$,
\begin{equation*}
\begin{split}
&[\,x_{n-1}^+(0),\, X\,] \\
=&\frac{1}{2}\,\Big([\,x_{n-1}^+(0),\,[\,x_{n-1}^-(0),\,x_{n-1}^-(0),\, x_{n-1}^-(0),\, x_{n}^-(0)\,]_{(r^{-2},\,(rs)^{-1},\,s^{-2})}\,]\Big)\\
=&\frac{1}{2}\,\Big([\,[\,x_{n-1}^+(0),\,x_{n-1}^-(0)\,],\,x_{n-1}^-(0),\, x_{n-1}^-(0),\, \,x_{n}^-(0)\,]_{(r^{-2},\,(rs)^{-1},\,s^{-2})}\\
&+[\,x_{n-1}^-(0),\,[\,x_{n-1}^+(0),\,x_{n-1}^-(0)\,],\, x_{n-1}^-(0),\, \,x_{n}^-(0)\,]_{(r^{-2},\,(rs)^{-1},\,s^{-2})}\\
&+[\,x_{n-1}^-(0),\, x_{n-1}^-(0),\,[\,x_{n-1}^+(0),\,x_{n-1}^-(0)\,],\, \,x_{n}^-(0)\,]_{(r^{-2},\,(rs)^{-1},\,s^{-2})}\Big)\\
=&\frac{(rs)^{-2}(r+s)}{2}\,[\,x_{n-1}^-(0),\, x_{n-1}^-(0),\,x_{n}^-(0)\,]_{(r^{-2},\,(rs)^{-1})}\omega_{n-1}\\
&-\frac{(rs)^{-2}(r+s)}{2}\,[\,x_{n-1}^-(0),\, x_{n-1}^-(0),\,x_{n}^-(0)\,]_{(r^{-2},\,(rs)^{-1})}\omega_{n-1}\\
=&0.
\end{split}
\end{equation*}
Consequently, theorem $\mathcal{A}$ is proved for the case of $\mathrm{A}_{2n-1}^{(2)}$.
\end{proof}
\medskip

\subsection{Proof of theorem $\mathcal{A}$ for the case of $U_{r,s}(\mathrm{A}_{2n}^{(2)})$}

Let us turn to the case of $\mathrm{A}_{2n}^{(2)}$.  Similarly we only show some key relations $(\mathcal{B}4)$--$(\mathcal{B}6)$ involving $i=0$.

When $i\neq 0$, observe that
\begin{equation*}
\begin{split}
[\,E_0,F_i\,]
&=a\bigl[\,x^-_{\theta}(1)\cdot(\gamma'^{-1}\om_{\theta}^{-1}),\,
\frac{1}{p_i}x_i^-{(0)}\,\bigr]\\
&=-\frac{a}{p_i}\bigl[\,x_i^-{(0)},\,x^-_{\theta}(1)\,\bigr]_{\langle i,~
0 \rangle^{-1}}(\gamma'^{-1}\om_{\theta}^{-1}).
\end{split}
\end{equation*}

Hence, in order to verify  relation $(\mathcal{B}4)$, it is enough to check the following result.

\begin{prop} \label{p:2} $\bigl[\,x_i^-{(0)},\,
 x^-_{\theta}(1)\,\bigr]_{\langle i,~
0 \rangle^{-1}}=0$,\quad for \ $i\in I$.
\end{prop}

Before giving the proof of Proposition \ref{p:2}, we need the following crucial
lemma which can be proved directly.

%
%
%
%

\begin{lemm} One has that
\begin{align}\label{l:10}
&[\,x_{i-1}^-(0),\,y_{1\, i+1}^-(1)\,]=0, \qquad 1<i<n\\ \label{l:11}
&[\,x_{i}^-(0),\,y_{1\,  i}^-(1)\,]_{s^{-1}}=0, \qquad1\leqslant i \leqslant n-1\\ \label{l:12}
&[\,x_{n-1}^-(0),\,x_{1\,  n}^-(1)\,]=0,\\ \label{l:13}
&[\,x_{n-1}^-(0),\,x_{1\,  n-1}^-(1)\,]_r=0\\ \label{l:14}
&[\,x_{n}^-(0),\,y_{1\,  n}^-(1)\,]_r=0.
\end{align}
\end{lemm}

\noindent {\bf Proof of Proposition \ref{p:2}}  \ (I) \ When $i=1$,
$\langle 1,\, 0\rangle=s^2$, one gets from the above (\ref{l:10}),
\begin{eqnarray*}
&&[\,x_1^-(0),x_{\theta}^-(1)\,]_{s^{-2}}\qquad {\hbox{(by definition)}}\\
&=&[\,x_1^-(0),\underbrace{
[\,x_1^-(0),\,x_2^-(0),\,y_{1\,3}^-(1)\,]_{(r^{-1},\,(rs)^{-1})}}
\,]_{s^{-2}}\qquad {\hbox{(using (\ref{b:1}))}}\\
&=&[\,x_1^-(0),[\,[\,x_1^-(0),\,x_2^-(0)\,]_{r^{-1}}
,\,y_{1\,3}^-(1)\,]_{(rs)^{-1}}
\,]_{s^{-2}}\qquad {\hbox{(using (\ref{b:1}))}}\\
&&+r^{-1}[\,x_1^-(0),\,x_2^-(0),\,\underbrace{[\,x_1^-(0),\,y_{1\,3}^-(1)\,]_{s^{-1}}}
\,]_{(1,\,s^{-2})}\qquad {\hbox{(=0 by  (\ref{l:10}))}}\\
&=&[\,\underbrace{[\,x_1^-(0),[\,x_1^-(0),\,x_2^-(0)\,]_{r^{-1}}\,]_{s^{-1}}}
,\,y_{1\,3}^-(1)\,]_{r^{-1}s^{-2}}\quad {\hbox{(=0 by the Serre relation)}}\\
&&+s^{-1}[\,[\,x_1^-(0),\,x_2^-(0)\,]_{r^{-1}}
,\,\underbrace{[\,x_1^-(0),\,y_{1\,3}^-(1)\,]_{s^{-1}}}
\,]_{r^{-1}}\quad {\hbox{(=0 by (\ref{l:10})})}\\
&=&0.
\end{eqnarray*}

\ (II) \   When $1<i<n$, $\la i,\, 0\ra=1$. In this case, it holds by the following direct calculations,
\begin{eqnarray*}
&&[\,x_i^-(0),\,y_{1\,i-1}^-(1)\,]_{r^{-1}s}\qquad {\hbox{(by definition)}}\\
&=&[\,x_i^-(0),\underbrace{
[\,x_{i-1}^-(0),\,x_i^-(0),\,y_{1\,i+1}^-(1)\,]_{(r^{-1},\,r^{-1})}}
\,]_{r^{-1}s}\qquad {\hbox{(using (\ref{b:1}))}}\\
&=&[\,x_i^-(0),[\,[\,x_{i-1}^-(0),\,x_i^-(0)\,]_{r^{-1}}
,\,y_{1\,i+1}^-(1)\,]_{r^{-1}}
\,]_{r^{-1}s}\qquad {\hbox{(using (\ref{b:1}))}}\\
&&+r^{-1}[\,x_i^-(0),\,x_i^-(0),\,\underbrace{[\,x_{i-1}^-(0),\,y_{1\,i+1}^-(1)\,]_1}
\,]_{(1,\,r^{-1}s)}\qquad {\hbox{(=0 by (\ref{l:10}))}}\\
&=&[\,\underbrace{[\,x_i^-(0),[\,x_{i-1}^-(0),\,x_i^-(0)\,]_{r^{-1}}\,]_{s}}
,\,y_{1\,i+1}^-(1)\,]_{r^{-2}}\quad {\hbox{(=0 by the Serre relation)}}\\
&&+s[\,[\,x_{i-1}^-(0),\,x_i^-(0)\,]_{r^{-1}}
,\,\underbrace{[\,x_i^-(0),\,y_{1\,i+1}^-(1)\,]_{r^{-1}}}
\,]_{(rs)^{-1}}\quad {\hbox{( by the definition)}}\\
&=&s[\,[\,x_{i-1}^-(0),\,x_i^-(0)\,]_{r^{-1}}
,\,y_{1\,i}^-(1)\,]_{(rs)^{-1}}\quad {\hbox{( using (\ref{b:1}))}}\\
&=&s[\,x_{i-1}^-(0),\,
\underbrace{[\,x_{i}^-(0),\,\,y_{1\,i}^-(1)\,]_{s^{-1}}}\,]_{r^{-2}}\quad {\hbox{( =0 by (\ref{l:11}))}}\\
&&+[\,[\,x_{i-1}^-(0),\,\,y_{1\,i}^-(1)\,]_{r^{-1}},\,x_i^-(0)\,]_{r^{-1}s}\quad {\hbox{( by the definition)}}\\
&=&[\,y_{1\,i-1}^-(1),\,x_i^-(0)\,]_{r^{-1}s}.
\end{eqnarray*}
As a consequence of above result, it yields that $(1+r^{-1}s)[\,x_i^-(0),\,y_{1\,i-1}^-(1)\,]=0$. Under the condition $r\neq -s$, it follows 
that
$$[\,x_i^-(0),\,y_{1\,i-1}^-(1)\,]=0.$$

Using the above result,  it yields from an immediate calculation that

\begin{eqnarray*}
&&[\,x_i^-(0),x_{\theta}^-(1)\,]\qquad {\hbox{(by the definition)}}\\
&=&[\,x_1^-(0),\,\ldots,\,x_{i-2}^-(0),\,\underbrace{[\,x_i^-(0),\,y_{1\,i-1}^-(1)\,]}\,]_{(r^{-1},\ldots,\,r^{-1})}
\\
&=&0.
\end{eqnarray*}

(III)  \ When $i=n$, $\langle n,\, 0\rangle=(rs)^{-1}$. Observe that,
\begin{equation*}
\begin{split}
&y_{1\,n-1}^-(1)\\
=&[\,x_{n-1}^-(0),\,x_n^-(0),\,x_n^-(0),\,
x_{1\,n-1}^-(1)\,]_{(s,\,(rs)^{\frac{1}{2}},r^{-1})}\quad {\hbox{(using (\ref{b:1}))}}\\
=&[\,[\,x_{n-1}^-(0),\,x_n^-(0)\,]_{r^{-1}},\,[\,x_n^-(0),\,
x_{1\,n-1}^-(1)\,]_s\,]_{(rs)^{\frac{1}{2}}}\quad{\hbox{(using (\ref{b:1}))}}\\
&+r^{-1}[\,x_{n}^-(0),\,\underbrace{[\,x_{n-1}^-(0),\,x_n^-(0),\,
x_{1\,n-1}^-(1)\,]_{(s,\,1)}}\,]_{r^{\frac{3}{2}}s^{\frac{1}{2}}}\quad {\hbox{(=0 by (\ref{l:12}))}}\\
=&[\,[\,[\,x_{n-1}^-(0),\,x_n^-(0)\,]_{r^{-1}},\,x_n^-(0)\,]_{(rs)^{-\frac{1}{2}}}
,\,x_{1\,n-1}^-(1)\,]_{rs^2}\\
&+(rs)^{-\frac{1}{2}}[\,x_n^-(0),\,\underbrace{[\,
[\,x_{n-1}^-(0),\,x_n^-(0)\,]_{r^{-1}},\,
x_{1\,n-1}^-(1)\,]_{rs}}\,]_{r^{\frac{1}{2}}s^{\frac{3}{2}}}
\quad{\hbox{(using (\ref{b:1}))}}\\
=&[\,[\,[\,x_{n-1}^-(0),\,x_n^-(0)\,]_{r^{-1}},\,x_n^-(0)\,]_{(rs)^{-\frac{1}{2}}}
,\,x_{1\,n-1}^-(1)\,]_{rs^2}\\
&+(rs)^{-\frac{1}{2}}[\,x_n^-(0),\,\underbrace{[\,x_{n-1}^-(0),\,[\,x_n^-(0),\,
x_{1\,n-1}^-(1)\,]_s\,]_1}\,]_{(r^{\frac{1}{2}}s^{\frac{3}{2}})}\quad{\hbox{(=0 by (\ref{l:12}))}}\\
&+(r^{-1}s)^{\frac{1}{2}}[\,x_n^-(0),\,
\underbrace{[\,x_{n-1}^-(0),\,x_{1\,n-1}^-(1)\,]_r}
,\,x_n^-(0)\,]_{((rs)^{-1},\,r^{\frac{1}{2}}s^{\frac{3}{2}})}\quad{\hbox{(=0 by (\ref{l:13}))}}\\
=&[\,[\,[\,x_{n-1}^-(0),\,x_n^-(0)\,]_{r^{-1}},\,x_n^-(0)\,]_{(rs)^{-\frac{1}{2}}}
,\,x_{1\,n-1}^-(1)\,]_{rs^2}
\end{split}
\end{equation*}
Applying the above result, one sees that
\begin{eqnarray*}
&&[\,x_n^-(0),\,y_{1\,n-1}^-(1)\,]_{s^2}\\
&=&[\,x_n^-(0),\,[\,[\,x_{n-1}^-(0),\,x_n^-(0)\,]_{r^{-1}}
,\,x_n^-(0)\,]_{(rs)^{-\frac{1}{2}}}
,\,x_{1\,n-1}^-(1)\,]_{(rs^2,\,s^2)}\quad{\hbox{(using (\ref{b:1}))}}\\
&=&[\,\underbrace{[\,x_n^-(0),\,[\,x_{n-1}^-(0),\,x_n^-(0)\,]_{r^{-1}}
,\,x_n^-(0)\,]_{((rs)^{-\frac{1}{2}},\,s)}}
,\,x_{1\,n-1}^-(1)\,]_{rs^3}\\
&&+s[\,[\,[\,x_{n-1}^-(0),\,x_n^-(0)\,]_{r^{-1}}
,\,x_n^-(0)\,]_{(rs)^{-\frac{1}{2}}}
,\,[\,x_n^-(0),\,x_{1\,n-1}^-(1)\,]_s\,]_{rs}\quad{\hbox{(using (\ref{b:1}))}}\\
&=& s[\,[\,x_{n-1}^-(0),\,x_n^-(0)\,]_{r^{-1}}
,\,[\,x_n^-(0),\,x_{1\,n}^-(1)\,]_{(rs)^{\frac{1}{2}}}\,]_{1}\quad{\hbox{(using (\ref{b:1}))}}\\
&&+ r^{\frac{1}{2}}s^{\frac{3}{2}}[\,[\,[\,x_{n-1}^-(0),\,x_n^-(0)\,]_{r^{-1}}
,\,x_{1\,n}^-(1)\,]_{(rs)^{\frac{1}{2}}}
,\,x_n^-(0)\,]_{(rs)^{-1}}\quad{\hbox{(using (\ref{b:1}))}}\\
&=& s[\,x_{n-1}^-(0),\,\underbrace{[\,x_n^-(0),\,
y_{1\,n}^-(1)\,]_r}\,]_{r^{-2}}\qquad{\hbox{(=0 by (\ref{l:14}))}}\\
&&+ rs[\,\underbrace{[\,x_{n-1}^-(0),\,y_{1\,n}^-(1)\,]_{r^{-1}}},\,
x_n^-(0)\,]_{r^{-2}}\qquad{\hbox{(by the definition)}}\\
&&+ r^{\frac{1}{2}}s^{\frac{3}{2}}[\,[\,x_{n-1}^-(0),\,
y_{1\,n-1}^-(1)\,]_{r^{-1}}
,\,x_{n}^-(0)\,]_{(rs)^{-1}}\qquad{\hbox{(by the definition)}}\\
&&+ rs^{2}[\,[\,\underbrace{[\,x_{n-1}^-(0),\,
x_{1\,n}^-(1)\,]_{1}}
,\,x_{n}^-(1)\,]_{r^{-\frac{3}{2}}s^{-\frac{1}{2}}},\,x_n^-(0)\,]_{(rs)^{-1}}
\quad{\hbox{(=0 by (\ref{l:12}))}}\\
&=& rs[\,y_{1\,n-1}^-(1),\,x_n^-(0)\,]_{r^{-2}}+
r^{\frac{1}{2}}s^{\frac{3}{2}}[\,y_{1\,n-1}^-(1),\,x_n^-(0)\,]_{(rs)^{-1}}
\end{eqnarray*}
The above equation means that
$$(1+r^{-\frac{1}{2}}s^{\frac{1}{2}}+r^{-1}s)
[\,x_n^-(0),\,y_{1\,n-1}^-(1)\,]_{rs}=0.$$
So, if $(r/s)^{3/2}\ne 1$,
 it follows that
\begin{equation*}[\,x_n^-(0),\,y_{1\,n-1}^-(1)\,]_{rs}=0.
\end{equation*}
By this result it follows from (\ref{b:1}) and the Serre relation that,
\begin{eqnarray*}
&&[\,x_n^-(0),\,x_{\theta}^-(1)\,]_{rs}\qquad{\hbox{(by definition)}}\\
&=&[\,x_n^-(0),\,[\,x_2^-(0),\ldots,\,x_{n-2}^-(0),y_{1\,n-1}^-(1)
\,]_{(r^{-1},\,\ldots,\,r^{-1})}\,]_{rs} \\
&=&[\,x_2^-(0),\ldots,\,x_{n-2}^-(0),\,\underbrace{
[\,x_n^-(0),\,y_{1\,n-1}^-(1)\,]_{rs}}
\,]_{(r^{-1},\,\ldots,\,r^{-1})}\,]_{rs}\\
&=&0.
\end{eqnarray*}
Thus Proposition \ref{p:2} has been proved.

Now, we are ready to check the commutation relation ($\mathcal{B}4$), that is,

\begin{prop} {\label{p:3}
$[\,E_0, F_0\,]=
\frac{\gamma'^{-1}\om_{\theta}^{-1}-\gamma^{-1}\om_{\theta'}^{-1}}{r-s}$.}
\end{prop}

\begin{proof}\, First we observe that
\begin{equation*}
\begin{split}
\bigl[\,E_0,
F_0\,\bigr]&=(rs)^{n-2}\bigl[\,x^-_{\theta}(1)\,
\gamma'^{-1}{\om_\theta}^{-1},\,\gamma^{-1}{\om'_\theta}^{-1}
x^+_{\theta}(-1)\,\bigr]\\
&=(rs)^{n-2}\bigl[\,x^-_{\theta}(1),\,x^+_{\theta}(-1)\,\bigr]\,
\cdot(\gamma^{-1}\gamma'^{-1}{\om_\theta}^{-1}{\om'_\theta}^{-1}).
\end{split}
\end{equation*}
Hence we have to compute the bracket $[\,x^-_{\theta}(1),\,x^+_{\theta}(-1)\,]$. Recalling the construction of $x^-_{\theta}(1)$ and $x^+_{\theta}(-1)$, we have by the inductive step that
\begin{gather*}
x^-_{\theta}(1)=y_{1\,2}^-(1)=[\,x_2^-(0),\ldots, x_n^-(0),\,x_{1\,n-1}^-(1)\,]
_{(s^2,\,r^{-1},\,\ldots,r^{-1})},\\
x_{\theta}^+(-1)=\tau(x^-_{\theta}(1))=y_{1\,2}^+(-1)\\
=[\,[\,[\,x_{1\,n-1}^+(-1),\,x_n^+(0)\,]_{r^2},\,
x_{n-1}^+(0)\,]_{s^{-1}},\ldots, x_2^+(0)\,]_{s^{-1}}.
\end{gather*}
As a consequence of the case $A_{n-1}^{(1)}$ \cite{HRZ} it follows that
\begin{equation*}[\,x_{1\,n-1}^-(1),\,x_{1\,n-1}^+(-1)\,]=
\frac{\gamma\om_{\alpha_{1\,n-1}}'-\gamma'\om_{\alpha_{1\,n-1}}}{r-s}.
  \end{equation*}

Next, we consider
\begin{equation*}
\begin{split}
&[\,x_{1\,n}^-(1),\,x_{1\,n}^+(-1)\,]\qquad{\hbox{(by definition)}}\\
=&[\,[\,x_{n}^-(0),\,x_{1\,n-1}^-(1)\,]_{s}
,\,[\,x_{1\,n-1}^+(-1),\,x_{n}^+(0)\,]_{r}\,]\qquad{\hbox{(using (\ref{b:1}))}}\\
=&[\,[\,[\,x_{n}^-(0),\,x_{1\,n-1}^+(-1)\,],\,x_{1\,n-1}^-(1)\,]_{s}
,\,x_{n}^+(0)\,]_{r}\quad{\hbox{(=0 by (\ref{b:1}) and (D9))}}\\
&+[\,[\,x_{n}^-(0),\,[\,x_{1\,n-1}^-(1),\,x_{1\,n-1}^+(-1)\,]\,]_{s}
,\,x_{n}^+(0)\,]_{r}\quad{\hbox{(using (\ref{b:1}), (D6) and (D9))}}\\
&+[\,x_{1\,n-1}^+(-1),\,[\,
[\,x_{n}^-(0),\,x_{n}^+(0)\,],\,x_{1\,n-1}^-(1)\,]_{s}\,]_{r}\quad{\hbox{(using (D9), (D6) and (\ref{b:3}))}}\\
&+[\,x_{1\,n-1}^+(-1),\,[\,x_{n}^-(0),\,
[\,x_{1\,n-1}^-(1),\,x_{n}^+(0)\,]\,]_{s}
\,]_{r}\quad{\hbox{(=0 by (\ref{b:1}) and (D9))}}\\
=&\gamma\om'_{\alpha_{1\,n-1}}\cdot \frac{\om'_n-\om_n}{r^{\frac{1}{2}}-s^{\frac{1}{2}}}
+\frac{\gamma\om'_{\alpha_{1\,n-1}}-\gamma'\om_{\alpha_{1\,n-1}}}
{r^{\frac{1}{2}}-s^{\frac{1}{2}}}\om_n\\
=&\frac{\gamma\om'_{\alpha_{1\, n}}-\gamma'\om_{\alpha_{1\,n}}}{r^{\frac{1}{2}}-s^{\frac{1}{2}}}.
\end{split}
\end{equation*}

Repeating the above steps, we also get that
\begin{equation*}
\begin{split}
&[\,y_{1\,n}^-(1),\,y_{1\,n}^+(-1)\,]\quad{\hbox{(by definition)}}\\
=&[\,[\,x_{n}^-(0),\,x_{1\,n}^-(1)\,]_{(rs)^{\frac{1}{2}}}
,\,[\,x_{1\,n}^+(-1),\,x_{n}^+(0)\,]_{(rs)^{\frac{1}{2}}}\,]
\quad{\hbox{(using (\ref{b:1}))}}\\
=&[\,[\,[\,x_{n}^-(0),\,x_{1\,n}^+(-1)\,],\,x_{1\,n}^-(1)\,]_{(rs)^{\frac{1}{2}}}
,\,x_{n}^+(0)\,]_{(rs)^{\frac{1}{2}}}\,{\hbox{(using (\ref{b:1}),(D9) and (\ref{b:3}))}}\\
&+[\,[\,x_{n}^-(0),\,[\,x_{1\,n}^-(1),\,x_{1\,n}^+(-1)\,]\,]_{(rs)^{\frac{1}{2}}}
,\,x_{n}^+(0)\,]_{(rs)^{\frac{1}{2}}}\,{\hbox{(=0 by the above result and (D6))}}\\
&+[\,x_{1\,n}^+(-1),\,[\,
[\,x_{n}^-(0),\,x_{n}^+(0)\,],\,x_{1\,n}^-(1)\,]_{(rs)^{\frac{1}{2}}}
\,]_{(rs)^{\frac{1}{2}}}\,{\hbox{(=0 by (D9) and (D6))}}\\
&+[\,x_{1\,n}^+(-1),\,[\,x_{n}^-(0),\,
[\,x_{1\,n}^-(1),\,x_{n}^+(0)\,]\,]_{(rs)^{\frac{1}{2}}}
\,]_{(rs)^{\frac{1}{2}}}\,{\hbox{(using (\ref{b:1}), (D9) and (\ref{b:3}))}}\\
=&\frac{r-s}{r^{\frac{1}{2}}-s^{\frac{1}{2}}}\gamma\om'_{\alpha_{1\,n}}\cdot \frac{\om'_{n}-\om_{n}}{r^{\frac{1}{2}}-s^{\frac{1}{2}}}
+\frac{r-s}{r^{\frac{1}{2}}-s^{\frac{1}{2}}}\frac{\gamma\om'_{\alpha_{1\,n}}-\gamma'\om_{\alpha_{1\,n}}}
{r^{\frac{1}{2}}-s^{\frac{1}{2}}}\om_{n}\\
=&\frac{r-s}{r^{\frac{1}{2}}-s^{\frac{1}{2}}}\frac{\gamma\om'_{\beta_{1\,n}}-\gamma'\om_{\beta_{1\,n}}}
{r^{\frac{1}{2}}-s^{\frac{1}{2}}}.
\end{split}
\end{equation*}

Now we arrive at
\begin{equation*}
\begin{split}
&[\,y_{1\,n-1}^-(1),\,y_{1\,n-1}^+(-1)\,]\quad{\hbox{(by definition)}}\\
=&[\,[\,x_{n-1}^-(0),\,y_{1\,n}^-(1)\,]_{r^{-1}}
,\,[\,y_{1\,n}^+(-1),\,x_{n-1}^+(0)\,]_{s^{-1}}\,]
\quad{\hbox{(using (\ref{b:1}))}}\\
=&[\,[\,[\,x_{n-1}^-(0),\,y_{1\,n}^+(-1)\,],\,y_{1\,n}^-(1)\,]_{r^{-1}}
,\,x_{n}^+(0)\,]_{s^{-1}}\,{\hbox{(=0 by using (\ref{b:1}) and (D9))}}\\
&+[\,[\,x_{n-1}^-(0),\,[\,y_{1\,n}^-(1),\,y_{1\,n}^+(-1)\,]\,]_{r^{-1}}
,\,x_{n-1}^+(0)\,]_{s^{-1}}\,{\hbox{(using the above result and (D6)))}}\\
&+[\,y_{1\,n}^+(-1),\,[\,
[\,x_{n-1}^-(0),\,x_{n-1}^+(0)\,],\,y_{1\,n}^-(1)\,]_{r^{-1}}
\,]_{s^{-1}}\,{\hbox{(using by (D9) and (D6))}}\\
&+[\,y_{1\,n}^+(-1),\,[\,x_{n-1}^-(0),\,
[\,y_{1\,n}^-(1),\,x_{n-1}^+(0)\,]\,]_{r^{-1}}
\,]_{s^{-1}}\,{\hbox{(=0 by (D9) and (\ref{b:3}))}}\\
=&(rs)^{-1}(\frac{r-s}{r^{\frac{1}{2}}-s^{\frac{1}{2}}})^2\gamma\om'_{\alpha_{1\,n}}\cdot \frac{\om'_{n-1}-\om_{n-1}}{r-s}
+(rs)^{-1}(\frac{r-s}{r^{\frac{1}{2}}-s^{\frac{1}{2}}})^2
\frac{\gamma\om'_{\alpha_{1\,n}}-\gamma'\om_{\alpha_{1\,n}}}{r-s}\om_{n-1}\\
=&(rs)^{-1}(\frac{r-s}{r^{\frac{1}{2}}-s^{\frac{1}{2}}})^2
\frac{\gamma\om'_{\beta_{1\,n-1}}-\gamma'\om_{\beta_{1\,n-1}}}
{r-s}.
\end{split}
\end{equation*}

At last, we get the following identity:
\begin{equation*}
\begin{split}
[\,y_{1\,1}^-(1),\,y_{1\,1}^+(-1)\,]
=(rs)^{2-n}[2]_n^2
\frac{\gamma\om'_{\beta_{1\,1}}-\gamma'\om_{\beta_{1\,1}}}
{r-s}.
\end{split}
\end{equation*}

So we have completed the proof
of Proposition \ref{p:3}.
\end{proof}

We now proceed to check relation ($\mathcal{B}5$).
It suffices to verify two key Serre relations, others are similar.

\begin{lemm} The following relations yield the Serre relations

 $(1)$ \  $[\,E_1,\,E_1,\,E_1,\, E_0\,]_{(r^{-1}s^{-2},\,(rs)^{-2},\,r^{-2}s^{-1})}=0$,

$(2)$ \ $[\,E_{n},\,E_{n},\, E_{n},\, E_{n-1}\,]_{(r^{-1},\,(rs)^{-\frac{1}{2}},\,s^{-1})}=0$,

\end{lemm}

\begin{proof} \ (1) \  First we consider
\begin{equation*}
\begin{split}
&\bigl[\,x_1^+(0),\,x_{\theta}^-(1)\bigr]_{1}\\
=&\bigl[\,[\,x_1^+(0),\, x_1^-(0)\,],\,y_{1\,2}^-(1)\,\bigr]_{(rs)^{-1}}\quad{\hbox{(=0 by (D9) and (D6))}} \\
&+\bigl[x_1^-(0),\ldots,x_{n}^-(0),\,x_n^-(0),\ldots,x_2^-(0), [x_1^+(0),x_1^-(1)]\bigr]_{(s,\ldots, s, (rs)^{\frac{1}{2}}, r^{-1}, \ldots, r^{-1}, (rs)^{-1})}\\
=&-(\gamma\gamma')^{-\frac{1}{2}}y_{2\,1}^-(1)\omega_1,
\end{split}
\end{equation*}
where $$y_{2\,1}^-(1)=[\,x_1^-(0),\,\ldots,\, x_n^-(0),\,x_n^-(0),\,\ldots x_3^-(0),\,x_2^-(1)\,]_{(s,\,\ldots,\, s,\,(rs)^{\frac{1}{2}},\, r^{-1},\,\ldots,\,r^{-1},\,1,\,r^{-2})}.$$

Applying the above result, we have that
\begin{equation*}
\begin{split}
&\bigl[\,x_1^+(0),\,x_1^+(0),\,x_{\theta}^-(1),\bigr]_{(1,\,r^{-1}s)}\\
=&-(\gamma\gamma')^{-\frac{1}{2}}
\bigl[\,[\,x_1^+(0),\, x_1^-(0)\,],\,y_{2\,2}^-(1)\,\bigr]_{r^{-2}}\omega_1\quad{\hbox{(using (D9) and (D6))}} \\
=&-(\gamma\gamma')^{-\frac{1}{2}}(rs)^{-2}(r+s)y_{2\,2}^-(1)\omega_1^2,
\end{split}
\end{equation*}
where $$y_{2\,2}^-(1)=[\,x_2^-(0),\,\ldots,\, x_n^-(0),\,x_n^-(0),\,\ldots x_3^-(0),\,x_2^-(1)\,]_{(s,\,\ldots,\, s,\,(rs)^{\frac{1}{2}},\, r^{-1},\,\ldots,\,r^{-1},\,1)}.$$

As a consequence of these results, it follows that

\begin{equation*}
\begin{split}
&[\,E_1,\,E_1,\,E_1,\, E_0\,]_{(r^{-1}s^{-2},\,(rs)^{-2},\,r^{-2}s^{-1})}\\
=&a\,[\,x_1^+(0),\,x_1^+(0),\,x_1^+(0),\, x_{\theta}^-(1)\,]_{(1,\,r^{-1}s,\,r^{-2}s^{2})}{\gamma'}^{-1}\omega_{\theta}^{-1}\\
=&-a\,(\gamma\gamma')^{-\frac{1}{2}}(rs)^{-2}(r+s)\underbrace{[\,x_1^+(0),\,y_{2\,2}^-(1)\,]}\omega_1^2=0
\qquad{\hbox{(by (D9))}}
\end{split}
\end{equation*}

(2) \  Denote $Y=[\,E_{n},\,E_{n},\, E_{n},\, E_{n-1}\,]_{(r^{-1},\,(rs)^{-\frac{1}{2}},\,s^{-1})}$.
It is clear that $Y\in \mathcal{U}_{r,s}(\frak{g}^{\sigma})^+$. By Lemma \ref{lemma1}, in order to prove $Y=0$, it suffices to check $[\,x_i^-(0),\, Y\,]=0$ for $i\in I$, which is trivial for $i$ not equal to $n-1$ or $n$.

In the case of $i=n-1$, one gets that
\begin{equation*}
\begin{split}
&[\,x_{n-1}^+(0),\, Y\,] \\
=&[\,x_{n}^+(0),\,x_{n}^+(0),\, x_{n}^+(0),\, [\,x_{n-1}^-(0),\,x_{n-1}^+(0)\,]\,]_{(r^{-1},\,(rs)^{-\frac{1}{2}},\,s^{-1})}\,{\hbox{(using (D9) and (D6))}}\\
=&-r^{-1}s^{-3}\omega_{n-1} [\,x_{n}^+(0),\,x_{n}^+(0),\, x_{n}^+(0)]_{(r^{-\frac{1}{2}}s^{\frac{1}{2}},\,1)}\\
=&0.
\end{split}
\end{equation*}

In the case of $i=n$, it is easy to see that
\begin{equation*}
\begin{split}
&[\,x_{n}^-(0),\, Y\,] \\
=&[\,x_{n}^-(0),\,[\,x_{n}^+(0),\,x_{n}^+(0),\, x_{n}^+(0),\, x_{n-1}^+(0)\,]_{(r^{-1},\,(rs)^{-\frac{1}{2}},\,s^{-1})}\,] \\
=&[\,[\,x_{n}^-(0),\,x_{n}^+(0)\,],\,x_{n}^+(0),\, x_{n}^+(0),\, x_{n-1}^+(0)\,]_{(r^{-1},\,(rs)^{-\frac{1}{2}},\,s^{-1})}\, {\hbox{(using (D9) and (D6))}}\\
+&[\,x_{n}^+(0),\,[\,x_{n}^-(0),\,x_{n}^+(0)\,],\, x_{n}^+(0),\, x_{n-1}^+(0)\,]_{(r^{-1},\,(rs)^{-\frac{1}{2}},\,s^{-1})}\,{\hbox{(=0 by (D9) and (D6))}}\\
+&[\,x_{n}^+(0),\,x_{n}^+(0),\, [\,x_{n}^-(0),\,x_{n}^+(0)\,],\, x_{n-1}^+(0)\,]_{(r^{-1},\,(rs)^{-\frac{1}{2}},\,s^{-1})}\,{\hbox{(using (D9) and (D6))}} \\
=&-(rs)^{-1}[2]_n[\,x_{n}^+(0),\, x_{n}^+(0),\,x_{n-1}^+(0)\,]_{(r^{-1},\,(rs)^{-\frac{1}{2}})}\omega'_{n}\\
&+(rs)^{-1}[2]_n[\,x_{n}^+(0),\, x_{n}^+(0),\,x_{n-1}^+(0)\,]_{(r^{-1},\,(rs)^{-\frac{1}{2}})}\omega'_{n}\\
=&0.
\end{split}
\end{equation*}

\end{proof}

\subsection{Proof of Theorem $\mathcal{A}$ for the case of $U_{r,s}(\mathrm{D}_{n+1}^{(2)})$}

We now proceed to check those relations of $(\mathcal{C}4)$--$(\mathcal{C}6)$ involving with $i=0$.

In order to verify $(\mathcal{C}4)$, it is enough to check the following proposition as before

\begin{prop} \label{p:5} $\bigl[\,x_i^-{(0)},\,
 x^-_{\theta}(1)\,\bigr]_{\langle i,~
0 \rangle^{-1}}=0$,\quad for \ $i\in I$.
\end{prop}

The proof uses the following crucial and technical lemma.

\begin{lemm}\, \label{l:16} $[x_i^-(0),x_{i-1}^-(0),x_{i}^-(0),x_{i+1}^-(0)]_{(r^{-2},r^{-2},1)}=0$ for $i=2,\ldots, n-1$.
\end{lemm}
\begin{proof}\, Combining \ref{b:1} with the Serre relations, one gets
 \begin{eqnarray*}
&&[\,x_i^-(0),\,[\,x_{i-1}^-(0),\,x_i^-(0),\,x_{i+1}^-(0)\,]_{(r^{-2},\,r^{-2})}
\,]_{r^{-1}s}
\qquad {\hbox{(using (\ref{b:1}))}}\\
&=&[\,x_i^-(0),\,[\,x_{i-1}^-(0),\,x_i^-(0)\,]_{r^{-2}},\,x_{i+1}^-(0)
\,]_{(r^{-2},\,r^{-2}s^2)}
\qquad {\hbox{(using (\ref{b:1}))}}\\
&&+r^{-2}[\,x_i^-(0),\,x_i^-(0),\,\underbrace{[\,x_{i-1}^-(0),\,x_{i+1}^-(0)\,]}
\,]_{(1,\,r^{-2}s^2)}
\qquad {\hbox{(=0 by Serre relation)}}\\
&=&[\,\underbrace{[\,x_i^-(0),\,[\,x_{i-1}^-(0),\,x_i^-(0)\,]_{r^{-2}}\,]_{s^2}}
,\,x_{i+1}^-(0)\,]_{r^{-4}}
\qquad {\hbox{(=0 by Serre relation)}}\\
&&+s^2[\,[\,x_{i-1}^-(0),\,x_i^-(0)\,]_{r^{-2}}
,\,[\,x_{i}^-(0),\,x_{i+1}^-(0)\,]_{r^{-2}}
\,]_{(rs)^{-2}}\qquad {\hbox{(using (\ref{b:2}))}}\\
&=&s^2[\,x_{i-1}^-(0),\,
\underbrace{[\,x_i^-(0),\,x_{i}^-(0),\,x_{i+1}^-(0)\,]_{(r^{-2},s^{-2})}}
\,]_{r^{-4}}\qquad {\hbox{(=0 by Serre relation)}}\\
&&+[\,[\,x_{i-1}^-(0),\,x_i^-(0),\,x_{i+1}^-(0)\,]_{(r^{-2},r^{-2})}
,\,x_{i}^-(0)\,]_{r^{-2}s^2},
\end{eqnarray*}
which implies that $(1+r^{-2}s^2)[\,x_i^-(0),\,
[\,x_{i-1}^-(0),\,x_i^-(0),\,x_{i+1}^-(0)\,]_{(r{-2},\,r^{-2})}\,]=0.$
Thus, when $r\ne -s$, it follows that $$[\,x_i^-(0),\,
[\,x_{i-1}^-(0),\,x_i^-(0),\,x_{i+1}^-(0)\,]_{(r{-2},\,r^{-2})}\,]=0.$$
\end{proof}

\noindent{\bf Proof of Proposition \ref{p:5}} \ (I) \ When $i=1$,
$\langle 1,\, 0\rangle=s^2$, one has,
\begin{eqnarray*}
&&[\,x_1^-(0),x_{\theta}^-(1)\,]_{s^{-2}}\qquad {\hbox{(by definition)}}\\
&=&[\,x_1^-(0),\underbrace{
[\,x_1^-(0),\,x_2^-(0),\,x_{n\,3}^-(1)\,]_{(r^{-2},\,r^{-2})}}
\,]_{s^{-2}}\qquad {\hbox{(using (\ref{b:1}))}}\\
&=&[\,x_1^-(0),[\,[\,x_1^-(0),\,x_2^-(0)\,]_{r^{-2}}
,\,x_{n\,3}^-(1)\,]_{r^{-2}}
\,]_{s^{-2}}\qquad {\hbox{(using (\ref{b:1}))}}\\
&&+r^{-2}[\,x_1^-(0),\,x_2^-(0),\,\underbrace{[\,x_1^-(0),\,x_{n\,3}^-(1)\,]}
\,]_{(1,\,s^{-2})}\qquad {\hbox{(=0 by Serre relation)}}\\
&=&[\,\underbrace{[\,x_1^-(0),[\,x_1^-(0),\,x_2^-(0)\,]_{r^{-2}}\,]_{s^{-2}}}
,\,x_{n\,3}^-(1)\,]_{r^{-2}}\quad {\hbox{(=0 by Serre relation)}}\\
&&+s^{-2}[\,[\,x_1^-(0),\,x_2^-(0)\,]_{r^{-2}}
,\,\underbrace{[\,x_1^-(0),\,x_{n\,3}^-(1)\,]}
\,]_{r^{-2}s^2}\quad {\hbox{(=0 by Serre relation)}}\\
&=&0.
\end{eqnarray*}

\ (II) \   When $1<i<n$, $\langle i,\, 0\rangle=1$.
Thanks to Lemma {\ref{l:16}}, it is easy to see that,
\begin{eqnarray*}
&&[\,x_i^-(0),x_{\theta}^-(1)\,]\qquad {\hbox{(by definition)}}\\
&=&[x_1^-(0),\ldots,
\underbrace{[x_i^-(0), x_{i-1}^-(0), x_{i}^-(0), x_{i+1}^-(0)]_{(r{-2},r^{-2},1)}}
, x_{n\,i+2}^-(1)]_{(r^{-2},\ldots, r^{-2})}
\\
&=&0.
\end{eqnarray*}

\medskip

Next we can show the communication relation ($\mathcal{C}4$).

\begin{prop} {\label{p:6}
$[\,E_0, F_0\,]=
\frac{\gamma'^{-1}\om_{\theta}^{-1}-\gamma^{-1}\om_{\theta'}^{-1}}{r-s}$.}
\end{prop}

\begin{proof}\, First, using
relation $(\textrm{D6})$ and induction, we see that
\begin{equation*}
\begin{split}
\bigl[\,E_0,
F_0\,\bigr]&=(rs)^{2(n-1)}\bigl[\,x^-_{\theta}(1)\,
\gamma'^{-1}{\om_\theta}^{-1},\,\gamma^{-1}{\om'_\theta}^{-1}
x^+_{\theta}(-1)\,\bigr]\\
&=(rs)^{2(n-1)}\bigl[\,x^-_{\theta}(1),\,x^+_{\theta}(-1)\,\bigr]\,
\cdot(\gamma^{-1}\gamma'^{-1}{\om_\theta}^{-1}{\om'_\theta}^{-1}).
\end{split}
\end{equation*}
Recalling the notations, we have that
\begin{gather*}
x^-_{\theta}(1)=x_{n\,1}^-(1)=[\,x_1^-(0),\ldots, x_{n-1}^-(0),\,x_{n}^-(1)\,]
_{(r^{-2},\,\ldots,r^{-2})},\\
x_{\theta}^+(-1)=\tau(x^-_{\theta}(1))=x_{n\,1}^+(-1)\\
=[\,x_{n}^+(-1),\,x_{n-1}^+(0),\,
\ldots, x_1^+(0)\,]_{\langle s^{-2},\,\ldots, s^{-2}\rangle}.
\end{gather*}

The first step of induction is to check the following
\begin{equation*}
\begin{split}
&[\,x_{n\,n-1}^-(1),\,x_{n\,n-1}^+(-1)\,]\qquad{\hbox{(by definition)}}\\
=&[\,[\,x_{n-1}^-(0),\,x_{n}^-(1)\,]_{r^{-2}}
,\,[\,x_{n}^+(-1),\,x_{n-1}^+(0)\,]_{s^{-2}}\,]\qquad{\hbox{(using (\ref{b:1}))}}\\
=&[\,[\,x_{n-1}^-(0),\,[\,x_{n}^-(1),x_{n}^+(-1)\,]\,]_{r^{-2}}
,\,x_{n-1}^+(0)\,]_{s^{-2}}\qquad{\hbox{(using (\ref{b:1}))}}\\
&+[\,x_{n}^+(-1),\,[\,[\,x_{n-1}^-(0),\,x_{n-1}^+(0)\,],\,x_{n}^-(1)\,]_{r^{-2}}
\,]_{s^{-2}}\,]\quad{\hbox{(using (\ref{b:1}))}}\\
=&(rs)^{-2}\gamma\om'_{n}\cdot \frac{\om'_{n-1}-\om_{n-1}}{r-s}
+(rs)^{-2}\frac{\gamma\om'_{n}-\gamma'\om_{n}}
{r-s}\om_{n-1}\\
=&(rs)^{-2}\frac{\gamma\om'_{\alpha_{n\, n-1}}-\gamma'\om_{\alpha_{n\,n-1}}}{r-s}.
\end{split}
\end{equation*}

Repeating the above steps, we obtain that
\begin{equation*}
\begin{split}
&[\,x_{n\,n-2}^-(1),\,x_{n\,n-2}^+(-1)\,]\qquad{\hbox{(by definition)}}\\
=&[\,[\,x_{n-2}^-(0),\,x_{n\,n-1}^-(1)\,]_{r^{-2}}
,\,[\,x_{n\,n-1}^+(-1),\,x_{n-2}^+(0)\,]_{s^{-2}}\,]\qquad{\hbox{(using (\ref{b:1}))}}\\
=&[\,[\,x_{n-2}^-(0),\,[\,x_{n\,n-1}^-(1),x_{n\,n-1}^+(-1)\,]\,]_{r^{-2}}
,\,x_{n-2}^+(0)\,]_{s^{-2}}\qquad{\hbox{(using (\ref{b:1}))}}\\
&+[\,x_{n\,n-1}^+(-1),\,[\,[\,x_{n-2}^-(0),\,x_{n-2}^+(0)\,],\,x_{n\,n-1}^-(1)\,]_{r^{-2}}
\,]_{s^{-2}}\,]\quad{\hbox{(using (\ref{b:1}))}}\\
=&(rs)^{-4}\gamma\om'_{\alpha_{n\, n-1}}\cdot \frac{\om'_{n-2}-\om_{n-2}}{r-s}
+(rs)^{-4}\frac{\gamma\om'_{\alpha_{n\, n-1}}-\gamma'\om_{\alpha_{n\, n-1}}}
{r-s}\om_{n-2}\\
=&(rs)^{-4}\frac{\gamma\om'_{\alpha_{n\, n-2}}-\gamma'\om_{\alpha_{n\,n-2}}}{r-s}.
\end{split}
\end{equation*}

By induction we arrive at the following
\begin{equation*}
\begin{split}
[\,x_{n\,1}^-(1),\,x_{n\,1}^+(-1)\,]
=(rs)^{-2(n-1)}\frac{\gamma\om'_{\alpha_{n\,1}}-\gamma'\om_{\alpha_{n\,1}}}
{r-s},
\end{split}
\end{equation*}
which completes the proof of Proposition \ref{p:6}.
\end{proof}

The last part of this subsection is to check Serre relation ($\mathcal{C}5$).
Here we only verify the following two key Serre relations. Other Serre relations are checked similarly.

\begin{lemm} We have that
$(1)$ \  $[\,E_1,\,E_1,\, E_0\,]_{(r^{-2},\,s^{-2})}=0$,

$(2)$ \ $[\,E_{n},\,E_{n},\, E_{n},\, E_{n-1}\,]_{(r^{-2},\,(rs)^{-1},\,s^{-2})}=0$.
\end{lemm}

\begin{proof} \ (1) \  Before checking the first one, we need the relation
\begin{equation*}
\begin{split}
&\bigl[\,x_1^+(0),\,x_{\theta}^-(1)\bigr]_{1}\\
=&\bigl[\,[\,x_1^+(0),\, x_1^-(0)\,],\,y_{n\,2}^-(1)\,\bigr]_{r^{-2}}\,{\hbox{(using (D9) and (D6))}} \\
=&(rs)^{-2}y_{n\,2}^-(1)\omega_1.
\end{split}
\end{equation*}

Consequently, it follows that
\begin{equation*}
\begin{split}
&[\,E_1,\,E_1,\, E_0\,]_{(r^{-2},\,s^{-2})}\\
=&a\bigl[\,x_1^+(0),\,x_1^+(0),\,x_{\theta}^-(1),\bigr]_{(1,\,r^{-2}s^2)}{\gamma'}^{-1}\omega_{\theta}^{-1}\\
=&a(rs)^{-2}
\underbrace{\bigl[\,x_1^+(0),\, y_{n\,2}^-(1)\,\bigr]}\omega_1{\gamma'}^{-1}\omega_{\theta}^{-1}\,{\hbox{(=0 by (D9))}} \\
=&0.
\end{split}
\end{equation*}

(2) \ By using Lemma \ref{lemma1}, one can check the second result. Set $$Z=[\,E_{n},\,E_{n},\, E_{n},\, E_{n-1}\,]_{(r^{-2},\,(rs)^{-1},\,s^{-2})}.$$
It is clear that $Z\in \mathcal{U}_{r,s}(\frak{g}^{\sigma})^+$. In order to prove $Z=0$, it suffices to check $[\,x_i^-(0),\, Z\,]=0$ for $i\in I$,  which is trivial for $i$ not equal to $n-1$ or $n$.

As for the case of $i=n-1$,
\begin{equation*}
\begin{split}
&[\,x_{n-1}^+(0),\, Z\,] \\
=&[\,x_{n}^+(0),\,x_{n}^+(0),\, x_{n}^+(0),\, [\,x_{n-1}^-(0),\,x_{n-1}^+(0)\,]\,]_{(r^{-2},\,(rs)^{-1},\,s^{-2})}\\
=&-r^{-2}s^{-6}\omega_{n-1} [\,x_{n}^+(0),\,x_{n}^+(0),\, x_{n}^+(0)]_{(r^{-1}s,\,1)}\\
=&0.
\end{split}
\end{equation*}

In the case of $i=n$, one computes directly that
\begin{align*}
&[\,x_{n}^-(0),\, Z\,] \\
=&[\,x_{n}^-(0),\,[\,x_{n}^+(0),\,x_{n}^+(0),\, x_{n}^+(0)
,\, x_{n-1}^+(0)\,]_{(r^{-2},\,(rs)^{-1},\,s^{-2})}\,] \\
=&[\,[\,x_{n}^-(0),\,x_{n}^+(0)\,],\,x_{n}^+(0),\, x_{n}^+(0),\, x_{n-1}^+(0)\,]_{(r^{-2},\,(rs)^{-1},\,s^{-2})}\, {\hbox{(using (D9) and (D6))}}\\
+&[\,x_{n}^+(0),\,[\,x_{n}^-(0),\,x_{n}^+(0)\,],\, x_{n}^+(0),\, x_{n-1}^+(0)\,]_{(r^{-2},\,(rs)^{-1},\,s^{-2})}\,{\hbox{(=0 by (D9) and (D6))}}\\
+&[\,x_{n}^+(0),\,x_{n}^+(0),\, [\,x_{n}^-(0),\,x_{n}^+(0)\,],\, x_{n-1}^+(0)\,]_{(r^{-2},\,(rs)^{-1},\,s^{-2})}\,{\hbox{(using (D9) and (D6))}} \\
=&-(rs)^{-2}(r+s)[\,x_{n}^+(0),\, x_{n}^+(0),\,x_{n-1}^+(0)\,]_{(r^{-2},\,(rs)^{-1})}\omega'_{n}\\
&+(rs)^{-1}(r+s)[\,x_{n}^+(0),\, x_{n}^+(0),\,x_{n-1}^+(0)\,]_{(r^{-2},\,(rs)^{-1})}\omega'_{n}\\
=&0.
\end{align*}
\end{proof}

\subsection{Proof of Theorem $\mathcal{A}$ for $U_{r,s}(\mathrm{D}_{4}^{(3)})$}

The aim of this subsection is to check relations $(\mathcal{D}4)$--$(\mathcal{D}6)$
involving with $i=0$ for the case of $\mathrm{D}_{4}^{(3)}$.

Similar to the above cases,  the following proposition 
implies relation $(\mathcal{D}4)$.

Recall the notation defined in subsection 5.2:
\begin{equation*}
\begin{split}
x_{\theta}^-(1)&=
[\,x_{1}^-(0),\, x_2^-(0),\,x_1^-(1)\,]_{(s^3,\,r^{-2}s^{-1})}.
\end{split}
\end{equation*}

\begin{prop} \label{p:7} $\bigl[\,x_i^-{(0)},\,
 x^-_{\theta}(1)\,\bigr]_{\langle i,~
0 \rangle^{-1}}=0$,\quad for \ $i=1, 2$.
\end{prop}

\begin{proof} \ (I) \ When $i=1$,
$\langle 1,\, 0\rangle=rs^2$. Repeatedly using (\ref{b:1}), we have
\begin{equation*}
\begin{split}
&\bigl[\,x_1^-(0),\, x^-_{\theta}(1)\,\bigr]_{r^{-1}s^{-2}}\\
=&\bigl[\,x_1^-(0),\underbrace{\bigl[\,x_1^-(0),\,x_2^-(0),\,x^-_{1}(1)\,\bigr]_{(s^3,\,r^{-2}s^{-1})}}
\bigr]_{r^{-1}s^{-2}}\,( \hbox{using (\ref{b:1})})\\
=&\bigl[\,x_1^-(0),\bigl[\,[\,x_1^-(0),\,x_2^-(0)\,]_{r^{-3}},\,x^-_{1}(1)\,\bigr]_{rs^{2}}\bigr]_{r^{-1}s^{-2}}
\,( \hbox{using (\ref{b:1})})\\
&+r^{-3}\bigl[\,x_1^-(0),\,\bigl[\,x_2^-(0),\,\underbrace{[\,x_1^-(0),\,x^-_{1}(1)\,]_{rs^{-1}}}\,\bigr]_{(rs)^{3}}
\bigr]_{r^{-1}s^{-2}}
\,( \hbox{=0 by the Serre relation})\\
=&\bigl[\,\bigl[\,x_1^-(0),\,x_1^-(0),\,x_2^-(0)\,\bigr]_{(r^{-3},\,r^{-2}s^{-1})},\,x^-_{1}(1)\,\bigr]_{r^{2}s}\\
&+r^{-2}s^{-1}\bigl[\,[\,x_1^-(0),\,x_2^-(0)\,]_{r^{-3}},\,
\underbrace{[\,x_1^-(0),x^-_{1}(1)\,]_{rs^{-1}}}\,\bigr]_{(rs)^{3}}\,( \hbox{=0 by (D8)})\\
=&-r^{2}s\bigl[\,x^-_{1}(1),\,\bigl[\,x_1^-(0),\,x_1^-(0),\,x_2^-(0)\,\bigr]_{(r^{-3},\,r^{-2}s^{-1})}
\,\bigr]_{r^{-2}s^{-1}}\,( \hbox{using (\ref{b:1})})\\
=&-r^{2}s\bigl[\,\underbrace{[\,x^-_{1}(1),\,x_1^-(0)\,]_{r^{-1}s}},\,[\,x_1^-(0),\,x_2^-(0)\,]_{r^{-3}}
\,\bigr]_{(rs)^{-3}}\,( \hbox{=0 by (D8)})\\
&-rs^{2}\bigl[\,x^-_{1}(0),\,\underbrace{[\,x_1^-(1),\,x_1^-(0),\,x_2^-(0)\,]_{(r^{-3},\,r^{-1}s^{-2})}}
\,\bigr]_{r^{-1}s^{-2}}\,( \hbox{using (\ref{b:1})})\\
=&-rs^{2}\bigl[\,x^-_{1}(0),\,[\,\underbrace{[\,x_1^-(1),\,x_1^-(0)\,]_{r^{-1}s}},\,x_2^-(0)\,]_{(rs)^{-3}}
\,\bigr]_{r^{-1}s^{-2}}\,( \hbox{=0 by (D8)})\\
&-rs^2\bigl[\,x^-_{1}(0),\,[\,x_1^-(0),\,[\,x_1^-(1),\,x_2^-(0)\,]_{s^{-3}}\,]_{r^{-2}s^{-1}}
\,\bigr]_{r^{-1}s^{-2}}\\
=&rs^{-1}\bigl[\,x_1^-(0),\bigl[\,x_1^-(0),\,x_2^-(0),\,x^-_{1}(1)\,\bigr]_{(s^3,\,r^{-2}s^{-1})}
\bigr]_{r^{-1}s^{-2}}
\end{split}
\end{equation*}
Thus if $r \neq s$, it is easy to see that
$\bigl[\,x_1^-(0),\, x^-_{\theta}(1)\,\bigr]_{r^{-1}s^{-2}}=0$.
\smallskip

(II) \ When $i=2$, $\langle 2,\, 0\rangle=(rs)^{-3}$. Using (\ref{b:1}) and the Serre relation, we get that
\begin{equation*}
\begin{split}
&\bigl[\,x_{2}^-(0),\, x^-_{\theta}(1)\,\bigr]_{s^6}\quad
 \hbox{(by definition)} \\
=&\bigl[\,x_{2}^-(0), \,\underbrace{\bigl[\,x_{1}^-(0),\,x_{2}^-(0),\,x^-_{1}(1)\,\bigr]_{(s^3,\,r^{-2}s^{-1})}}\bigr]_{s^{6}} \,(\hbox{using (\ref{b:1})})\\
 =&\bigl[\,x_{2}^-(0), \bigl[\,[\,x_{1}^-(0),\, x_{2}^-(0)\,]_{r^{-3}},\,x_1^-(1) \,\bigr]_{rs^2},
\,\bigr]_{s^6} \,\hbox{(using (\ref{b:1}))}\\
 &+r^{-3}\bigl[\,x_{2}^-(0),\,x_{2}^-(0),\,
 \underbrace{\bigl[\,x_{1}^-(0), \,x^-_{1}(1)\,
 \bigr]_{rs^{-1}}}\,\bigr]_{((rs)^3,\,s^6)}\,(=0 \hbox{ by (\ref{b:1}) and the Serre relation})\\
 =&\bigl[\,\underbrace{[\,x_{2}^-(0),\,x_{1}^-(0),\, x_{2}^-(0)\,]_{(r^{-3},\,s^3)}},\,x_1^-(1) \,\bigr]_{rs^5}
\,\hbox{(=0 by (\ref{b:1}))}\\
&+s^3\bigl[\,[\,x_{1}^-(0),\, x_{2}^-(0)\,]_{r^{-3}},\,[\,x_{2}^-(0),\,,\,x_1^-(1)\,]_{s^3} \,\bigr]_{rs^{-1}}
\,\hbox{(=0 by (\ref{b:1}))}\\
 =&s^3\bigl[\,[\,x_{1}^-(0),\, \underbrace{[\,x_{2}^-(0),\,x_{2}^-(0),\,,\,x_1^-(1)\,]_{(s^3,\,r^3)}} \,\bigr]_{r^{-5}s^{-1}}
\,\hbox{(=0 by (\ref{b:1}))}\\
&+(rs)^3\bigl[\,[\,x_{1}^-(0),\,x_{2}^-(0),\,,\,x_1^-(1)\,]_{(s^3,\,r^{-2}s^{-1})} ,\, x_{2}^-(0)\,\bigr]_{r^{-6}}\\
=&-r^{-3}s^3\bigl[\,x_{2}^-(0),\,[\,x_{1}^-(0),\,x_{2}^-(0),\,,\,x_1^-(1)\,]_{(s^3,\,r^{-2}s^{-1})} \,\bigr]_{r^{6}}.
\end{split}
\end{equation*}
The result yields that
$(1+r^{-3}s^3)\bigl[\,x_{2}^-(0),\, x^-_{\theta}(1)\,\bigr]_{(rs)^3}=0$, which implies that  for $r\neq -s$
$$\bigl[\,x_{2}^-(0),\, x^-_{\theta}(1)\,\bigr]_{(rs)^3}=0.$$
\end{proof}

Next we check the commutation relation given below.
\begin{prop} {\label{p:8}
$[\,E_0, F_0\,]=
\frac{\gamma'^{-1}\om_{\theta}^{-1}-\gamma^{-1}\om_{\theta'}^{-1}}{r-s}$.}
\end{prop}

\begin{proof}\,
Similarly, one proves the relation by induction. First, note that
\begin{equation*}
\begin{split}
\bigl[\,E_0,
F_0\,\bigr]&=(rs)^{2}\bigl[\,x^-_{\theta}(1)\,
\gamma'^{-1}{\om_\theta}^{-1},\,\gamma^{-1}{\om'_\theta}^{-1}
x^+_{\theta}(-1)\,\bigr]\\
&=(rs)^{2}\bigl[\,x^-_{\theta}(1),\,x^+_{\theta}(-1)\,\bigr]\,
\cdot(\gamma^{-1}\gamma'^{-1}{\om_\theta}^{-1}{\om'_\theta}^{-1}).
\end{split}
\end{equation*}

Second, consider the first step,
\begin{equation*}
\begin{split}
&[\,[\,x_2^-(0),\,x_1^-(1)\,]_{s^3},\,[\,x_1^+(-1),\,x_2^+(0)\,]_{r^3}\,]\,{\hbox{(using (\ref{b:1}))}}\\
=&\big[\,[\,x_{2}^-(0),\,[\,x_{1}^-(1),\,x_{1}^+(-1)\,]\,]_{s^3}
,\,x_{2}^+(0)\,\big]_{r^3}\,{\hbox{(=0 by (\ref{b:1}) and the Serre relation)}}\\
&+\big[\,x_{1}^+(-1),\,[\,[\,x_{2}^-(0),\,x_2^+(0)\,],\,x_1^-(1)\,]_{s^3}
\,\big]_{r^3}\quad{\hbox{(using (\ref{b:1}) and (D6))}}\\
=&\gamma\om'_{1}\cdot \frac{\om'_2-\om_2}{r-s}
+\frac{\gamma\om'_{1}-\gamma'\om_{1}}
{r-s}\om_2\\
=&\frac{\gamma\om'_{1}\om'_{2}-\gamma'\om_{1}\om_{2}}{r-s}.
\end{split}
\end{equation*}

Third, one obtains that
\begin{equation*}
\begin{split}
&\bigl[\,x^-_{\theta}(1),\,x^+_{\theta}(-1)\,\bigr]\,{\hbox{(by definition)}}\\
=&[\,[\,x_{1}^-(0),\,x_{1\,2}^-(1)\,]_{r^{-2}s^{-1}}
,\,[\,x_{1\,2}^+(-1),\,x_{1}^+(0)\,]_{r^{-1}s^{-2}}\,]
\,{\hbox{(using (\ref{b:1}))}}\\
=&[\,x_{1\,2}^+(-1),\,[\,
[\,x_{1}^-(0),\,x_{1}^+(0)\,],\,x_{1\,2}^-(1)\,]_{r^{-2}s^{-1}}
\,]_{r^{-1}s^{-2}}\,{\hbox{(using (\ref{b:1}) and (D6))}}\\
&+[\,x_{1\,2}^+(-1),\,[\,x_{1}^-(0),\,
[\,x_{1\,2}^-(1),\,x_{1}^+(0)\,]\,]_{r^{-2}s^{-1}}
\,]_{r^{-1}s^{-2}}\,{\hbox{(=0 by (D6) and (D9))}}\\
&+[\,[\,[\,x_{1}^-(0),\,x_{1\,2}^+(-1)\,],\,x_{1\,2}^-(1)\,]_{r^{-2}s^{-1}}
,\,x_{1}^+(0)\,]_{r^{-1}s^{-2}}\,{\hbox{(=0 by the above result)}}\\
&+[\,[\,x_{1}^-(0),\,[\,x_{1\,2}^-(1),\,x_{1\,2}^+(-1)\,]\,]_{r^{-2}s^{-1}}
,\,x_{1}^+(0)\,]_{r^{-1}s^{-2}}\quad{\hbox{(using (\ref{b:1}), (D9) and (D6))}}\\
=&(rs)^{-2}\om_1\cdot \frac{\gamma\om'_{1}\omega'_2-\gamma'\om_{1}\omega_2}{r-s}
+(rs)^{-2}\gamma\omega'_1\omega'_2\frac{\om'_{1}-\gamma'\om_{1}}
{r-s}\om_{n-1}\\
=&(rs)^{-2}\frac{\gamma(\om'_1)^2\om'_2-\gamma'(\om_{1})^2\om_2}{r-s}.
\end{split}
\end{equation*}

So one obtains the required conclusion and the proof
of Proposition \ref{p:8} is completed.
\end{proof}

In the remaining part of this section, we check some key Serre relations of ($\mathcal{D}5$).

\begin{lemm} One gets that

$(1)$ \ $[\,E_0, \,E_2\,]_{(rs)^{-3}}=0$,

$(2)$ \ $[\,E_1,\,E_1,\, E_0\,]_{(r^{-2}s^{-1},\,r^{-1}s^{-2})}=0$,

$(3)$ \ $[\,E_{1},\,E_{1},\, E_{1},\,E_1,\, E_2\,]_{(s^{3},\,rs^{2},\,r^2s,\,r^{3})}=0$.

\end{lemm}

\begin{proof} \ (1) \  For the first relation, one immediately gets from the construction of $E_0$,
\begin{equation*}
\begin{split}
&\bigl[\,E_0,\,E_2\bigr]_{(rs)^{-2}}\\
=&(rs)^{-2}\bigl[\,x_{\theta}^-(1),\,
x_2^+(0)\,\bigr]\cdot {\gamma'}^{-1}\omega_{\theta}^{-1} \\
=&a\bigl[\,[\,x_1^-(0),\, x_{2}^-(0),\,x_{1}^-(1)\,]_{(s^3,\, r^{-2}s^{-1})},\,x_2^+(0)\bigr]\cdot {\gamma'}^{-1}\omega_{\theta}^{-1}\\
=&a\bigl[\,x_1^-(0),\, [\,x_{2}^-(0),\,x_2^+(0)\,],\,x_{1}^-(1)\,\bigr]_{(s^3,\, r^{-2}s^{-1})}\cdot {\gamma'}^{-1}\omega_{\theta}^{-1}\\
=&a\underbrace{\bigl[\,x_1^-(0),\, x_{1}^-(1)\,\bigr]_{rs^{-1}}}\cdot\om_2 {\gamma'}^{-1}\omega_{\theta}^{-1}\\
=&0
\end{split}
\end{equation*}

(2) \  To verify the second relation, one first computes that
\begin{equation*}
\begin{split}
&\bigl[\,x_1^+(0),\,x_{\theta}^-(1)\bigr]_{1}\\
=&\bigl[\,[\,x_1^+(0),\, x_1^-(0)\,],\,x_2^-(0),\,x_{1}^-(1)\,\bigr]_{(s^3,\, r^{-2}s^{-1})} \\
+&\bigl[\,x_1^-(0),\, x_{2}^-(0),\,[\,x_1^+(0),\, x_1^-(1)\,]\,\bigr]_{(s^3,\, r^{-2}s^{-1})}\\
=&(rs)^{-2}\om_1[\,x_2^-(0),\,x_1^-(1)\,]_{s^3}
-3(\gamma\gamma')^{-\frac{1}{2}}[\,x_1^-(0),\,x_2^-(1)\,]_{r^{-3}}\omega_1\\
=&(rs)^{-2}\om_1[\,x_2^-(0),\,x_1^-(1)\,]_{s^3}.
\end{split}
\end{equation*}
Then one has that
\begin{equation*}
\begin{split}
&[\,E_1,\,E_1,\, E_0\,]_{(r^{-2}s^{-1},\,r^{-1}s^{-2})}\\
=&a\,\bigl[\,x_1^+(0),\,x_1^+(0),\,x_{\theta}^-(1)\,\bigr]_{(1,\, r^{-1}s)} \\
=&a\,\bigl[\,x_1^+(0),\,\om_1[\,x_2^-(0),\,x_{1}^-(1)\,]_{s^3}\,\bigr]_{r^{-1}s}\\ =&a\,\om_1\bigl[\,x_1^+(0),\,[\,x_2^-(0),\,x_{1}^-(1)\,]_{s^3}\,\bigr]_{1}\\
=&a\,\om_1\bigl[\,x_2^-(0),\,[\,x_1^+(0),\,x_{1}^-(1)\,]\,\bigr]_{s^3}\\
=&-a\,(\gamma\gamma')^{-\frac{1}{2}}(\om_1)^2x_2^-(1)=0
\end{split}
\end{equation*}

(3) \  Let $X'=[\,E_{1},\,E_{1},\, E_{1},\,E_1,\, E_2\,]_{(s^{3},\,rs^{2},\,r^2s,\,r^{3})}$.
It is obvious that $X'\in \mathcal{U}_{r,s}(\frak{g}^{\sigma})^+$, so by Lemma \ref{lemma1}, it suffices to show $[\,x_i^-(0),\, X'\,]=0$ for $i\in I$. In fact, it is enough to check the cases of $i=1$ and $2$.

For $i=2$, one has
\begin{equation*}
\begin{split}
&[\,x_{2}^-(0),\, X'\,] \\
=&[\,x_2^-(0),\,[\,x_{1}^+(0),\,x_{1}^+(0),\,x_{1}^+(0),\, x_{1}^+(0),\, x_{2}^+(0)\,]_{(s^{3},\,rs^{2},\,r^2s,\,r^{3})}\,]\\
=&[\,x_{1}^+(0),\,x_{1}^+(0),\,x_{1}^+(0),\, x_{1}^+(0),\, [\,x_2^-(0),\,x_{2}^+(0)\,]\,]_{(s^{3},\,rs^{2},\,r^2s,\,r^{3})}\\
=&3[\,x_{1}^+(0),\,x_{1}^+(0),\,x_{1}^+(0),\, x_{1}^+(0)\,]_{(r^{-2}s^{2},\,r^{-1}s,\,1)}\om_2\\
=&0.
\end{split}
\end{equation*}

As for $i=1$, one computes that
\begin{equation*}
\begin{split}
&[\,x_{1}^-(0),\, X'\,] \\
=&[\,x_{1}^-(0),\,[\,x_{1}^+(0),\,x_{1}^+(0),\, x_{1}^+(0),\, x_{1}^+(0),\,x_{2}^+(0)\,]_{(s^{3},\,rs^{2},\,r^2s,\,r^{3})}\,]\\
=&[\,[\,x_{1}^-(0),\,x_{1}^+(0)\,],\,x_{1}^+(0),\, x_{1}^+(0),\, x_{1}^+(0),\,x_{2}^+(0)\,]_{(s^{3},\,rs^{2},\,r^2s,\,r^{3})}\\
&+[\,x_{1}^+(0),\,[\,x_{1}^-(0),\,x_{1}^+(0)\,],\, x_{1}^+(0),\, x_{1}^+(0),\,x_{2}^+(0)\,]_{(s^{3},\,rs^{2},\,r^2s,\,r^{3})}\\
&+[\,x_{1}^+(0),\,x_{1}^+(0),\, [\,x_{1}^-(0),\,x_{1}^+(0)\,],\, x_{1}^+(0),\,x_{2}^+(0)\,]_{(s^{3},\,rs^{2},\,r^2s,\,r^{3})}\\
&+[\,x_{1}^+(0),\,x_{1}^+(0),\, x_{1}^+(0),\, [\,x_{1}^-(0),\,x_{1}^+(0)\,],\,x_{2}^+(0)\,]_{(s^{3},\,rs^{2},\,r^2s,\,r^{3})}
\end{split}
\end{equation*}

\begin{equation*}
\begin{split}
=&\frac{s^3-r^3}{r-s}[\,x_{1}^+(0),\,x_{1}^+(0),\, x_{1}^+(0),\,x_{2}^+(0)\,]_{(s^{3},\,rs^{2},\,r^2s)}\om'_1\\
&-rs[\,x_{1}^+(0),\,x_{1}^+(0),\, x_{1}^+(0),\,x_{2}^+(0)\,]_{(s^{3},\,rs^{2},\,r^2s)}\om'_1\\
&+rs[\,x_{1}^+(0),\,x_{1}^+(0),\, x_{1}^+(0),\,x_{2}^+(0)\,]_{(s^{3},\,rs^{2},\,r^2s)}\om'_1\\
&+\frac{r^3-s^3}{r-s}[\,x_{1}^+(0),\,x_{1}^+(0),\, x_{1}^+(0),\,x_{2}^+(0)\,]_{(s^{3},\,rs^{2},\,r^2s)}\om'_1\\
=&0.
\end{split}
\end{equation*}
Therefore, we have finished the proof of theorem $\mathcal{A}$ for the case of $U_{r,s}(\mathrm{D}_{4}^{(3)})$.
\end{proof}

\subsection{Proof of theorem $\mathcal{A}$ for the case of $U_{r,s}(\mathrm{E}_{6}^{(2)})$}

Now we are left to check the last case of type $\mathrm{E}_{6}^{(2)}$. As before we only check those nontrivial relations of $(\mathcal{E}4)$--$(\mathcal{E}6)$
involving $i=0$.

We begin by listing the following simple lemmas.

\begin{lemm} One has that
\begin{align} \label{l:20}
&[\,x_{1}^-(0),\,z_{6}^-(1)\,]_{(rs)^{-1}}=0,\\ \label{l:21}
&[\,x_{2}^-(0),\,z_{7}^-(1)\,]_{r^{-1}s^{-2}}=0, \\ \label{l:22}
&[\,x_{3}^-(0),\,z_{7}^-(1)\,]_{(rs)^{2}}=0,\\ \label{l:23}
& [\,x_{4}^-(0),\,z_{6}^-(1)\,]_{(rs)^{2}}=0.
\end{align}
\end{lemm}

The following proposition yields relation $(\mathcal{E}4)$ for the case of $i\neq 0$ and $j=0$.
\begin{prop} \label{p:10} $\bigl[\,x_i^-{(0)},\,
 x^-_{\theta}(1)\,\bigr]_{\langle i,~
0 \rangle^{-1}}=0$,\quad for \ $i\in \{1,\,2,\,3,\,4\}$.
\end{prop}

\begin{proof} \ (I) \ When $i=1$,
$\langle 1,\, 0\rangle=rs^2$. We compute that,
\begin{equation*}
\begin{split}
&\bigl[\,x_1^-(0),\, x^-_{\theta}(1)\,\bigr]_{r^{-1}s^{-2}}\\
=&\bigl[\,x_1^-(0),\underbrace{\bigl[\,x_1^-(0),\,x_2^-(0),\,z^-_{6}(1)\,\bigr]_{(r^{-2}s^{-1},\,r^{-2}s^{-1})}}
\bigr]_{r^{-1}s^{-2}}\,( \hbox{using (\ref{b:1})})\\
=&\bigl[\,x_1^-(0),\bigl[\,[\,x_1^-(0),\,x_2^-(0)\,]_{r^{-1}},\,z^-_{6}(1)\,\bigr]_{r^{-3}s^{-2}}\bigr]_{r^{-1}s^{-2}}
\,( \hbox{using (\ref{b:1})})\\
&+r^{-1}\bigl[\,x_1^-(0),\,\bigl[\,x_2^-(0),\,\underbrace{[\,x_1^-(0),\,z^-_{6}(1)\,]_{(rs)^{-1}}}\,\bigr]_{(rs)^{-1}}
\bigr]_{r^{-1}s^{-2}}
\,( \hbox{=0 by \eqref{l:20}})\\
=&\bigl[\,\bigl[\,x_1^-(0),\,x_1^-(0),\,x_2^-(0)\,\bigr]_{(r^{-1},\,s^{-1})},\,z^-_{6}(1)\,\bigr]_{r^{-4}s^{-3}}
\,( \hbox{=0 by Serre relation})\\
&+s^{-1}\bigl[\,[\,x_1^-(0),\,x_2^-(0)\,]_{r^{-1}},\,
\underbrace{[\,x_1^-(0),\,z^-_{6}(1)\,]_{(rs)^{-1}}}\,\bigr]_{r^{-4}s^{-3}}\,( \hbox{=0 by \eqref{l:20}})\\
=&0.
\end{split}
\end{equation*}

(II) \ When $i=2$, $\langle 2,\, 0\rangle=rs$. In order to check $\bigl[\,x_2^-(0),\, x^-_{\theta}(1)\,\bigr]_{(rs)^{-1}}=0$, we have that
\begin{equation*}
\begin{split}
&\bigl[\,x_2^-(0),\, x^-_{\theta}(1)\,\bigr]_{r^{-2}}\\
=&\bigl[\,x_1^-(0),\underbrace{\bigl[\,x_1^-(0),\,x_2^-(0),\,z^-_{6}(1)\,\bigr]_{(r^{-2}s^{-1},\,r^{-2}s^{-1})}}
\bigr]_{r^{-2}}\,( \hbox{using (\ref{b:1})})\\
=&\bigl[\,x_2^-(0),\bigl[\,[\,x_1^-(0),\,x_2^-(0)\,]_{r^{-1}},\,z^-_{6}(1)\,\bigr]_{r^{-3}s^{-2}}\bigr]_{r^{-2}}
\,( \hbox{using (\ref{b:1})})\\
&+r^{-1}\bigl[\,x_2^-(0),\,\bigl[\,x_2^-(0),\,\underbrace{[\,x_1^-(0),\,z^-_{6}(1)\,]_{(rs)^{-1}}}\,\bigr]_{(rs)^{-1}}
\bigr]_{r^{-2}}
\,( \hbox{=0 by \eqref{l:20}})\\
=&\bigl[\,\underbrace{\bigl[\,x_2^-(0),\,x_1^-(0),\,x_2^-(0)\,\bigr]_{(r^{-1},\,s)}},\,z^-_{6}(1)\,\bigr]_{r^{-5}s^{-3}}
\,( \hbox{=0 by the Serre relation})\\
&+s\bigl[\,[\,x_1^-(0),\,x_2^-(0)\,]_{r^{-1}},\,
\underbrace{[\,x_2^-(0),\,z^-_{6}(1)\,]_{r^{-2}s^{-1}}}\,\bigr]_{(rs)^{-3}}\,( \hbox{by definition and (\ref{b:1})})\\
=&s\bigl[\,x_1^-(0),\,\underbrace{[\,x_2^-(0),\,z^-_{7}(1)\,]_{r^{-1}s^{-2}}}\,\bigr]_{r^{-3}s^{-1}}\,(  \hbox{=0 by \eqref{l:21}})\\
&+(rs)^{-1}\bigl[\,[\,x_1^-(0),\,z^-_{7}(1)\,]_{r^{-2}s^{-1}},\,x_2^-(0)\,\bigr]_{s^{2}}\,( \hbox{by definition})\\
=&(rs)^{-1}\bigl[\,x^-_{\theta}(1),\,x_2^-(0)\,\bigr]_{s^{2}}
\end{split}
\end{equation*}

which implies that $(1+r^{-1}s)\bigl[\,x_2^-(0),\, x^-_{\theta}(1)\,\bigr]_{(rs)^{-1}}=0$.
Consequently, for $r\neq -s$ $$\bigl[\,x_2^-(0),\, x^-_{\theta}(1)\,\bigr]_{(rs)^{-1}}=0.$$

(III) \ When $i=3$, $\langle 3,\,0\rangle=(rs)^{-2}$. Observe that,
\begin{equation*}
\begin{split}
&\bigl[\,x_3^-(0),\, x^-_{\theta}(1)\,\bigr]_{(rs)^{2}}\\
=&\bigl[\,x_3^-(0),\,[\,x_1^-(0),\,z^-_{7}(1)\,]_{r^{-2}s^{-1}}
\bigr]_{(rs)^{2}}\,( \hbox{using (\ref{b:1})})\\
=&\bigl[\,\underbrace{[\,x_3^-(0),\,x_1^-(0)\,]},\,z^-_{6}(1)\,\bigr]_{s}
\,( \hbox{=0 by the Serre relation})\\
&+\bigl[\,x_1^-(0),\,\underbrace{[\,x_3^-(0),\,z^-_{7}(1)\,]_{(rs)^{2}}}\,\bigr]_{r^{-2}s^{-1}}
\,( \hbox{=0 by \eqref{l:22}})\\
=&0.
\end{split}
\end{equation*}

(IV) \ When $i=4$, $\langle 4,\, 0\rangle=(rs)^{-2}$. Using the above results, one has
\begin{equation*}
\begin{split}
&\bigl[\,x_4^-(0),\, x^-_{\theta}(1)\,\bigr]_{(rs)^{2}}\\
=&\bigl[\,x_4^-(0),\,[\,x_1^-(0),\,z^-_{7}(1)\,]_{r^{-2}s^{-1}}
\bigr]_{(rs)^{2}}\,( \hbox{using (\ref{b:1})})\\
=&\bigl[\,\underbrace{[\,x_4^-(0),\,x_1^-(0)\,]},\,z^-_{7}(1)\,\bigr]_{s}
\,( \hbox{=0 by the Serre relation})\\
&+\bigl[\,x_1^-(0),\,\underbrace{[\,x_4^-(0),\,z^-_{7}(1)\,]_{(rs)^{2}}}\,\bigr]_{r^{-2}s^{-1}}
\,( \hbox{using the definition and (\ref{b:1})})\\
=&\bigl[\,x_1^-(0),\,[\,\underbrace{[\,x_4^-(0),\,x_2^-(0)\,]},\,z^-_{6}(1)\,]_{s}\,\bigr]_{r^{-2}s^{-1}}
\,( \hbox{=0 by the Serre relation})\\
&+\bigl[\,x_1^-(0),\,[\,x_2^-(0),\,\underbrace{[\,x_4^-(0),\,z^-_{6}(1)\,]_{(rs)^{2}}}
\,]_{r^{-2}s^{-1}}\,\bigr]_{r^{-2}s^{-1}}
\,( \hbox{=0 by \eqref{l:23}})\\
=&0.
\end{split}
\end{equation*}
\end{proof}

Relation $(\mathcal{E}4)$ for the case of $i,j=0$ follows from Proposition \ref{p:11}.

\begin{prop} {\label{p:11}
$[\,E_0, F_0\,]=
\frac{\gamma'^{-1}\om_{\theta}^{-1}-\gamma^{-1}\om_{\theta'}^{-1}}{r-s}$.}
\end{prop}

\begin{proof}\,
We argue by induction. First, using
relation $(\textrm{D6})$ it follows that,
\begin{equation*}
\begin{split}
\bigl[\,E_0,
F_0\,\bigr]&=(rs)^{5}[2]_{2}^{-1}\bigl[\,x^-_{\theta}(1)\,
\gamma'^{-1}{\om_\theta}^{-1},\,\gamma^{-1}{\om'_\theta}^{-1}
x^+_{\theta}(-1)\,\bigr]\\
&=(rs)^{5}[2]_{2}^{-1}\bigl[\,x^-_{\theta}(1),\,x^+_{\theta}(-1)\,\bigr]\,
\cdot(\gamma^{-1}\gamma'^{-1}{\om_\theta}^{-1}{\om'_\theta}^{-1}).
\end{split}
\end{equation*}
Now recall the notations mentioned above.
\begin{equation*}
\begin{split}
&x^-_{\theta}(1)=z_{8}^-(1)\\
=&[\,x_1^-(0), x_2^-(0),x_3^-(0), x_2^-(0), x_{4}^-(0), \ldots, x_{1}^-(1)\,]
_{(s, s^2, s^{2}, r^{-1}, s^{2}, r^{-2}s^{-1}, r^{-2}s^{-1})},\\
&x_{\theta}^+(-1)=z_{8}^+(-1)\\
=&[\,x_1^+(-1),\ldots, x_4^+(0), x_{2}^+(0), x_3^+(0), x_2^+(0),\,x_{1}^+(0)\,]
_{\langle r, r^2, r^{2}, s^{-1}, r^{2}, r^{-1}s^{-2}, r^{-1}s^{-2}\rangle}.
\end{split}
\end{equation*}

Consider the first step,
\begin{equation*}
\begin{split}
&[\,z_2^-(1),\,z_2^+(-1)\,]\\
=&[\,[\,x_1^-(0),\,x_2^-(1)\,]_{r^{-1}},\,[\,x_2^+(-1),\,x_1^+(0)\,]_{s^{-1}}\,]\quad{\hbox{(using (\ref{b:1}))}}\\
=&\big[\,[\,x_{1}^-(0),\,[\,x_{2}^-(1),\,x_{2}^+(-1)\,]\,]_{r^{-1}}
,\,x_{2}^+(0)\,\big]_{s^{-1}}\quad{\hbox{(using (D6) and (D9))}}\\
&+\big[\,x_{2}^+(-1),\,[\,[\,x_{1}^-(0),\,x_1^+(0)\,],\,x_2^-(1)\,]_{r^{-1}}
\,\big]_{s^{-1}}\quad{\hbox{(using (D6) and (D9))}}\\
=&(rs)^{-1}\gamma\om'_{2}\cdot \frac{\om'_1-\om_1}{r-s}
+(rs)^{-1}\frac{\gamma\om'_{2}-\gamma'\om_{2}}
{r-s}\om_1\\
=&(rs)^{-1}\frac{\gamma\om'_{\eta_2}-\gamma'\om_{\eta_2}}{r-s}.
\end{split}
\end{equation*}

For simplicity, we list the results for the intermediate steps.
\begin{equation*}
[\,z_i^-(1),\,z_i^+(-1)\,]=(rs)^{-1}\frac{\gamma\om'_{\eta_i}-\gamma'\om_{\eta_i}}{r-s}, \quad i=3, 4, 5, 6, 7.
\end{equation*}

One has from the last induction step
\begin{equation*}
\begin{split}
&\bigl[\,x^-_{\theta}(1),\,x^+_{\theta}(-1)\,\bigr]\qquad{\hbox{(by definition)}}\\
=&[\,[\,x_{1}^-(0),\,z_{7}^-(1)\,]_{r^{-2}s^{-1}}
,\,[\,z_{7}^+(-1),\,x_{1}^+(0)\,]_{r^{-1}s^{-2}}\,]
\qquad{\hbox{(using (\ref{b:1}))}}\\
=&[\,z_{7}^+(-1),\,[\,
[\,x_{1}^-(0),\,x_{1}^+(0)\,],\,z_{7}^-(1)\,]_{r^{-2}s^{-1}}
\,]_{r^{-1}s^{-2}}\,{\hbox{(using (\ref{b:1}),(D9) and (D6))}}\\
&+[\,z_{7}^+(-1),\,[\,x_{1}^-(0),\,
[\,z_{7}^-(1),\,x_{1}^+(0)\,]\,]_{r^{-2}s^{-1}}
\,]_{r^{-1}s^{-2}}\,{\hbox{(=0 by (D9) and (D6))}}\\
&+[\,[\,[\,x_{1}^-(0),\,z_{7}^+(-1)\,],\,z_{7}^-(1)\,]_{r^{-2}s^{-1}}
,\,x_{1}^+(0)\,]_{r^{-1}s^{-2}}\,{\hbox{(=0 by the above result and (D5))}}\\
&+[\,[\,x_{1}^-(0),\,[\,z_{7}^-(1),\,z_{7}^+(-1)\,]\,]_{r^{-2}s^{-1}}
,\,x_{1}^+(0)\,]_{r^{-1}s^{-2}}\,{\hbox{(using (D6) and (D9))}}\\
=&(rs)^{-5}[2]_2\gamma\om'_{\eta_7}\frac{\gamma\om'_{1}-\gamma'\om_{1}}{r-s}
+(rs)^{-5}[2]_2\frac{\gamma\om'_{\eta_7}-\gamma'\om_{\eta_7}}
{r-s}\om_{1}\\
=&(rs)^{-5}[2]_2\frac{\gamma\om'_{\theta}-\gamma'\om_{\theta}}{r-s}.
\end{split}
\end{equation*}

So we have obtained the required conclusion.
\end{proof}

In the last part of this section, we check the remaining Serre relations of ($\mathcal{E}5$).

\begin{lemm} For the case of $U_{r,s}(\mathrm{E}_{6}^{(2)})$, one has,

 $(1)$ $[\,E_1,\,E_1,\, E_0\,]_{(r^{-2}s^{-1},\,r^{-1}s^{-2})}=0$,

$(2)$ \ $[\,E_{2},\,E_{2},\, E_{2},\, E_3\,]_{(s^{2},\,rs,\,r^2)}=0$.
\end{lemm}

\begin{proof} \ (1) \ In order to verify the first relation, one considers that,
\begin{equation*}
\begin{split}
&\bigl[\,x_1^+(0),\,x_{\theta}^-(1)\bigr]_{1}\\
=&\bigl[\,[\,x_1^+(0),\, x_1^-(0)\,],\,z_7^-(1)\,\bigr]_{r^{-2}s^{-1}} \\
+&[\,x_1^-(0),x_2^-(0), x_3^-(0), x_{4}^-(0), x_3^-(0), x_2^-(0),\\
&\hspace{2.5cm}[\,x_1^+(0), x_1^-(0)\,],\, x_{2}^-(1)\,]
_{(r^{-1}, r, s^{4}, s^{2}, s^{2}, r^{-2}s^{-1}, r^{-2}s^{-1})}\\
=&(rs)^{-2}z_7^-(1)\omega_1+(rs)^{-1}[\,x_1^-(0),x_2^-(0), x_3^-(0), x_{4}^-(0), x_3^-(0), \\
&\hspace{2.5cm}\underbrace{[\,x_2^-(0),x_{2}^-(1)\,]_{rs^{-1}}}\omega_1\,]
_{(s^{4}, s^{2}, s^{2}, r^{-2}s^{-1}, r^{-2}s^{-1})}\\
=&(rs)^{-2}z_7^-(1)\omega_1.
\end{split}
\end{equation*}
Applying the above result, one gets by direct calculation that
\begin{equation*}
\begin{split}
&[\,E_1,\,E_1,\, E_0\,]_{(r^{-2}s^{-1},\,r^{-1}s^{-2})}\\
=&a\bigl[\,x_1^+(0),\,x_1^+(0),\,x_{\theta}^-(1)\,\bigr]_{(1,\, r^{-1}s)} \\
=&a(rs)^{-2}\bigl[\,x_1^+(0),\,z_7^-(1)\bigr]\omega_1\\
=&a(rs)^{-2}[\,x_2^-(0), x_3^-(0), x_{4}^-(0), x_3^-(0), x_2^-(0),\\
&\hspace{2.5cm}[\,x_1^+(0), x_1^-(0)\,],\, x_{2}^-(1)\,]
_{(r^{-1}, r, s^{4}, s^{2}, s^{2}, r^{-2}s^{-1})}\omega_1\\
=&a(rs)^{-3}[\,x_2^-(0), x_3^-(0), x_{4}^-(0), x_3^-(0),\underbrace{[x_2^-(0), x_{2}^-(1)]_{rs^{-1}}}\omega_1\,]
_{(s^{4}, s^{2}, s^{2}, r^{-2}s^{-1})}\omega_1\\
=&0
\end{split}
\end{equation*}

(3) \  Using Lemma \ref{lemma1}, one can show the third relation. Let $Y'=[\,E_{2},\,E_{2},\, E_{2},\, E_3\,]_{(s^{2},\,rs,\,r^2)}$.
To prove $[\,x_i^-(0),\, Y'\,]=0$ for $i\in I$, it is enough to check the cases of $i$ being $2$ and $3$.

For $i=3$, one immediately has,
\begin{equation*}
\begin{split}
&[\,x_{3}^-(0),\, Y\,] \\
=&[\,x_3^-(0),\,[\,x_{2}^+(0),\,x_{2}^+(0),\,x_{2}^+(0),\, x_{3}^+(0)\,]_{(s^{2},\,rs,\,r^2)}\,]\\
=&[\,x_{2}^+(0),\,x_{2}^+(0),\,x_{2}^+(0),\, [\,x_3^-(0),\,x_{3}^+(0)\,]\,]_{(s^{2},\,rs,\,r^2)}\\
=&-\omega_3[\,x_{2}^+(0),\,x_{2}^+(0),\,x_{2}^+(0)\,]_{(r^{-1}s,\,1)}\\
=&0.
\end{split}
\end{equation*}

For $i=2$, one gets from the following direct calculation:
\begin{equation*}
\begin{split}
&[\,x_{2}^-(0),\, Y\,] \\
=&[\,x_{2}^-(0),\,[\,x_{2}^+(0),\,x_{2}^+(0),\, x_{2}^+(0),\, x_{3}^+(0)\,]_{(s^{2},\,rs,\,r^2)}\,]\\
=&[\,[\,x_{2}^-(0),\,x_{2}^+(0)\,],\,x_{2}^+(0),\, x_{2}^+(0),\,x_{3}^+(0)\,]_{(s^{2},\,rs,\,r^2)}\\
&+[\,x_{2}^+(0),\,[\,x_{2}^-(0),\,x_{2}^+(0)\,],\, x_{2}^+(0),\, x_{2}^+(0)\,]_{(s^{2},\,rs,\,r^2)}\\
&+[\,x_{2}^+(0),\,x_{2}^+(0),\, [\,x_{1}^-(0),\,x_{1}^+(0)\,],\, x_{2}^+(0)\,]_{(s^{2},\,rs,\,r^2)}\\
=&-(r+s)[\,x_{2}^+(0),\,x_{2}^+(0),\, x_{3}^+(0)\,]_{(s^{2},\,rs)}\om'_2\\
&+(r+s)[\,x_{2}^+(0),\,x_{2}^+(0),\, x_{3}^+(0)\,]_{(s^{2},\,rs)}\om'_2\\
=&0.
\end{split}
\end{equation*}
\end{proof}
So far, we have proved Theorem $\mathcal{A}$ for all twisted cases, that is, there exists an algebra homomorphism $\Psi$ between the two realizations of the two-parameter twisted quantum affine algebras.
In the next section, we will give a new proof of the other two theorems.

\section{The inverse map $\Phi$ of $\Psi$}

 The goal of this section is to obtain the inverse map of $\Psi$, more precisely, there exists an algebra isomorphism from Drinfeld realization to Drinfeld-Jimbo form of the two-parameter twisted quantum affine algebras. We remark
 that the proof works not only for  untwisted cases but also for twisted cases. In particular, we have another proof of Drinfeld isomorphism for the quantum affine algebras for all the cases, which were proved using the braid group \cite{B, JZ4}. From now on, we denote by $\hat{\frak{g}}$ any affine Lie algebra.

Fix $k\in I$, denote by $\mathcal{U}^k_{r,s}(\hat{\frak{g}})$ the subalgebra of
$\mathcal{U}_{r,s}(\hat{\frak{g}})$ generated by
$x_i^{\pm}(0),\,x_k^+(-1)$, $x_k^-(1)$, $\om_i^{\pm1},\,\om_i'^{\pm1}$ ($i\in I$), and
$\gamma^{\pm\frac{1}2},\, \gamma'^{\,\pm\frac{1}2}$, 
satisfying the following relations $(7.1)-(7.6)$, that is,
$$
{\mathcal U}^k_{r,s}(\hat{\frak{g}}):=\left.\left\langle\, x_i^{\pm}(0), x_k^+(-1),\,x_k^-(1), \,\om_i^{\pm1},\, {\om}_i'^{\,\pm1},
\gamma^{\pm\frac{1}2}, \gamma'^{\,\pm\frac{1}2}, 
\; \right| i\in I\;\right\rangle/\sim.
$$

\noindent $(\textrm{7.1})$ \  $\gamma^{\pm\frac{1}{2}}$,
$\gamma'^{\,\pm\frac{1}{2}}$ are central such that
$\gamma\gamma'=(rs)^c$,\,
$\omega_i\,\omega_i^{-1}=\omega_i'\,\omega_i'^{\,-1}=1$ $(i\in I)$,
and for $i,\,j\in I$, one has
\begin{equation*}
\begin{split}
[\,\omega_i^{\pm 1},\omega_j^{\,\pm 1}\,]
=[\,\omega_i^{\pm 1},\omega_j'^{\,\pm 1}\,]
=[\,\omega_i'^{\pm1},\omega_j'^{\,\pm 1}\,]=0.
\end{split}
\end{equation*}
$$x_{\sigma(i)}^{\pm}(l)=\omega^lx_i^{\pm}(l),\, a_{\sigma(i)}(m)=\omega^ma_i(m)\leqno(\textrm{7.2})$$
$$
\om_i\,x_j^{\pm}(k)\, \om_i^{-1} =\sum_{t=0}^{k-1}  \langle \sigma^t(j),
i\rangle^{\pm 1} x_j^{\pm}(k),\leqno(\textrm{7.3})$$
$$ \om'_i\,x_j^{\pm}(k)\,
\om_i'^{\,-1} = \sum_{t=0}^{k-1} \langle i, \sigma^t(j)\rangle
^{\mp1}x_j^{\pm}(k).
$$
$$
F_{ij}^\pm(z,\,w)\,x_i^{\pm}(z)x_j^{\pm}(w)=G_{ij}^\pm(z,\,w)\,x_j^{\pm}(w)\,x_i^{\pm}(z),\leqno{(\textrm{7.4})}
$$

$$
[\,x_i^{+}(k),~x_j^-(k')\,]=\frac{\delta_{ij}}{r_i-s_i}\Big(\gamma'^{-k}\,{\gamma}^{-\frac{k+k'}{2}}\,
\phi_i(k{+}k')-\gamma^{k'}\,\gamma'^{\frac{k+k'}{2}}\,\varphi_i(k{+}k')\Big),\leqno(\textrm{7.5})
$$
where $\phi_i(m)$, $\varphi_i(-m)~(m\in \mathbb{Z}_{\geq 0})$ such that
$\phi_i(0)=\om_i$ and  $\varphi_i(0)=\om_i'$ are defined as below:
\begin{gather*}\sum\limits_{m=0}^{\infty}\phi_i(m) z^{-m}=\om_i \exp \Big(
(r_i{-}s_i)\sum\limits_{\ell=1}^{\infty}
 a_i(\ell)z^{-\ell}\Big),\quad \bigl(\phi_i(-m)=0, \ \forall\;m>0\bigr); \\
\sum\limits_{m=0}^{\infty}\varphi_i(-m) z^{m}=\om'_i \exp
\Big({-}(r_i{-}s_i)
\sum\limits_{\ell=1}^{\infty}a_i(-\ell)z^{\ell}\Big), \quad
\bigl(\varphi_i(m)=0, \ \forall\;m>0\bigr).
\end{gather*}
$$x_i^{\pm}(m)x_j^{\pm}(k)=\langle j,i\rangle^{\pm1}x_j^{\pm}(k)x_i^{\pm}(m),
\qquad\ \hbox{for} \quad a^{\sigma}_{ij}=0,\leqno(\textrm{$7.6_a$})$$
$$
\begin{array}{lll}
& Sym_{z_1,z_2,z_{3}}\Big\{((rs^{-2})^{\mp \frac{k}{4}}z_1-(r^{\frac{k}{4}}+s^{\frac{k}{4}})z_2+
(r^{-2}s)^{\mp \frac{k}{4}}z_3)x_i^{\pm}(z_1)x_i^{\pm}(z_2)x_i^{\pm}(z_{3})\Big\}=0,\\
&  \hskip1.8cm \quad\hbox{for} \quad A_{i,\sigma(i)}=-1, \\
\end{array} \leqno{(\textrm{$7.6_b$})}
$$
$$
\begin{array}{lll}
&Sym_{z_1, z_{2}}\Big\{P_{ij}^{\pm}(z_1,z_2)\sum_{t=0}^{t=2}(-1)^t(r_is_i)^{\pm\frac{t(t-1)}{2}}
\Big[{2\atop  t}\Big]_{\pm{i}}x_i^{\pm}(z_1)\cdots x_i^{\pm}(z_t)\\
&\hskip0.8cm \times x_j^{\pm}(w)x_i^{\pm}(z_{t+1})\cdots x_i^{\pm}(z_{2})\Big\}=0,\\
&\hskip1.8cm \quad\hbox{for} \quad A_{i,j}=-1, \quad  \hbox{and}\quad 1\leqslant j<i\leqslant N \quad\hbox{such\quad that}\quad \sigma(i)\neq j,\\
\end{array} \leqno{(\textrm{$7.6_c$})}
$$
$$
\begin{array}{lll}
&Sym_{z_1, z_{2}}\Big\{P_{ij}^{\pm}(z_1,z_2)\sum_{t=0}^{t=2}(-1)^t(r_is_i)^{\mp\frac{t(t-1)}{2}}
\Big[{2\atop  t}\Big]_{\mp{i}}x_i^{\pm}(z_1)\cdots x_i^{\pm}(z_t)\\
&\hskip0.8cm \times x_j^{\pm}(w)x_i^{\pm}(z_{t+1})\cdots x_i^{\pm}(z_{2})\Big\}=0,\\
&\hskip1.8cm\quad\hbox{for} \quad A_{i,j}=-1, \quad  \hbox{and}\quad 1\leqslant i<j\leqslant N \quad\hbox{such\quad that}\quad \sigma(i)\neq j.
\end{array} \leqno{(\textrm{$7.6_d$})}
$$
where $[l ]_{\pm{i}}= \frac{r_i^{\pm l}-s_i^{\pm
l}}{r_i^{\pm1}-s_i^{\pm1}}$,\, $[l]_{\mp{i}}= \frac{r_i^{\mp
l}-s_i^{\mp l}}{r_i^{\mp1}-s_i^{\mp1}}$, $\textit{Sym}_{z_1,\ldots, z_n}$  denotes symmetrization w.r.t. the
indices $(z_1, \ldots, z_n)$, and
\begin{eqnarray*}
&&\hbox{If}\,\, \sigma(i)=i,\, \hbox{then}\,\,  P_{ij}^{\pm}(z,w)=1,\,  d_{ij}=k,\\
&&\hbox{If}\,\, A_{i,\sigma(i)}=0, \, \sigma(j)=j, \, \hbox{then} P_{ij}^{\pm}(z,w)=
\frac{z^r(rs^{-1})^{\pm k}-w^r}{z(rs^{-1})^{\pm 1}-w},\,  d_{ij}=k, \\
&&\hbox{If}\,\,  A_{i,\sigma(i)}=0, \, \sigma(j)\neq
j,\, \hbox{then}\,\, P_{ij}^{\pm}(z,w)=1,\, d_{ij}=1/2, \\
&&\hbox{If}\,\,   A_{i,\sigma(i)}=-1,\, \hbox{then}\,\, P_{ij}^{\pm}(z,w)=
z(rs^{-1})^{\pm k/4}+w,\,  d_{ij}=k/2.
\end{eqnarray*}

In fact, we have the following result.

\begin{prop}\, ${\mathcal U}^k_{r,s}(\hat{\frak{g}})={\mathcal U}_{r,s}(\hat{\frak{g}})$
\end{prop}
\begin{proof}\, It suffices to show all other generators of ${\mathcal U}_{r,s}(\hat{\frak{g}})$ are in the algebra ${\mathcal U}^k_{r,s}(\hat{\frak{g}})$.

First, we denote
\begin{gather*}
a_k(1)=\om_i^{-1}\gamma^{1/2}\,\bigl[\,x_k^+(0),\,x_k^-(1)\,\bigr]
\in {\mathcal U}^k_{r,s}(\hat{\frak{g}}),\\
a_k(-1)=\om_i'^{-1}{\gamma}'^{1/2}\,\bigl[\,x_k^+(-1),\,x_k^-(0)\,
\bigr]\in
{\mathcal U}^k_{r,s}(\hat{\frak{g}}).
\end{gather*}
It is not difficult to show that $a_k(1)$ and $a_k(-1)$ satisfy the defining relations $(\textrm{$D1$})-(\textrm{$D10$})$ (those involving
$a_k(1)$ and $a_k(-1)$)

Moreover, we denote that for $j\in I$ such that $a_{jk}\neq 0$,

\begin{gather*}
x_j^+(-1)=-(r_ks_k)^{-\frac{a_{kj}}{2}}[-a_{kj}]_k^{-1}{\gamma'}^{-\frac{1}{2}}\,\bigl[\,a_k(-1),\,x_j^+(0)\,\bigr]
\in {\mathcal U}^k_{r,s}(\hat{\frak{g}}),\\
x_j^-(1)=-(r_ks_k)^{\frac{a_{kj}}{2}}[a_{kj}]_k^{-1}{\gamma}^{-\frac{1}{2}}\,
\bigl[\,a_k(1),\,x_j^-(0)\,\bigr]\in
{\mathcal U}^k_{r,s}(\hat{\frak{g}}).
\end{gather*}
It follows from direct computation  that $x_j^+(-1)$ and $x_j^-(1)$ satisfy the defining relations $(\textrm{$D1$})-(\textrm{$D10$})$ (those involving with $x_j^+(-1)$ and $x_j^-(1)$ solely)

Repeating the above two steps, one obtains that for all $i\in I$,  $a_i(1)$, $a_i(-1)$, $x_i^-(1)$ and $x_i^+(-1)$
are  in ${\mathcal U}^k_{r,s}(\hat{\frak{g}})$. For more details, see \cite{JZ5}.

At the same time, $x_i^-(-1)$ and $x_i^+(1)$ are both in ${\mathcal U}^k_{r,s}(\hat{\frak{g}})$.
Therefore, all degree one generators are in the subalgebra, which satisfy defining relations $(\textrm{$D1$})-(\textrm{$D10$})$.

Next for $\ell\in \mathbb{Z}/ \{0\}$, suppose that $x_i^\pm(\ell)\in {\mathcal U}^k_{r,s}(\hat{\frak{g}})$
and $a_i(\ell)\in {\mathcal U}^k_{r,s}(\hat{\frak{g}})$.
One has,
\begin{eqnarray*}
{\mathcal U}^k_{r,s}(\hat{\frak{g}})&\ni&[\,x_i^+(\ell), x_i^-(1)\,]\\
&=&{\ast}\, a_i(\ell+1)+\sum\limits_{1<t<\ell+1, \sum l_k=\ell+1}{\ast'}\, a_i(\ell_{j_1})\cdots a_i(\ell_{j_t}),
\end{eqnarray*}
where scalars $\ast,\,\ast'\in \mathbb{K}/\{0\}$, hence $a_i(\ell+1)\in {\mathcal U}^k_{r,s}(\hat{\frak{g}})$.

It follows that,
\begin{eqnarray*}
{\mathcal U}^k_{r,s}(\hat{\frak{g}})\ni[\,a_i(\ell), x_i^{\pm}(1)\,]
=\star\, x_i^{\pm}(\ell+1),
\end{eqnarray*}
where scalars $\star\in \mathbb{K}/\{0\}$, therefore $x_i^{\pm}(\ell+1)\in {\mathcal U}^k_{r,s}(\hat{\frak{g}})$.
So all generators of ${\mathcal U}_{r,s}(\hat{\frak{g}})$ are in ${\mathcal U}^k_{r,s}(\hat{\frak{g}})$by induction,
 which satisfy defining relations $(\textrm{$D1$})-(\textrm{$D10$})$.
\end{proof}

\begin{remark}\, As a consequence of the above proposition, there exist $n$ subalgebras ${\mathcal U}^k_{r,s}(\hat{\frak{g}})$ which are all isomorphic to ${\mathcal U}_{r,s}(\hat{\frak{g}})$, or rather
there are $n$ sets of generators ${\mathcal U}_{r,s}(\hat{\frak{g}})$ of degree bounded by $\pm 1$.
From now on, fix $k\in I$, we use the presentation (subalgebra) ${\mathcal U}^k_{r,s}(\hat{\frak{g}})$ instead of ${\mathcal U}_{r,s}(\hat{\frak{g}})$. We remark that the result has been generalized to the quantum toroidal algebra of type A \cite{JZ5} to give
its Hopf algebra structure with a finite comultiplication.
\end{remark}

We keep the previous assumptions and notations, and let $i_1,\,\ldots,\, i_{h-1}$ be a sequence of indices of the
fixed reduced expression given in Eq. (\ref{a1}).  We also need a few more notations for our purpose.
\begin{align*}
p_{i_j}&=\langle i_j,\, i_{j-1}\ldots i_2 i_{1}\rangle,\quad p'_{i_j}=\langle i_j,\, i_{j}\ldots i_2 i_{1}\rangle,\quad p''_{i_j}=\langle i_j,\, i_{j+1}\ldots i_2 i_{1}\rangle,\\
q_{i_j}&=\langle i_1i_2\ldots i_{j-1} ,\, i_j\rangle^{-1},\quad q'_{i_j}=\langle i_1i_2\ldots i_{j} ,\, i_j\rangle^{-1},\quad q''_{i_j}=\langle i_1i_2\ldots i_{j+1} ,\, i_j\rangle^{-1}.
\end{align*}
Denote $t'_{i_j}=\frac{q'_{i_j}-p'_{i_j}}{r_{i_j}-s_{i_j}}$.


We now define the inverse map of $\Psi$, more precisely, we have the following statement.

{\bf Theorem $\mathcal{B}$.} {\it Fix $k\in I$, let $k=i_1,\,i_2,\, \ldots, i_{h-1}$ be the fixed root chain associated to the maximum root
$\theta$ in \eqref{a1}. Then $\Phi:\,{\mathcal U}^k_{r,s}(\hat{\frak{g}})\longrightarrow U_{r,s}(\hat{\frak{g}})$ is an epimorphism such that  $\forall i \in I$
\begin{align*}
&\Phi(\gamma)=\gamma,\qquad \Phi(\gamma')=\gamma',\qquad \Phi(x_i^+(0))=e_i,\\
&\Phi(x_i^-(0))=f_i,\qquad \Phi(\omega_i)=\omega_i,\qquad \Phi(\omega'_i)=\omega'_i,\\
&\Phi(x_k^-(1))={t'}_{i_{h-1}}^{-1}\cdots {t'}_{i_2}^{-1}[\,e_{i_2},\,e_{i_3},\,\ldots,\, e_{i_{h-1}},\, e_0\,]_{(p'_{i_{h-1}},\, \ldots,\, p'_{i_2})}\gamma'\omega_{k},\\
&\Phi(x_k^+(-1))=\gamma\omega'_{k}[\,f_0,\,f_{i_{h-1}},\,f_{i_{h-2}},\,\ldots,\, f_{i_{2}}\,]_{\langle q'_{i_{h-1}},\, \ldots,\, q'_{i_2}\rangle}
\end{align*}}
\begin{proof}\, We show that $\Phi$ satisfies relations $(D1)-(D9)$. From our previous discussion
it is enough to check the relations only involving $x_k^+(-1)$ and $x_k^-(1)$. In particular, we verify relation $(D8)$.

For $1<\ell<h$, let us denote
\begin{eqnarray*}
\tilde{e}_{i_\ell}(1)&&=[\,e_{i_\ell},\,e_{i_{\ell+1}},\,\ldots,\, e_{i_{h-1}},\, e_0\,]_{(p'_{i_{h-1}},\, \ldots,\, p'_{i_\ell})}.\\
\tilde{f}_{i_l}(-1)&&=[\,f_0,\,f_{i_{h-1}},\,f_{i_{h-2}},\,\ldots,\, f_{i_{\ell}}\,]_{\langle q'_{i_{h-1}},\, \ldots,\, q'_{i_\ell}\rangle}.
\end{eqnarray*}
Then $\Phi(x_k^-(1))={t'}_{i_{h-1}}^{-1}\cdots {t'}_{i_2}^{-1}\tilde{e}_{i_2}(1)\gamma'\omega_{k}$,
and $\Phi(x_k^+(-1))={t'}_{i_{h-1}}^{-1}\cdots {t'}_{i_2}^{-1}\gamma\omega'_{k}\tilde{f}_{i_2}(-1)$.

We compute that
\begin{eqnarray*}
[\,\tilde{f}_{i_{h-1}}(-1), \tilde{e}_{i_{h-1}}(1)\,]
&=&\Big[\,\big[\,f_0,\,f_{i_{h-1}}\,\big]_{q'_{i_{h-1}}},\,
\big[\,e_{i_{h-1}},\,e_0\,\big]_{p'_{i_{h-1}}}\,\Big]\\
&=&\Big[\,\big[\,[\,f_0,\,e_{i_{h-1}}\,],\,f_{i_{h-1}}\,\big]_{q'_{i_{h-1}}},\,
e_0\Big]_{p'_{i_{h-1}}}\\
&&+\Big[\,\big[\,f_0,\,[\,f_{i_{h-1}},\,e_{i_{h-1}}\,]\,\big]_{q'_{i_{h-1}}},\,
e_0\,\Big]_{p'_{i_{h-1}}}\\
&&+\Big[\,e_{i_{h-1}},\,\big[\,[\,f_0,\,e_0\,]
,\,f_{i_{h-1}}\,\big]_{q'_{i_{h-1}}}\,\Big]_{p'_{i_{h-1}}}\\
&&+\Big[\,e_{i_{h-1}},\,\big[\,f_0,\,[\,f_{i_{h-1}},\,e_0\,]
\,\big]_{q'_{i_{h-1}}}\,\Big]_{p'_{i_{h-1}}}\\
&=&-t'_{i_{h-1}}\omega'_{i_{h-1}}\frac{\omega'_0-\omega_0}{r-s}
+t'_{i_{h-1}}\frac{\omega_{i_{h-1}-\omega'_{i_{h-1}}}}{r-s}\omega_0\\
&=&t'_{i_{h-1}}\frac{\omega_{i_{h-1}}\omega_0-\omega'_{i_{h-1}}\omega'_0}{r-s},
\end{eqnarray*}
where the first summand and the last summand are null due to the commutation relation.

Now we consider,
\begin{eqnarray*}
[\,\tilde{f}_{i_l}(-1), \tilde{e}_{i_l}(1)\,]
&=&\Big[\,\big[\,\tilde{f}_{i_{l-1}}(-1),\,f_{i_l}\,\big]_{q'_{i_l}},\,
\big[\,e_{i_l},\,\tilde{e}_{i_{l-1}}(1)\,\big]_{p'_{i_l}}\,\Big]\\
&=&\Big[\,\big[\,[\,\tilde{f}_{i_{l-1}}(-1),\,e_{i_l}\,],\,f_{i_l}\,\big]_{q'_{i_l}},\,
\tilde{e}_{i_{l-1}}(1)\Big]_{p'_{i_l}}\\
&&+\Big[\,\big[\,\tilde{f}_{i_{l-1}}(-1),\,[\,f_{i_l},\,e_{i_l}\,]\,\big]_{q'_{i_l}},\,
\tilde{e}_{i_{l-1}}(1)\,\Big]_{p'_{i_l}}\\
&&+\Big[\,e_{i_l},\,\big[\,[\,\tilde{f}_{i_{l-1}}(-1),\,\tilde{e}_{i_{l-1}}(1)\,]
,\,f_{i_l}\,\big]_{q'_{i_l}}\,\Big]_{p'_{i_l}}\\
&&+\Big[\,e_{i_l},\,\big[\,\tilde{f}_{i_{l-1}}(-1),\,[\,f_{i_l},\,\tilde{e}_{i_{l-1}}(1)\,]
\,\big]_{q'_{i_l}}\,\Big]_{p'_{i_l}}.
\end{eqnarray*}
Iteratively we get that
\begin{eqnarray*}
[\,\tilde{f}_{i_l}(-1), \tilde{e}_{i_l}(1)\,]
&=&t'_{i_{h-1}}\cdots t'_{i_{l}} \frac{\omega_{i_l}\cdots\omega_{i_{h-1}}\omega_0-\omega'_{i_l}\cdots\omega'_{i_{h-1}}\omega'_0}{r-s}.
\end{eqnarray*}

As a consequence of the above results, we get immediately,
\begin{eqnarray*}
&&\Phi([\,x_k^+(-1),\,x_k^-(1)\,])\\
&=&-{t'}_{i_{h-1}}^{-1}\cdots {t'}_{i_2}^{-1}\gamma\gamma'\omega_{k}\omega'_{k}[\,\tilde{f}_{i_2}(-1), \tilde{e}_{i_2}(1)\,]\\
&=&\Phi(\frac{\gamma'\omega_k-\gamma\omega'_k}{r-s})
\end{eqnarray*}
where we have used $\omega_0={\gamma'}^{-1}\omega_{\theta}^{-1}$ and $\omega'_0={\gamma}^{-1}{\omega'}_{\theta}^{-1}$.
\end{proof}

In fact, the map $\Psi$ is the inverse of $\Phi$.

{\bf Theorem $\mathcal{C}$.} {\it \qquad $\Psi=\Phi^{-1}$.}

\begin{proof}\, It suffices to check the action of $\Psi\Phi$ on the generators are trivial. Most of these are trivial except the generators $x_k^-(1)$ and $x_k^+(-1)$, and we directly compute that
 $\Psi\Phi(x_k^-(1))=x_k^-(1)$ and $\Psi\Phi(x_k^+(-1))=x_k^+(-1)$.

In fact, for $2\leq j \leq h-1$, set $$\tilde{y}_{1,1}^-(1)=x_k^-(1){\gamma'}^{-1}\omega_{k}^{-1},$$ and $$\tilde{y}_{1,j}^-(1)=[\,x_{i_j}^-(0),\, \ldots,x_{i_2}^-(0),\, x_{i_1}^-(1)\,]_{{(p_{i_{2}},\, \ldots,\, p_{i_{j}})}}{\gamma'}^{-1}\omega_{i_{1}}^{-1}\cdots \omega_{i_{j}}^{-1}.$$
Using these new notation, we can write that $$\tilde{y}_{1,{h-1}}^-(1)=\Psi(e_0).$$

 Firstly, note that for $2\leq j \leq h-1$
\begin{eqnarray*}
[\,x_{i_{j}}^+(0),\, \tilde{y}_{1,j}^-(1)\,]_{p'_{i_{j}}}&=&[\,x_{i_{j}}^+(0),\, y_{1,j}^-(1)\,]{\gamma'}^{-1}\omega_{i_{1}}^{-1}\cdots \omega_{i_{j}}^{-1}\\
&=&t_{i_j}\tilde{y}_{1,j-1}^-(1)
\end{eqnarray*}
It holds by direct computation,
\begin{eqnarray*}
&&\Psi\Phi(x_k^-(1))\\
&=&\Psi\big(t_{i_{h-1}}^{-1}\cdots t_{i_2}^{-1}[\,e_{i_2},\,e_{i_3},\,\ldots,\, e_{i_{h-1}},\, e_0\,]_{(p'_{i_{h-1}},\, \ldots,\, p'_{i_2})}\gamma'\omega_{k}\big)\\
&=&t_{i_{h-1}}^{-1}\cdots t_{i_2}^{-1}[\,x_{i_2}^+(0),\,x_{i_3}^+(0),\,\ldots,\, x_{i_{h-1}}^+(0),\, \Psi(e_0)\,]_{(p'_{i_{h-1}},\, \ldots,\, p'_{i_2})}\gamma'\omega_{k}\\
&=&t_{i_{h-2}}^{-1}\cdots t_{i_2}^{-1}[\,x_{i_2}^+(0),\,x_{i_3}^+(0),\,\ldots,\, x_{i_{h-2}}^+(0),\, \tilde{y}_{1,h-2}^-(1)\,]_{(p'_{i_{h-2}},\, \ldots,\, p'_{i_2})}\gamma'\omega_{k}\\
&=&\tilde{y}_{1,1}^-(1)\gamma'\omega_{k}\\
&=&x_k^-(1)
\end{eqnarray*}
The action of $\Psi\Phi$ on $x_k^+(-1)$ can be checked similarly. Consequently, $\Psi$ and $\Phi$ are invertible to each other.
\end{proof}

Therefore we have proved there exists an algebra isomorphism between the two realizations of two-parameter twisted quantum affine algebras. In particular, we also prove the Drinfeld isomorphism for quantum affine algebras as a special case.

\section{The Hopf algebra structure of $\mathcal{U}^k_{r,s}(\hat{\frak{g}})$}

In the last section, we will discuss the Hopf structure of the subalgebra $\mathcal{U}^k_{r,s}(\hat{\frak{g}})$ of Drinfeld realization. It
is well-known that Drinfeld-Jimbo realization ${U}_{r,s}(\hat{\frak{g}})$ admits a Hopf algebra structure. The aim of the present section is to establish
the Hopf algebra structure of Drinfeld realization $\mathcal{U}^k_{r,s}(\hat{\frak{g}})$. Furthermore, we show that there exists
  a Hopf isomorphism between these two realizations. This
generalizes the corresponding result for one-parameter quantum affine algebras.

We adopt the same notations and assumptions from the previous sections. In particular, let $k\in I$,
let $k=i_1,\,i_2,\, \ldots, i_{h-1}$ be a fixed sequence of indices satisfying \eqref{a1}. We also need a few more notations.

For $i_2\leq j_1<\ldots<j_l\leq i_{h-1}$, define $x_{1,\,l}^-(1)$ and $x_{1,\,l}^+(-1)$ inductively by
\begin{eqnarray*}
&&x_{1\,l}^-(1)=[\,x_{j_l}^-(0),\,\ldots x_{j_1}^-(0),\,x_{1}^-(1)\,]_{(p_{j_{1}},\, \ldots,\, p_{j_l})},\\
&&x_{1,\,l}^+(-1)=[\,x_{1}^+(-1),\, x_{j_1}^+(0),\,\ldots ,\,x_{j_{l}}^+(0)\,]_{\langle q_{j_{1}},\, \ldots,\, q_{j_l}\rangle}.
\end{eqnarray*}

Denote $x_{i_2\,l}^+(0)$ and $x_{i_2,\,l}^-(0)$ for $2\leq l\leq h-1$ as follows.
\begin{eqnarray*}
&&x_{j_1\,l}^+(0)=[\,x_{j_1}^+(0),\, x_{j_2}^+(0),\,\ldots ,\,x_{j_{l}}^+(0)\,]_{(u_{j_{l-1}},\, \ldots,\, u_{j_1})},\\
&&x_{j_1,\,l}^-(0)=[\,x_{j_l}^-(0),\, x_{j_{l-1}}^-(0),\,\ldots ,\,x_{j_{1}}^-(0)\,]_{\langle v_{j_{l-1}},\, \ldots,\, v_{j_1}\rangle},
\end{eqnarray*}
where  $u_{i_j}=p'_{i_j}q''_{i_j}$ and $v_{i_j}=q'_{i_j}p''_{i_j}$. 

Now we can define the actions of a comultiplication on the simple generators of the algebra $\mathcal{U}^k_{r,s}(\hat{\frak{g}})$.

\begin{defi} \,Fix $k\in I$, let $k=i_1,\,i_2,\, \ldots, i_{h-1}$ be a sequence of indices given by Eq. (\ref{a1}). Define the action of the comultiplication $\Delta'$ on the generators of the algebra $\mathcal{U}^k_{r,s}(\hat{\frak{g}})$ as follows. For
 $i\in I$,
\begin{gather*}
\Delta'(\omega_i)=\omega_i\ot \omega_i, \qquad \Delta'(\omega_i')=\omega_i'\ot \omega_i',\\
\Delta'(\gamma^{\pm\frac{1}2})=\gamma^{\pm\frac{1}2}\otimes
\gamma^{\pm\frac{1}2}, \qquad
\Delta'(\gamma'^{\,\pm\frac{1}2})=\gamma'^{\,\pm\frac{1}2}\otimes
\gamma'^{\,\pm\frac{1}2}, \\
\Delta'(D^{\pm1})=D^{\pm1}\otimes D^{\pm1},\qquad
\Delta'(D'^{\,\pm1})=D'^{\,\pm1}\otimes D'^{\,\pm1},\\
\Delta'(x_i^+(0))=x_i^+(0)\ot 1+\omega_i\ot x_i^+(0), \qquad \Delta'(x_i^-(0))=x_i^-(0)\ot w_i'+1\ot
x_i^-(0),\\
\Delta'(x_k^-(1))=x_k^-(1)\ot \gamma'\omega_k+1\ot x_k^-(1)+\sum_{i_2\leq j_1<\cdots<j_l\leq i_{h-1}}\xi_l\, x_{1\,l}^-(1)\ot x_{j_1\,l}^+(0)\gamma'\omega_k\\
\Delta'(x_k^+(-1))=x_k^+(-1)\ot 1+\gamma\omega'_k\ot x_k^+(-1)+\sum_{i_2\leq j_1<\cdots<j_l\leq i_{h-1}}\zeta_l\, \gamma\omega'_k\, x_{j_1\,l}^-(0)\ot x_{1\,l}^+(-1)
\end{gather*}
where $$\xi_l=(q'_{i_l}-p'_{i_l})t_{i_2}^{-1}\ldots t_{i_l}^{-1}q''_{i_2}\ldots q''_{i_{l-1}},$$ $$\zeta_l=(p'_{i_l}-q'_{i_l})t_{i_2}^{-1}\cdots t_{i_l}^{-1}p''_{i_2}\cdots p''_{i_{l-1}}.$$

\end{defi}
\smallskip

The above comultiplication $\Delta'$ is well-defined, which will be verified by the following proposition. We also note that the formulas for one-parameter cases were essentially given in \cite[Th. 2.2]{JM}.

\begin{prop}{\label{8.1} \, Fix $k\in I$, let $k=i_1,\,i_2,\, \ldots, i_{h-1}$ be a sequence of indices given in Eq. (\ref{a1}). The algebra $\mathcal{U}^k_{r,s}(\hat{\frak{g}})$  is a Hopf algebra with
the above comultiplication $\Delta'$, the counit $\vep$ and the antipode $S$
defined below, for $i\in I$, we have
\begin{eqnarray*}
&&\varepsilon(x_i^+(0))=\varepsilon(x_i^-(0))=\varepsilon(x_k^+(-1))=\varepsilon(x_k^-(1))=0,\\
&& \varepsilon(\gamma^{\pm\frac{1}2})
=\varepsilon(\gamma'^{\,\pm\frac{1}2})=\varepsilon(D^{\pm1})
=\varepsilon(D'^{\,\pm1})=\varepsilon(\omega_i)=\varepsilon(\omega_i')=1,\\
&&S(\gamma^{\pm\frac{1}2})=\gamma^{\mp\frac{1}2},\qquad
S(\gamma'^{\pm\frac{1}2})=\gamma'^{\mp\frac{1}2},\qquad
S(D^{\pm1})=D^{\mp1},\\
&&S(\omega_i)=\omega_i^{-1}, \qquad S(\omega_i')=\omega_i'^{-1}\,\qquad S(D'^{\,\pm1})=D'^{\,\mp1},\\
&&S(x_i^+(0))=-\omega_i^{-1}x_i^+(0),\qquad S(x_i^-(0))=-x_i^-(0)\,\omega_i'^{-1},\\
&&S(x_k^+(-1))=-\gamma^{-1}{\omega'_k}^{-1}x_k^+(-1)\\
&&\hskip3cm-\sum_{i_2\leq j_1<\cdots<j_l\leq i_{h-1}}\zeta_l\,\, y_{j_1\,l}^-(0)\gamma^{-1}{\omega'_{j_1}}^{-1}\cdots {\omega'_{j_l}}^{-1}\,x_{1,\,l}^+(-1),\\
&&S(x_k^-(1))=-x_k^-(1){\gamma'}^{-1}{\omega_k}^{-1}-\\
&&\hskip3cm-\sum_{i_2\leq j_1<\cdots<j_l\leq i_{h-1}}\xi_l\,\, x_{1\,l}^-(1){\gamma'}^{-1}{\omega_{j_1}}^{-1}\cdots {\omega_{j_l}}^{-1}\, y_{j_1,\,l}^+(0),
\end{eqnarray*}
where $$y_{j_1\,l}^-(0)=a\,[\,x_{j_l}^-(0),\, x_{j_{l-1}}^-(0),\,\ldots ,\,x_{j_{1}}^-(0)\,]_{\langle v'_{j_{l-1}},\, \ldots,\, v'_{j_1}\rangle}$$
and $$y_{j_1,\,l}^+(0)=b\,[\,x_{j_1}^+(0),\, x_{j_2}^+(0),\,\ldots ,\,x_{j_{l}}^+(0)\,]_{(u'_{j_{l-1}},\, \ldots,\, u'_{j_1})},$$ and for $2\leqslant j \leqslant l-1$, $v'_{i_{j}}=\langle i_j,\, i_{j+1}\rangle$, $u'_{i_j}=\langle i_{j+1},\,i_j\rangle$, here the constant $a=\prod\limits_{j=2}^{l-1}v_{i_j}\cdot \langle i_{j+1}\cdots i_{l},\, i_j\rangle$ and
$b=\prod\limits_{j=2}^{l-1}u_{i_j}\cdot \langle i_j,\, i_{j+1}\cdots i_{l}\rangle$.}
\end{prop}

\begin{proof}\,{\bf (a)}\, We first show that $\Delta'$ defines a morphism of algebra from $\mathcal{U}^k_{r,s}(\hat{\frak{g}})$  into $\mathcal{U}^k_{r,s}(\hat{\frak{g}})\otimes\mathcal{U}^k_{r,s}(\hat{\frak{g}})$. Note that the actions on all generators are the same as that of Drinfeld-Jimbo generators except
$x_k^-(1)$ and $x_k^+(-1)$. So it is enough to show the relations involving $x_k^-(1)$ and $x_k^+(-1)$. We first check for $k'\neq k$
\begin{eqnarray*}
[\,\Delta'(x_{k'}^+(0)),\,\Delta'(x_k^-(1))\,]=0.
\end{eqnarray*}
By direct computation, one gets that
\begin{eqnarray*}
&&[\,\Delta'(x_{k'}^+(0)),\,\Delta'(x_k^-(1))\,]\\
&=&[\,x_{k'}^+(0)\otimes 1+K_{k'}\otimes x_{k'}^+(0) ,\,\Delta'(x_k^-(1))\,]\\
&=&\sum_{i_2\leq j_1<\cdots<j_l\leq i_{h-1}}\xi_l\, [x_{k'}^+(0),\,x_{1\,l}^-(1)]\ot x_{j_1\,l}^+(0)\gamma'\omega_k
+[K_{k'}\otimes x_{k'}^+(0),\,x_k^-(1)\otimes \gamma'\omega_k]\\
&&+\sum_{i_2\leq j_1<\cdots<j_l\leq i_{h-1}}\xi_l\, [K_{k'}\otimes x_{k'}^+(0),\,x_{1\,l}^-(1)\ot x_{j_1\,l}^+(0)]
\end{eqnarray*}
The first term will be killed by the last two terms for its two cases.

Then we are left to verify the following relations
\begin{eqnarray*}
[\,\Delta'(x_k^+(-1)),\,\Delta'(x_k^-(1))\,]=\frac{\Delta'(\gamma'\omega_k)-\Delta'(\gamma\omega'_k)}{r_k-s_k}
\end{eqnarray*}
By definition, we have immediately,

\begin{eqnarray*}
&&[\,\Delta'(x_k^+(-1)),\,\Delta'(x_k^-(1))\,]\\
&=&\big[\,(x_k^+(-1)\ot 1+\gamma\omega'_k\ot x_k^+(-1)+\sum_{i_2\leq {j'}_1<\cdots<{j'}_{l'}\leq i_{h-1}}\zeta_l'\, \gamma\omega'_k\, x_{j_1\,l'}^-(0)\ot x_{1\,l'}^+(-1)),\, \\
&&\hskip1.05cm(x_k^-(1)\ot \gamma'\omega_k+1\ot x_k^-(1)+\sum_{i_2\leq j_1<\cdots<j_l\leq i_{h-1}}\xi_l\, x_{1\,l}^-(1)\ot x_{j_1\,l}^+(0)\gamma'\omega_k)\,\big]
\end{eqnarray*}

By direct calculations, we pull out the common factors, then the above bracket can be divided four summands, that is,

\begin{eqnarray*}
&&[\,\Delta'(x_k^+(-1)),\,\Delta'(x_k^-(1))\,]\\
&=&[\,x_k^+(-1),\,x_k^-(1)\,]\ot \gamma\omega'_k+\gamma\omega'_k\ot[\,x_k^+(-1),\,x_k^-(1)\,]\\
&&+\sum_{i_2\leq {j'}_1<\cdots<{j'}_{l'}\leq i_{h-1}}\sum_{i_2\leq j_1<\cdots<j_l\leq i_{h-1}}\Big\{\zeta_{l'}\,\gamma\omega'_k\,[\,x_{j_1\,l'}^-(0),\,x_k^-(1)\,]_{\langle j_1\cdots j_{l'},\,k\rangle}\\
&&\hskip0.65cm\ot x_{1\,l'}^+(-1)\gamma'\omega_k+\xi_l \gamma\omega'_k x_{1\,l}^-(1)\ot[\,x_k^+(-1),\,x_{j_1\,l}^+(0)\,]_{\langle k,\,j_1\cdots j_l\rangle^{-1}}\gamma'\omega_k\Big\}\\
&&+\sum_{i_2\leq {j'}_1<\cdots<{j'}_{l'}\leq i_{h-1}}\sum_{i_2\leq j_1<\cdots<j_l\leq i_{h-1}}\Big\{\zeta_{l'}\,\gamma\omega'_k\,x_{j_1\,l'}^-(0)\ot [\,x_{1\,l'}^+(-1),\,x_k^-(1)\,]\\
&&\hskip1.65cm+\xi_l [\,x_k^+(-1),\,x_{1\,l}^-(1)\,]\ot x_{j_1\,l}^+(0)\gamma'\omega_k\Big\}\\
&&+\sum_{i_2\leq {j'}_1<\cdots<{j'}_{l'}\leq i_{h-1}}\sum_{i_2\leq j_1<\cdots<j_l\leq i_{h-1}}\zeta_{l'}\xi_l\, \Big\{\gamma\omega'_k\,x_{j_1\,l'}^-(0)x_{1\,l}^-(1)\\
&&\hskip0.5cm\ot x_{1\,l'}^+(-1)x_{j_1\,l}^+(0)\gamma'\omega_k-x_{1\,l}^-(1)\gamma\omega'_k x_{j_1\,l'}^-(0)\ot x_{j_1\,l}^+(0)\gamma'\omega_k x_{1\,l'}^+(-1)\Big\}
\end{eqnarray*}

In fact, the last three summands are 0. Using Drinfeld relation $(D9)$ and   $(D5)$, the second summand and the third summand also vanish.
The last summand is 0 for Serre relations.  For simplicity, we proceed with the example of the case of $D_4^{(3)}$.

In this case, the last summand becomes:
\begin{eqnarray*}
 &&\zeta_1\xi_1\, \Big\{\gamma\omega'_1\,x_{2}^-(0)[\,x_2^-(0),\,x_{1}^-(1)\,]_{s^3}\ot [\, x_{1}^+(-1),\,x_2^+(0)\,]_{r^3}x_2^+(0)\gamma'\omega_1\\
&&\hskip0.85cm-[\,x_2^-(0),\,x_{1}^-(1)\,]_{s^3}\gamma\omega'_1x_{2}^-(0)\ot x_2^+(0)\gamma'\omega_1[\, x_{1}^+(-1),\,x_2^+(0)\,]_{r^3}\Big\}\\
&=&\zeta_1\xi_1\, \Big\{r^3\gamma\omega'_1\,[\,x_2^-(0),\,x_{1}^-(1)\,]_{s^3}x_{2}^-(0)\ot s^3 x_2^+(0)[\, x_{1}^+(-1),\,x_2^+(0)\,]_{r^3}\gamma'\omega_1\\
&&\hskip0.85cm-(rs)^3\gamma\omega'_1[\,x_2^-(0),\,x_{1}^-(1)\,]_{s^3}x_{2}^-(0)\ot x_2^+(0)[\, x_{1}^+(-1),\,x_2^+(0)\,]_{r^3}\gamma'\omega_1\Big\}\\
&=&0.
\end{eqnarray*}

Therefore, we get the required relation,
\begin{eqnarray*}
&&[\,\Delta'(x_k^+(-1)),\,\Delta'(x_k^-(1))\,]\\
&=&[\,x_k^+(-1),\,x_k^-(1)\,]\ot \gamma\omega'_k+\gamma\omega'_k\ot[\,x_k^+(-1),\,x_k^-(1)\,]\\
&=&\frac{\Delta'(\gamma'\omega_k)-\Delta'(\gamma\omega'_k)}{r_k-s_k}
\end{eqnarray*}

{\bf (b)}\, Next, we need to show that $\Delta'$ is coassociative. Similarly it suffices to check the actions of  $\Delta'$ on the generators $x_k^-(1)$ and $x_k^+(-1)$.

For the case of $x_k^-(1)$, we obtain by definition,
\begin{eqnarray*}
&&(\Delta'\otimes id)\Delta'(x_k^-(1))\\
&=&(\Delta'\otimes id)(x_k^-(1)\ot \gamma'\omega_k+1\ot x_k^-(1)+\sum_{i_2\leq j_1<\cdots<j_l\leq i_{h-1}}\xi_l\, x_{1\,l}^-(1)\ot x_{j_1\,l}^+(0)\gamma'\omega_k)\\
&=&\Delta'(x_k^-(1))\ot \gamma'\omega_k+1\ot 1\ot x_k^-(1)+\sum_{i_2\leq j_1<\cdots<j_l\leq i_{h-1}}\xi_l\,\Delta'(x_{1\,l}^-(1))\ot
x_{j_1\,l}^+(0)\gamma'\omega_k.
\end{eqnarray*}
On the other hand, we have,

\begin{eqnarray*}
&&(id\otimes\Delta')\Delta'(x_k^-(1))\\
&=&(id\otimes\Delta')(x_k^-(1)\ot \gamma'\omega_k+1\ot x_k^-(1)+\sum_{i_2\leq j_1<\cdots<j_l\leq i_{h-1}}\xi_l\, x_{1\,l}^-(1)\ot x_{j_1\,l}^+(0)\gamma'\omega_k)\\
&=&x_k^-(1)\ot \gamma'\omega_k\ot \gamma'\omega_k+1\ot\Delta'(x_k^-(1))+\sum_{i_2\leq j_1<\cdots<j_l\leq i_{h-1}}\xi_l\,x_{1\,l}^-(1)\ot
\Delta'(x_{j_1\,l}^+(0)\gamma'\omega_k)
\end{eqnarray*}

By direct calculation of  $\Delta'(x_{1\,l}^-(1))$ and $\Delta'(x_{j_1\,l}^+(0))$, the two expressions are same. Hence we get the required relation $$(\Delta'\otimes id)\Delta'(x_k^-(1))=(id\otimes\Delta')\Delta'(x_k^-(1)).$$

The proof for $x_k^+(-1)$ is analogous.

{\bf (c)}\, It is easy to check that $\varepsilon$ defines a morphism of algebra from $\mathcal{U}^k_{r,s}(\hat{\frak{g}})$ onto $\mathbb{K}$ and satisfies the counit axiom.

{\bf (d)}\, It remains to verify that $S$ defines an antipode for $\mathcal{U}^k_{r,s}(\hat{\frak{g}})$. First we have to show that $S$ is a morphism of algebra from $\mathcal{U}^k_{r,s}(\hat{\frak{g}})$ into
$\mathcal{U}^{k\,op}_{r,s}(\hat{\frak{g}})$, that is,

\begin{eqnarray*}
[\,S(x_k^-(1)),\,S(x_k^+(-1))\,]
=\frac{S(\gamma'\omega_k)-S(\gamma\omega'_k)}{r_k-s_k}
\end{eqnarray*}
The verification is similar to the above and is left to the reader.

To conclude that $S$ is an antipode, it suffices to check that the relations
$$\sum\limits_{(x)}x'S(x'')=\sum\limits_{(x)}S(x')x''=\varepsilon(x)1$$
holds when $x$ is any of the generators. Similarly, we have only to check it on $x_k^-(1)$ and  $x_k^+(-1)$.
But this follows from the construction easily.

\end{proof}

Then we arrive at our second main theorem as follows.

\begin{theo}\, The morphisms $\Phi$ and $\Psi$ are two coalgebra homomorphisms,  that is,
 $$\Delta' \circ \Psi=(\Psi\otimes \Psi)\Delta, \qquad\Delta \circ \Phi=(\Phi\otimes \Phi)\Delta'.$$

In particular, the maps $\Phi$ and $\Psi$ between the algebra $\mathcal{U}^k_{r,s}(\hat{\frak{g}})$ and $U_{r,s}(\hat{\frak{g}})$ are two Hopf algebra isomorphisms.
\end{theo}
\begin{proof}\, It follows from Proposition \ref{8.1} together with the constructions of $\Delta$ and $\Delta'$.
\end{proof}

\begin{remark}\, The result generalizes that of \cite{Dr}.
\end{remark}

\vskip30pt \centerline{\bf ACKNOWLEDGMENTS}

N. Jing would like to thank the partial support of
Simons Foundation grant 198129, NSFC grant 11271138 and NSF grants
1014554 and 1137837. H. Zhang would
like to thank the support of NSFC grants 11371238 and 11101258.
\bigskip

\bibliographystyle{amsalpha}

\end{document}